\documentclass[11pt,centertags]{amsart}
\usepackage{enumitem}
\usepackage{hyperref}

\allowdisplaybreaks

\theoremstyle{definition}
	\newtheorem{definition}{Definition}[section]
	\numberwithin{definition}{section}
\theoremstyle{plain}
	\newtheorem{lemma}[definition]{Lemma}
	\newtheorem{proposition}[definition]{Proposition}
	\newtheorem{theorem}[definition]{Theorem}
	\newtheorem{corollary}[definition]{Corollary}
	\newtheorem*{claim*}{Claim}
\theoremstyle{remark}
	\newtheorem{remark}[definition]{Remark}

\renewcommand{\labelenumi}{(\roman{enumi})}

\DeclareMathOperator{\Aut}{Aut}
\DeclareMathOperator{\dop}{d}
\DeclareMathOperator{\dvol}{dvol}
\DeclareMathOperator{\volume}{Vol}
\DeclareMathOperator{\xop}{\times}
\DeclareMathOperator{\delbar}{\overline{\partial}}
\DeclareMathOperator{\del}{\partial}

\newcommand{\bs}{\boldsymbol}
\renewcommand{\rm}{\textrm}
\newcommand{\al}{\alpha}
\newcommand{\be}{\beta}
\newcommand{\ga}{\gamma}
\newcommand{\de}{\delta}
\newcommand{\De}{\Delta}
\newcommand{\eps}{\epsilon}
\newcommand{\veps}{\varepsilon}
\newcommand{\io}{\iota}
\newcommand{\lam}{\lambda}
\newcommand{\Lam}{\Lambda}
\newcommand{\om}{\omega}

\newcommand{\sig}{\sigma}
\newcommand{\Sig}{\Sigma}
\renewcommand{\th}{\theta}
\newcommand{\vphi}{\varphi}
\newcommand{\C}{\mathbb{C}}
\newcommand{\R}{\mathbb{R}}

\newcommand{\PSL}[2]{\rm{PSL}(#1,#2)}
\newcommand{\gfr}{\mathfrak{g}}
\newcommand{\id}{\rm{id}}
\newcommand{\proj}[1]{\rm{$\C P$}^{#1}}
\newcommand{\abs}[1]{\lvert#1\rvert}
\newcommand{\Abs}[1]{\left\lvert#1\right\rvert}
\newcommand{\norm}[1]{\lVert#1\rVert}
\newcommand{\Norm}[1]{\left\lVert#1\right\rVert}
\newcommand{\map}[3]{#1\colon#2\rightarrow#3}
\newcommand{\mapin}[3]{#1\colon#2\hookrightarrow#3}
\newcommand{\lto}{\longrightarrow}
\newcommand{\lmapsimeq}[3]{#1\colon#2\stackrel{\simeq}{\lto}#3}
\newcommand{\deq}{\mathrel{\mathop:}=}
\newcommand{\fall}[2]{\noindent Case #1: \textit{#2.}}
\renewcommand{\l}{\langle}
\renewcommand{\r}{\rangle}
\newcommand{\la}{\De}
\newcommand{\schritt}[1]{\noindent \textbf{Step #1} \,}
\newcommand{\rv}{0}
\newcommand{\pc}{\Sig_{\rv}}
\newcommand{\bub}{\Sig_{\al}}
\newcommand{\edge}[2]{#1E#2}
\newcommand{\sv}{A,\textbf{u},\textbf{z}}
\newcommand{\G}{\mathcal{G}}

\newcommand{\vtilde}{\tilde{v}}

\newcommand{\utilde}{\tilde{u}}
\newcommand{\Xtilde}{\widetilde{\X}}
\newcommand{\Jtilde}{\tilde{J}}
\newcommand{\X}{M}
\newcommand{\vnu}{A_{\nu},u_{\nu},\textbf{z}_{\nu}}
\newcommand{\A}{\mathcal{A}}
\newcommand{\loc}{\rm{loc}}
\newcommand{\hor}{\rm{hor}}
\newcommand{\ver}{\rm{vert}}
\newcommand{\Atilde}{\tilde{A}}
\newcommand{\vphitilde}{\widetilde{\vphi}}
\newcommand{\omtilde}{\widetilde{\om}}
\newcommand{\sigtilde}{\widetilde{\sig}}
\newcommand{\Ahat}{\hat{A}}
\newcommand{\uhat}{\hat{u}}

\begin{document}

\title[Gromov compactness for vortices]{Removal of singularities and Gromov compactness for symplectic vortices}

\author[A. Ott]{Andreas Ott}
\address{Centre for Mathematical Sciences, University of Cambridge, Wilberforce Road, Cambridge CB3 0WB, United Kingdom}
\email{a.ott@dpmms.cam.ac.uk}

\subjclass[2000]{37J15, 53D20, 53D45, 58J05}

\begin{abstract}
	We prove that the moduli space of gauge equivalence classes of symplectic vortices with uniformly bounded energy in a compact Hamiltonian manifold admits a Gromov compactification by polystable vortices. This extends results of Mundet i Riera for circle actions to the case of arbitrary compact Lie groups. Our argument relies on an a priori estimate for vortices that allows us to apply techniques used by McDuff and Salamon in their proof of Gromov compactness for pseudoholomorphic curves. As an intermediate result we prove a removable singularity theorem for symplectic vortices.
\end{abstract}

\maketitle

\tableofcontents

\section{Introduction and main results}
\label{SectionIntroductionAndMainResults}

For any compact Lie group $G$ and any Hamiltonian $G$-manifold $(\X,\om,\mu)$ with moment map $\map{\mu}{M}{\gfr^{\ast} \cong \gfr}$, Cieliebak, Gaio, Mundet i Riera, and Salamon \cite{Cieliebak/J-holomorphic-curves-moment-maps-and-invariants-of-Hamiltonian-group-actions, Cieliebak/The-symplectic-vortex-equations-and-invariants-of-Hamiltonian-group-actions, Gaio/Gromov-Witten-invariants-of-symplectic-quotients-and-adiabatic-limits} and Mundet i Riera and Tian \cite{Mundet-i-Riera/A-Hitchin-Kobayashi-correspondence-for-Kahler-fibrations, Mundet-i-Riera/Hamiltonian-Gromov-Witten-invariants, Mundet-i-Riera/A-compactification-of-the-moduli-space-of-twisted-holomorphic-maps} studied the \emph{symplectic vortex equations}
\begin{equation} \label{EquationVortexEquations}
	\delbar_{J,A}(u) = 0, \quad F_{A} + \mu(u) \dvol_{\Sig} = 0
\end{equation}
for pairs $(A,u)$, where $A$ is a connection on a fixed principal $G$-bundle $P$ over a compact Riemann surface $\Sig$ equipped with a fixed complex structure and a fixed area form $\dvol_{\Sig}$, $F_{A}$ denotes the curvature of $A$, $u$ is a $G$-equivariant map $P\to\X$, and $J$ is a $G$-invariant $\om$-compatible almost complex structure on $\X$. Solutions of these equations are called \emph{vortices} and may be regarded as gauge-theoretic deformations of $J$-holomorphic curves in $\X$. For a proper moment map $\mu$ and~$\X$ equivariantly convex at infinity, Cieliebak, Gaio, Mundet i Riera, and Salamon \cite{Cieliebak/The-symplectic-vortex-equations-and-invariants-of-Hamiltonian-group-actions} proved that the moduli space of gauge equivalence classes of vortices with uniformly bounded energy is compact under the additional assumption that~$\X$ is symplectically aspherical. The latter condition means that the symplectic form $\om$ vanishes on all spherical homology classes in~$\X$ and ensures that no bubbling off of spheres occurs. If this condition is dropped, the moduli space will in general no longer be compact, and the question arises as to whether it admits a compactification in a way similar to the Gromov compactification of the moduli space of pseudoholomorphic curves as in \cite{Gromov/Pseudoholomorphic-curves-in-symplectic-manifolds, Kontsevich/Enumeration-of-rational-curves-via-torus-actions, Hofer/Floer-homology-and-Novikov-rings, McDuff/J-holomorphic-curves-and-symplectic-topology, Parker/Pseudo-holomorphic-maps-and-bubble-trees, Ye/Gromovs-compactness-theorem-for-pseudo-holomorphic-curves, Hummel/Gromovs-compactness-theorem-for-pseudo-holomorphic-curves}. In the special case of $G=S^{1}$, a positive answer was first given by Mundet i Riera \cite{Mundet-i-Riera/Hamiltonian-Gromov-Witten-invariants} who constructed a Gromov compactification for fixed complex structure on~$\Sig$, using the compactness results for pseudoholomorphic curves of Ivashkovich and Shevchishin \cite{Ivashkovich/Gromov-compactness-theorem-for-J-complex-curves-with-boundary}. Later Mundet i Riera and Tian \cite{Mundet-i-Riera/A-compactification-of-the-moduli-space-of-twisted-holomorphic-maps} established a compactification in the case of $G=S^{1}$ also for varying complex structure on~$\Sig$.

The goal of this article is to construct a Gromov compactification of the moduli space of vortices for all compact Lie groups $G$ and for fixed complex structure on the Riemann surface~$\Sig$, see Theorem~\ref{TheoremGromovCompactness} below, combining methods from symplectic geometry and gauge theory. A feature of our approach is that we do not appeal to \cite{Ivashkovich/Gromov-compactness-theorem-for-J-complex-curves-with-boundary}; rather, we apply the techniques that were used by McDuff and Salamon \cite{McDuff/J-holomorphic-curves-and-symplectic-topology} to construct a Gromov compactification for the moduli space of pseudoholomorphic curves. Our result crucially relies on a removable singularity theorem for vortices on the punctured disk, see Theorem~\ref{TheoremRemovalOfSingularities} below.

The Gromov compactification we shall construct plays a central role in the definition of gauged Gromov-Witten invariants for Hamiltonian $G$-manifolds. More specifically, it is used in the algebro-geometric approach to gauged Gromov-Witten theory due to Gonz\'alez and Woodward \cite{Gonzalez/Area-dependence-in-gauged-Gromov-Witten-theory} in order to define such invariants for smooth projective $G$-varieties. Moreover, building on the results of the present paper, the author constructed a Gromov compactification for the moduli space of solutions of the non-local symplectic vortex equations in order to define gauged Gromov-Witten invariants for monotone Hamiltonian $G$-manifolds, see \cite{Ott/The-non-local-symplectic-vortex-equations-and-gauged-Gromov-Witten-invariants}. The present article grew out of a joint project with E.\,Gonz\'alez, C.\,Woodward, and F.\,Ziltener \cite{Gonzalez/Symplectic-vortices-with-fixed-holonomy-at-infinity} that carries these ideas further to the study of vortices with fixed holonomy on punctured Riemann surfaces, with the goal of defining the corresponding invariants. As another application, our result enters into the compactification of the moduli space of vortices on the affine line, which constitutes an intermediate step in the definition of the quantum Kirwan map in gauged Gromov-Witten theory due to Ziltener \cite{Ziltener/Symplectic-vortices-on-the-complex-plane-and-quantum-cohomology, Ziltener/A-quantum-Kirwan-map:-Bubbling-and-Fredholm-theory-for-symplectic-vortices-over-the-plane}, Nguyen, Woodward, and Ziltener~\cite{Nguyen/Morphisms-of-cohomological-field-theory-algebras-and-quantization-of-the-Kirwan-map} and Woodward \cite{Woodward/Quantum-Kirwan-morphism-and-Gromov-Witten-invariants-of-quotients}. Our approach conjecturally admits an extension so as to cover symplectic vortices with Lagrangian boundary conditions, in which case disk bubbling may also occur; this would allow for a generalization of the gauged Lagrangian Floer theory of Frauenfelder \cite{Frauenfelder/Floer-homology-of-symplectic-quotients-and-the-Arnold-Givental-conjecture}, see Woodward \cite{Woodward/Gauged-Floer-theory-of-toric-moment-fibers}.

\smallskip

We now state the two main theorems of this article. Let $G$ be a compact connected Lie group, with Lie algebra denoted by $\gfr$, and let $(\X,\om,\mu)$ be a closed Hamiltonian $G$-manifold. Explicitly, this means that $\X$ is a compact smooth $G$-manifold without boundary equipped with a $G$-invariant symplectic form $\om$ and a smooth $G$-equivariant moment map $\map{\mu}{\X}{\gfr^{\ast} \cong \gfr}$ such that the identity
\[
\io(X_{\xi})\,\om = \dop \l\mu,\xi\r_{\gfr}
\]
holds for every $\xi \in \gfr$, where $X_{\xi}$ denotes the fundamental vector field of the infinitesimal action of $\xi$ on $\X$ that is induced by the $G$-action. Here we identify the Lie algebra $\gfr$ with its dual $\gfr^{\ast}$ by means of some fixed invariant inner product $\l\cdot,\cdot\r_{\gfr}$ on $\gfr$. We further fix a smooth $G$-invariant $\om$-compatible almost complex structure $J$ on $\X$, which defines a $G$-invariant Riemannian metric $\l\cdot,\cdot\r_{J} \deq \om(\cdot,J\cdot)$ on $\X$. We refer to \cite{Cieliebak/The-symplectic-vortex-equations-and-invariants-of-Hamiltonian-group-actions, Cieliebak/J-holomorphic-curves-moment-maps-and-invariants-of-Hamiltonian-group-actions} for the details.

\smallskip

Our first result is a removable singularity theorem for vortices on the punctured disk. We begin by recalling from \cite{Cieliebak/The-symplectic-vortex-equations-and-invariants-of-Hamiltonian-group-actions} the definition of the symplectic vortex equations in local coordinates. Let \mbox{$D \subset \C$ }be an open subset and write the complex coordinate as $s + it$. Fix a smooth function $\map{\lam}{D}{(0,\infty)}$. The \emph{symplectic vortex equations on $D$} are the system of nonlinear partial differential equations
\begin{equation} \label{EquationLocalVortexEquations}
\begin{split}
	\del_{s}\!u + X_{\Phi}(u) + J \bigl( \del_{t}\!u + X_{\Psi}(u) \bigr) &= 0, \\
	\del_{s}\!\Psi - \del_{t}\!\Phi + [\Phi,\Psi] + \lam^{2} \cdot \mu(u) &= 0,
\end{split}
\end{equation}
where $\map{\Phi,\Psi}{D}{\gfr}$ and $\map{u}{D}{\X}$ are smooth maps. A triple $(\Phi,\Psi,u)$ that solves equations~(\ref{EquationLocalVortexEquations}) will be called a \emph{vortex on $D$}. Its \emph{Yang-Mills-Higgs energy} is defined by
\begin{eqnarray} \label{EquationYMHEnergyLocal}
	E(\Phi,\Psi,u;D) \deq \int_{D} e(\Phi,\Psi,u),
\end{eqnarray}
where 
\begin{eqnarray} \label{EquationYMHEnergyDensityLocal}
	e(\Phi,\Psi,u) \deq \Abs{\del_{s}\!u + X_{\Phi}(u)}_{J}^{2} + \lam^{2} \cdot \Abs{\mu(u)}_{\gfr}^{2}
\end{eqnarray}
denotes the \emph{Yang-Mills-Higgs energy density}. Here the norms are understood with respect to the metric $\l\cdot,\cdot\r_{J}$ on $\X$ and the inner product $\l\cdot,\cdot\r_{\gfr}$ on $\gfr$, respectively. Note that this energy may be infinite.

\smallskip

We are now ready to state our first theorem. Let $B \subset \C$ denote the closed unit disk. Fix a smooth function $\map{\lam}{B}{(0,\infty)}$ and consider the vortex equations (\ref{EquationLocalVortexEquations}) on the punctured disk~$B \setminus \{0\}$.

\begin{theorem}[Removal of singularities] \label{TheoremRemovalOfSingularities}
	Let $(\Phi,\Psi,u)$ be a smooth vortex on the punctured disk $B \setminus \{0\}$, and assume that the following holds.
\begin{enumerate}[itemsep=0.5ex, itemindent=0cm, labelsep=1ex, leftmargin=1cm]
	\item $\Phi$ and $\Psi$ extend continuously to all of $B$.
	\item $(\Phi,\Psi,u)$ has finite Yang-Mills-Higgs energy $E(\Phi,\Psi,u;B) < \infty$.
\end{enumerate}
Then $u$ is of Sobolev class $W^{1,p}$ on $B$ for every real number $p>2$.
\end{theorem}

The reader is referred to \cite[App.\,B]{Wehrheim/Uhlenbeck-compactness} for the definition of Sobolev spaces of maps into $\X$. We will prove Theorem~\ref{TheoremRemovalOfSingularities} in Section~\ref{SectionRemovalOfSingularities}; it will later play a crucial role in the proof of Gromov compactness for vortices. We now introduce some notation in order to state the main result, Theorem~\ref{TheoremGromovCompactness} below, which establishes Gromov compactness for symplectic vortices.

\smallskip

To begin with, we recall from \cite{Cieliebak/The-symplectic-vortex-equations-and-invariants-of-Hamiltonian-group-actions} the definition of the symplectic vortex equations on Riemann surfaces. Let $\Sig$ be a compact Riemann surface without boundary, that is endowed with a fixed complex structure $j_{\Sig}$ and a fixed area form~$\dvol_{\Sig}$, and denote the corresponding K\"ahler metric by $\l\cdot,\cdot\r_{\Sig} \deq \dvol_{\Sig}(\,\cdot, j_{\Sig}\,\cdot)$. Let $\map{\pi}{P}{\Sig}$ be a smooth principal $G$-bundle over $\Sig$. We shall write $\A(P)$ for the space of smooth connections on $P$ and $C^{\infty}(P,\X)^{G}$ for the space of smooth $G$-equivariant maps $P \to \X$ (see \cite[App.\,A]{Wehrheim/Uhlenbeck-compactness} for basic definitions in gauge theory). For any pair $(A,u) \in \A(P) \xop C^{\infty}(P,\X)^{G}$ we denote by
\begin{eqnarray} \label{EquationTwistedDerivative}
	\dop_{A}\!u \deq \dop\!u + X_{A}(u)
\end{eqnarray}
the twisted derivative of $u$, and define the corresponding nonlinear Cauchy-Riemann operator by
\[
\delbar_{J,A}(u) \deq \frac{1}{2} \, \bigl( \dop_{A}\!u + J(u) \circ \dop_{A}\!u \circ j_{\Sig} \bigr).
\]
The symplectic vortex equations on the Riemann surface $\Sig$ are the system of nonlinear partial differential equations (\ref{EquationVortexEquations}), i.\,e.\,
\[
\delbar_{J,A}(u) = 0, \quad F_{A} + \mu(u) \dvol_{\Sig} = 0,
\]
for pairs $(A,u) \in \A(P) \xop C^{\infty}(P,\X)^{G}$, where
\[
F_{A} \deq \dop\!A + \frac{1}{2} [A \wedge A]
\]
denotes the curvature of $A$. Its solutions are called vortices. We define the \emph{Yang-Mills-Higgs energy} of a vortex $(A,u)$ on an open subset $U \subset \Sig$ by
\begin{eqnarray} \label{EquationYMHEnergyGlobal}
	E(A,u;U) \deq \int_{U} \left( \frac{1}{2} \Abs{\dop_{A}\!u}_{J}^{2} + \Abs{\mu(u)}_{\gfr}^{2} \right) \dvol_{\Sig},
\end{eqnarray}
and write $E(A,u)$ for the Yang-Mills-Higgs energy of $(A,u)$ on $\Sig$. Here the norm $\abs{\,\cdot\,}_{J}$ is understood with respect to the metric $\l\cdot,\cdot\r_{J}$ on $\X$ and the metric $\l\cdot,\cdot\r_{\Sig}$ on $\Sig$ (see \cite[Sec.\,2.2]{McDuff/J-holomorphic-curves-and-symplectic-topology} for details), while the norm $\abs{\,\cdot\,}_{\gfr}$ is understood with respect to the inner product $\l\cdot,\cdot\r_{\gfr}$ on $\gfr$.

\begin{remark} \label{RemarkLocalFromGlobal}
	Equations (\ref{EquationLocalVortexEquations}) are a local version of equations (\ref{EquationVortexEquations}) in the following sense. Let $D \subset \C$ be an open subset of $\C$, and let $\map{\vphi}{D}{\Sig}$ be a holomorphic chart with a lift $\map{\vphitilde}{D}{P}$ that locally trivializes the bundle~$P$. A vortex $(A,u)$ determines a smooth triple $(\Phi,\Psi,u^{\loc})$ on $D$ by
\[
\vphitilde^{\,\ast}A = \Phi \dop\!s + \Psi \dop\!t \quad \text{and} \quad u^{\loc} = u \circ \vphitilde,
\]
and the area form $\dvol_{\Sig}$ gives rise to a smooth function $\map{\lam}{D}{(0,\infty)}$ by
\[
\vphi^{\ast} \dvol_{\Sig} = \lam^{2} \, \dop\!s \wedge \dop\!t.
\]
A short calculation now shows that the triple $(\Phi,\Psi,u^{\loc})$ satisfies the vortex equations
\[
\begin{split}
	\del_{s}\!u^{\loc} + X_{\Phi}\bigl( u^{\loc} \bigr) + J \bigl( \del_{t}\!u^{\loc} + X_{\Psi}\bigl( u^{\loc} \bigr) \bigr) &= 0, \\
	\del_{s}\!\Psi - \del_{t}\!\Phi + [\Phi,\Psi] + \lam^{2} \cdot \mu\bigl( u^{\loc} \bigr) &= 0
\end{split}
\]
on $D$ (see \cite[Prop.\,2.2]{Cieliebak/The-symplectic-vortex-equations-and-invariants-of-Hamiltonian-group-actions} for details). Moreover, the Yang-Mills-Higgs energy density (\ref{EquationYMHEnergyDensityLocal}) of the vortex $(\Phi,\Psi,u^{\loc})$ can be expressed in terms of $(A,u)$ by the identity
\begin{equation} \label{EqualityEnergyDensityLocalVsGlobal}
	\Abs{\del_{s}\!u^{\loc} + X_{\Phi}\bigl( u^{\loc} \bigr)}_{J}^{2} + \lam^{2} \cdot \Abs{\mu\bigl( u^{\loc} \bigr)}_{\gfr}^{2} = \left( \frac{1}{2} \Abs{\dop_{A}\!u \circ \vphitilde}_{J}^{2} + \Abs{\mu(u \circ \vphitilde)}_{\gfr}^{2} \right) \cdot \lam^{2};
\end{equation}
the corresponding Yang-Mills-Higgs energies are then related by
\begin{eqnarray} \label{EquationEnergiesLocalVsGlobal}
	E\bigl( \Phi,\Psi,u^{\loc};D \bigr) = E\bigl( A,u;\vphi(D) \bigr).
\end{eqnarray}
This justifies the ad hoc definitions at the beginning of this section.
\end{remark}

Next we introduce polystable vortices and give a definition of Gromov convergence for sequences of vortices. Our definitions are inspired by similar definitions due to Mundet i Riera \cite{Mundet-i-Riera/Hamiltonian-Gromov-Witten-invariants}, Gonz\'alez and Woodward \cite{Gonzalez/Area-dependence-in-gauged-Gromov-Witten-theory}, Ziltener \cite{Ziltener/A-quantum-Kirwan-map:-Bubbling-and-Fredholm-theory-for-symplectic-vortices-over-the-plane}, and McDuff and Salamon \cite{McDuff/J-holomorphic-curves-and-symplectic-topology}.

\smallskip

We begin by recalling some basic facts about trees and nodal curves from \cite[App.\,D.2 and Sec.\,5.1]{McDuff/J-holomorphic-curves-and-symplectic-topology}, slightly modifying the notation and terminology. A \emph{tree} is a connected graph without cycles. We denote it by $(V,E)$, where $V$ is a finite set of vertices and $E \subset V \xop V$ is the edge relation. A rooted tree is a tree $(V,E)$ which has a distinguished \emph{root vertex} $\rv\in V$. We will indicate this in the notation by writing the set of vertices $V$ as a disjoint union $V = \{\rv\} \sqcup V_{S}$. The elements of $V_{S}$ are called \emph{spherical vertices}. Note that $V_{S}$ may be empty. Let $n$ be a nonnegative integer. An \emph{n-labeled tree} is a triple $T=(V,E,\Lam)$ consisting of a rooted tree $(V = \{\rv\} \sqcup V_{S},E)$ and a labeling
\[
\map{\Lam}{\{1,\ldots,n\}}{V}, \quad i \mapsto \al_{i}.
\]
Given an $n$-labeled tree $T=(V=\{\rv\} \sqcup V_{S},E,\Lam)$, by a \emph{normalized nodal curve of combinatorial type $T$} we mean a tuple
\[
(\bs{\Sig},\textbf{z}) \deq \bigl( \{\bub\}_{\al\,\in\,V},\{z_{\al\be}\}_{\edge{\al}{\be}}, \{\al_{i},z_{i}\}_{1 \le i \le n} \bigr),
\]
often just written as
\[
\textbf{z} = \bigl( \{z_{\al\be}\}_{\edge{\al}{\be}}, \{\al_{i},z_{i}\}_{1 \le i \le n} \bigr),
\]
consisting of a compact Riemann surface $\pc$, called the \emph{principal component} associated to the root vertex $\rv$, a \emph{spherical component} $\bub \deq \proj{1}$ for every spherical vertex $\al\in V_{S}$, nodal points $z_{\al\be} \in \bub$ labeled by the directed edges $\edge{\al}{\be}$ of $T$, and~$n$ distinct marked points $z_{i} \in \Sig_{\al_{i}}$, $i = 1,\ldots,n$, such that for every $\al \in V$ the points $z_{\al\be}$ for $\edge{\al}{\be}$ and $z_{i}$ for $\al_{i} = \al$ are pairwise distinct. We denote the set of nodal points on the component $\bub$, $\al\in V$, by
\[
Z_{\al} \deq \big\{ z_{\al\be} \,\big|\, \edge{\al}{\be} \big\},
\]
and we define the set of special points on $\bub$ by
\[
Y_{\al} \deq Z_{\al} \cup \big\{ z_{i} \,\big|\, \al_{i}=\al \big\}.
\]
For any two vertices $\al,\be\in V$ not connected by an edge, we denote by $z_{\al\be}$ the unique nodal point on $\bub$ corresponding to the first edge on the chain of edges running from $\al$ to $\be$. Moreover, we define the point~$z_{\rv i}$ on the principal component $\pc$ to be
\[
z_{\rv i} \deq	\begin{cases}	z_{i}				&\text{ if $\al_{i}=\rv$;} \\
								z_{\rv \al_{i}}		&\text{ if $\al_{i}\in V_{S}$.}
				\end{cases}
\]
In other words, if $z_{i}$ lies on a spherical component then $z_{\rv i}$ is the unique nodal point on the principal component at which the bubble tree containing~$z_{i}$ is attached; otherwise, i.\,e., if $z_{i}$ lies on the principal component, then $z_{\rv i}$ coincides with $z_{i}$.

\smallskip

Let $P(\X) \deq P \xop_{G} \X$ be the symplectic fiber bundle over $\Sig$ that is associated to the $G$-bundle $P \to \Sig$ and the $G$-manifold $\X$. The points on $P(\X)$ will be denoted by $[p,x]$, for $p\in P$ and $x \in \X$.

\begin{remark} \label{RemarkMapVersusSection}
	Note that we may equivalently think of a $G$-equivariant map $\map{u}{P}{\X}$ as a section $\map{u}{\Sig}{P(\X)}$. In fact, this section is defined by
\[
\Sig \ni z \mapsto \bigl[ p,u(p) \bigr] \in P(\X), \quad \pi(p) = z,
\]
where $\map{\pi}{P}{\Sig}$ denotes the bundle projection. We will usually not distinguish between these two viewpoints in the notation and switch freely from one to the other, depending on the situation.
\end{remark}

\begin{definition}[Polystable vortices] \label{DefinitionPolystableVortex}
	Let $n$ be a nonnegative integer, and let $T=(V=\{\rv\} \sqcup V_{S},E,\Lam)$ be an $n$-labeled tree. A \emph{polystable vortex of combinatorial type $T$} is a tuple
\[
(\sv) \deq \bigl( (A,u_{\rv}),\{u_{\al}\}_{\al\,\in\,V_{S}},\{z_{\al\be}\}_{\edge{\al}{\be}}, \{\al_{i},z_{i}\}_{1 \le i \le n} \bigr)
\]
consisting of
\begin{itemize}[itemsep=0.3ex, itemindent=0cm, labelsep=0.3cm, leftmargin=0.7cm]
	\item a normalized nodal curve $\bigl( \{\bub\}_{\al\,\in\,V},\{z_{\al\be}\}_{\edge{\al}{\be}}, \{\al_{i},z_{i}\}_{1 \le i \le n} \bigr)$ of combinatorial type $T$ with principal component $\pc \deq \Sig$;
	\item a vortex $(A,u_{\rv})$ on the principal component $\pc$;
	\item a $J$-holomorphic sphere $\map{u_{\al}}{\bub}{P(\X)_{z_{\rv\al}} \cong \X}$ in the fiber of $P(\X)$ over the nodal point $z_{\rv\al} \in \pc$, for every $\al\in V_{S}$
\end{itemize}
such that the following conditions are satisfied.
\begin{description}[itemsep=1ex, itemindent=0cm, labelsep=0.5cm, leftmargin=0.7cm]
	\item[(Connectedness)] $u_{\al}(z_{\al\be}) = u_{\be}(z_{\be\al})$ for all $\al,\be\in V$ such that $\edge{\al}{\be}$.
	\item[(Polystability)] $\abs{Y_{\al}}\ge3$ for all $\al\in V_{S}$ such that $u_{\al}$ is constant.
\end{description}
\end{definition}

\begin{remark}
	To understand the meaning of the (Connectedness) condition in the case $\al=\rv$, we think of the $G$-equivariant map $\map{u_{\rv}}{P}{\X}$ as a section $\map{u_{\rv}}{\Sig}{P(\X)}$ as explained in Remark \ref{RemarkMapVersusSection}; this condition then says that $u_{\rv}(z_{\rv\be}) = u_{\be}(z_{\be\rv})$ in the fiber of $P(\X)$ over the nodal point~$z_{\rv\be} \in \pc$.
\end{remark}

Given a polystable vortex $(\sv)$ of combinatorial type $T=(V=\{\rv\} \sqcup V_{S},E,\Lam)$, we define its \emph{energy} to be
\[
E(A,\textbf{u}) \deq E(A,u_{\rv}) + \sum_{\al\in V_{S}} E(u_{\al}),
\]
where $E(A,u_{\rv})$ is the Yang-Mills-Higgs energy of the vortex $(A,u_{\rv})$ and $E(u_{\al})$ denotes the energy of the $J$-holomorphic curve $\map{u_{\al}}{\bub}{P(\X)_{z_{\rv\al}}}$ (see \cite[Sec.\,2.2]{McDuff/J-holomorphic-curves-and-symplectic-topology}). As a special case of Definition~\ref{DefinitionPolystableVortex}, by an \emph{$n$-marked vortex} we mean a tuple
\[
(A,u,\textbf{z}) = (A,u,z_{1}, \ldots, z_{n})
\]
consisting of a vortex $(A,u)$ and a sequence $z_{1}, \ldots, z_{n}$ of $n$ distinct marked points on $\Sig$.

The group $\G(P) \deq C^{\infty}(P,G)^{G}$ of smooth gauge transformations of $P$ acts on the space $\A(P) \xop C^{\infty}(P,\X)^{G}$ from the right by
\begin{eqnarray} \label{EquationActionOfGaugeTransformation}
	g^{\ast}(A,u) \deq \bigl( g^{-1}Ag + g^{-1}\dop\!g,g^{-1}u \bigr)
\end{eqnarray}
Note that the vortex equations (\ref{EquationVortexEquations}) and the Yang-Mills-Higgs energy (\ref{EquationYMHEnergyGlobal}) are invariant under this action.

\smallskip

For $z_{0} \in \Sig$ and $r > 0$, we denote by $B_{r}(z_{0}) \deq \{ z\in \Sig \,\big|\, \Abs{z - z_{0}} \le r \}$ the closed disk in $\Sig$ of radius $r$ centered at the point $z_{0}$, understood with respect to the metric $\l\cdot,\cdot\r_{\Sig}$. Let $B \subset \C$ be the closed unit disk, and fix an identification $\proj{1} \cong \C \cup \{\infty\}$. The next definition builds on the definition of Gromov convergence for pseudoholomorphic curves due to McDuff and Salamon \cite[Def.\,5.2.1]{McDuff/J-holomorphic-curves-and-symplectic-topology}.

\begin{definition}[Gromov convergence] \label{DefinitionGromovConvergence}
	Let $n$ be a nonnegative integer, and let $T=(V=\{\rv\} \sqcup V_{S},E,\Lam)$ be an $n$-labeled tree. A sequence of $n$-marked vortices
\[
(A_{\nu},u_{\nu},\textbf{z}_{\nu}) = (A_{\nu},u_{\nu},z^{\nu}_{1},\ldots,z^{\nu}_{n})
\]
is said to \emph{Gromov converge} to a polystable vortex of combinatorial type $T$
\[
(\sv) = \bigl( (A,u_{\rv}),\{u_{\al}\}_{\al\,\in\,V_{S}},\{z_{\al\be}\}_{\edge{\al}{\be}}, \{\al_{i},z_{i}\}_{1 \le i \le n} \bigr)
\]
if there exist
\begin{itemize}[itemsep=0.3ex, itemindent=0cm, labelsep=0.3cm, leftmargin=0.7cm]
	\item a sequence of smooth gauge transformations $g_{\nu} \in \G(P)$;
	\item a sufficiently small number $r>0$ such that the following holds: For every nodal point \mbox{$z_{\rv\al} \in Z_{\rv}$}, where $\al \in V_{S}$, there exists a holomorphic chart $\map{\vphi_{z_{\rv\al}}}{B}{B_{r}(z_{\rv\al})}$ such that $\vphi_{z_{\rv\al}}(0)=z_{\rv\al}$ and $B_{r}(z_{\rv\al}) \cap Z_{\rv} = \{z_{\rv\al}\}$;
	\item a sequence of M\"obius transformations $\phi^{\nu}_{\al} \in \Aut(\bub) \cong \PSL{2}{\C}$ for every $\al\in V_{S}$
\end{itemize}
such that the following holds.
\begin{description}[topsep=1ex, itemsep=1ex, itemindent=0cm, labelsep=0.3cm, leftmargin=0.7cm]
	\item[(Map)] The sequence
\[
\Bigl( g_{\nu}^{\ast} A_{\nu}, g_{\nu}^{-1} u_{\nu}, \bigl\{ (g_{\nu}^{-1}u_{\nu}) \circ \vphi_{z_{\rv\al}} \circ \phi^{\nu}_{\al} \bigr\}_{\al\,\in\,V_{S}} \Bigr)
\]
converges to
\[
\bigl( A,u_{\rv},\{u_{\al}\}_{\al\,\in\,V_{S}} \bigr)
\]
in the following sense.
\begin{enumerate}[topsep=1ex, itemsep=0.3ex, itemindent=0cm, labelsep=1ex, leftmargin=0.8cm]
	\item The sequence $g_{\nu}^{\ast}A_{\nu}$ converges to $A$ in $C^{0}$ on $\pc$.
	\item The sequence $(g_{\nu}^{\ast}A_{\nu},g_{\nu}^{-1} u_{\nu})$ converges to $(A,u_{\rv})$ in $C^{\infty}$ on compact subsets of $\pc \setminus Z_{\rv}$.
	\item For every $\al \in V_{S}$ the sequence $u^{\nu}_{\al} \deq ( g_{\nu}^{-1} u_{\nu} ) \circ \vphi_{z_{\rv\al}} \circ \phi^{\nu}_{\al}$ converges to $\map{u_{\al}}{\bub}{P(\X)_{z_{\rv\al}}}$ in $C^{1}$ on compact subsets of $\bub \setminus Z_{\al}$.
\end{enumerate}
	\item[(Energy)] The sequence $E(A_{\nu},u_{\nu})$ converges to $E(A,\textbf{u})$.
	\item[(Rescaling)] The sequence $\{ \phi^{\nu}_{\al} \}_{\al\in V_{S}}$ converges in the following sense.
\begin{enumerate}[topsep=1ex, itemsep=0.3ex, itemindent=0cm, labelsep=0.3cm, leftmargin=0.8cm]
	\item For every $\al \in V_{S}$ the sequence $\phi^{\nu}_{\al}$ converges to $0$ in $C^{\infty}$ on compact subsets of \mbox{$\proj{1} \setminus \{\infty\} \cong \C$}.
	\item If $\al,\be \in V_{S}$ are such that $\edge{\al}{\be}$ then the sequence $\phi^{\nu}_{\al\be} \deq (\phi^{\nu}_{\al})^{-1} \circ \phi^{\nu}_{\be}$ converges to $z_{\al\be}$ in $C^{\infty}$ on compact subsets of $\Sig_{\be} \setminus \{z_{\be\al}\}$.
\end{enumerate}
	\item[(Marked point)] For $i=1,\ldots,n$ the sequence of marked points $z_{i}^{\nu}$ converges in the following sense.
\begin{enumerate}[topsep=1ex, itemsep=0.3ex, itemindent=0cm, labelsep=0.3cm, leftmargin=0.8cm]
	\item If $\al_{i}=\rv$ then the sequence $z_{i}^{\nu}$ converges to $z_{i}$ in $\pc$.
	\item If $\al_{i}\in V_{S}$ then the sequence $(\vphi_{z_{\rv\al_{i}}} \circ \phi^{\nu}_{\al_{i}})^{-1}(z^{\nu}_{i})$ converges to $z_{i}$ in~$\Sig_{\al_{i}}$.
\end{enumerate}
\end{description}
\end{definition}

\begin{remark}
	To better understand how the M\"obius transformations $\phi^{\nu}_{\al}$ are used in Definition~\ref{DefinitionGromovConvergence}, we first recall that all spherical components $\bub$, $\al\in V_{S}$, are by definition just copies of the projective line $\proj{1}$, and that we have fixed an identification $\proj{1} \cong \C \cup \{\infty\}$. In (Map,\,iii), for large $\nu$ we may therefore think of the $\phi^{\nu}_{\al}$ as holomorphic maps $\map{\phi^{\nu}_{\al}}{B}{B}$, which are well-defined by (Rescaling,\,i). Likewise, in (Rescaling,\,ii) one should think of the~$\phi^{\nu}_{\al\be}$ as transformations of the projective line, and consider the nodal points~$z_{\al\be}$ and~$z_{\be\al}$ as lying thereon. Furthermore, we remind the reader that in (Map) we think of the maps $g_{\nu}^{-1} u_{\nu}$ as sections of the bundle $P(\X) = P \xop_{G} \X$ over~$\pc$ as in Remark \ref{RemarkMapVersusSection}.
\end{remark}

We are now in a position to state the main result of this article. The proof will be given in Sections \ref{SectionConvergenceModuloBubbling} and \ref{SectionGromovCompactness}.

\begin{theorem}[Gromov compactness] \label{TheoremGromovCompactness}
	Let $n$ be a nonnegative integer. Let $(A_{\nu},u_{\nu},\textbf{z}_{\nu})$ be a sequence of $n$-marked vortices whose Yang-Mills-Higgs energy satisfies a uniform bound
\[
\sup_{\nu} E(A_{\nu},u_{\nu}) < \infty.
\]
Then the sequence $(A_{\nu},u_{\nu},\textbf{z}_{\nu})$ has a Gromov convergent subsequence.
\end{theorem}

\begin{remark}
	One may extend Definition \ref{DefinitionGromovConvergence} so as to cover sequences of polystable vortices as well, by adapting the definition of Gromov convergence for sequences of stable pseudoholomorphic curves from McDuff and Salamon \cite[Def.\,5.5.1]{McDuff/J-holomorphic-curves-and-symplectic-topology}. Then Theorem \ref{TheoremGromovCompactness} generalizes in the sense that any sequence of polystable vortices $(A_{\nu},\textbf{u}_{\nu},\textbf{z}_{\nu})$ whose energy satisfies a uniform bound
\[
\sup_{\nu} E(A_{\nu},\textbf{u}_{\nu}) < \infty
\]
has a Gromov convergent subsequence. With some straightforward modifications, the proof of this generalization carries over from the proof of Gromov compactness for sequences of stable pseudoholomorphic curves in \cite[Thm.\,5.5.5]{McDuff/J-holomorphic-curves-and-symplectic-topology}. Moreover, one may define a \emph{Gromov topology} on the moduli space of polystable vortices with uniformly bounded energy as in \cite[Sec.\,5.6]{McDuff/J-holomorphic-curves-and-symplectic-topology}. The statement of the above-mentioned generalization of Theorem~\ref{TheoremGromovCompactness} may then be rephrased by saying that the moduli space of polystable vortices with uniformly bounded energy is compact.
\end{remark}

\begin{remark}
	Note that Theorem \ref{TheoremRemovalOfSingularities} and Theorem \ref{TheoremGromovCompactness} continue to hold for non-compact manifolds $\X$ under the additional assumption that the moment map $\mu$ is proper and~$\X$ is equivariantly convex (see hypotheses (H1) and (H2) in \cite{Cieliebak/The-symplectic-vortex-equations-and-invariants-of-Hamiltonian-group-actions}). To avoid additional technicalities, however, we will restrict ourselves to compact manifolds $\X$ throughout.
\end{remark}

This article is organized as follows. In Section~\ref{SectionAPrioriEstimate} we prove an a priori estimate for symplectic vortices, which will play a central role in all subsequent arguments. It is used in Section~\ref{SectionRemovalOfSingularities} to prove Theorem~\ref{TheoremRemovalOfSingularities}. The proof of Theorem~\ref{TheoremGromovCompactness} is divided into two parts. In Section~\ref{SectionConvergenceModuloBubbling} we establish a compactness result for vortices, ignoring any bubbling phenomena. Section~\ref{SectionGromovCompactness} is then concerned with the actual construction of the Gromov compactification, beginning with two preparatory subsections. In Section~\ref{SubSectionVorticesVersusPseudoholomorphicMaps} we explain how vortices may naturally be considered as pseudoholomorphic curves, and in Section~\ref{SubSectionBubblesConnectRevisited} we tailor the bubbling analysis from \cite{McDuff/J-holomorphic-curves-and-symplectic-topology} to our situation. We close with the proof of Gromov compactness in Section~\ref{SubSectionProofOfGromovCompactness} by combining the results from the earlier sections.

\bigskip

\noindent \textbf{Acknowledgments:} The author is very grateful to his supervisors, \mbox{Dietmar} Salamon and Christopher Woodward, for their encouragement and support, and for all their help in writing this article. He is also indebted to Fabian Ziltener for many helpful discussions, and he would further like to thank Eduardo Gonz\'alez, Tobias Hartnick, and Jan Swoboda for their valuable comments. Finally he would like to thank the Mathematics Department at Rutgers University, the Isaac Newton Institute for Mathematical Sciences, and the Max Planck Institute for Mathematics for their hospitality and excellent working conditions. The author was supported by ETH Research Grant TH-01 06-1 and by EPSRC Grant EP/F005431/1.

\section{A priori estimate}
\label{SectionAPrioriEstimate}

The goal of this section is to prove an a priori estimate for symplectic vortices. It relies on an a priori estimate for vortices proved by Gaio and Salamon \cite{Gaio/Gromov-Witten-invariants-of-symplectic-quotients-and-adiabatic-limits}, see also Frauenfelder \cite{Frauenfelder/Floer-homology-of-symplectic-quotients-and-the-Arnold-Givental-conjecture} and Ziltener \cite{Ziltener/The-invariant-symplectic-action-and-decay-for-vortices}. In fact we will prove two versions of this estimate: a local version for vortices on the punctured disk, and a global version for vortices on a Riemann surface.

\smallskip

We keep the notation introduced in Section \ref{SectionIntroductionAndMainResults}. Furthermore, for $w_{0} \in \C$ and $r > 0$ we denote by $B_{r}(w_{0}) \subset \C$ the closed disk of radius $r$ with center at~$w_{0}$.

\begin{theorem}[A priori estimate, local version] \label{TheoremAPrioriEstimate}
	Given a smooth function $\map{\lam}{B}{(0,\infty)}$, there exist constants $\de,C>0$ such that for all $w_{0}\in B$ and all $r > 0$ satisfying $B_{r}(w_{0}) \subset B$ the following holds. If $(\Phi,\Psi,u)$ is a vortex on $B_{r}(w_{0})$, then
\begin{eqnarray*}
	\hspace{2mm} E\bigl( \Phi,\Psi,u;B_{r}(w_{0}) \bigr) < \de \quad \Longrightarrow \quad e(\Phi,\Psi,u)(w_{0}) \le \frac{C}{r^{2}} \cdot E\bigl( \Phi,\Psi,u;B_{r}(w_{0}) \bigr).
\end{eqnarray*}
\end{theorem}

\smallskip

As a corollary of this theorem we obtain the following a priori estimate for vortices on the Riemann surface $\Sig$.

\begin{corollary}[A priori estimate, global version]
\label{CorollaryAPrioriEstimateOnSurface}
	There exist constants $R,\hbar,C>0$ such that for all $z_{0}\in \Sig$ and all $0 < r \le R$ the following holds. If $(A,u)$ is a vortex on $\Sig$, then
\[
E\bigl( A,u;B_{r}(z_{0}) \bigr) < \hbar \quad \Longrightarrow \quad
\frac{1}{2} \Abs{\dop_{A}\!u(z_{0})}_{J}^{2} + \Abs{\mu(u(z_{0}))}_{\gfr}^{2} \\ \le \frac{C}{r^{2}} \cdot E\bigl( A,u;B_{r}(z_{0}) \bigr).
\]
\end{corollary}

\bigskip

We will prove the corollary at the end of this section. The proof of Theorem~\ref{TheoremAPrioriEstimate} is based on the following two propositions.

\begin{proposition}[{Gaio and Salamon \cite[Sec.\,9]{Gaio/Gromov-Witten-invariants-of-symplectic-quotients-and-adiabatic-limits}}] \label{PropositionPDIForLaplacianOfEnergyDensity}
	Given a smooth function $\map{\lam}{B}{(0,\infty)}$, there exists a constant $c \ge 0$ such that for all $w_{0}\in B$ and all $r > 0$ satisfying $B_{r}(w_{0}) \subset B$ the following holds: If $(\Phi,\Psi,u)$ is a vortex on $B_{r}(w_{0})$, then its Yang-Mills-Higgs energy density $e \deq e(\Phi,\Psi,u)$ defined in (\ref{EquationYMHEnergyDensityLocal}) satisfies the partial differential inequality
\begin{eqnarray*}
	\la e \ge - c \cdot e^{2}.
\end{eqnarray*}
\end{proposition}

\begin{proof}
	The proof is the same as that of Claim 1 in the proof of \cite[Lemma 3.3]{Ziltener/The-invariant-symplectic-action-and-decay-for-vortices} and will therefore be omitted. It relies on \cite[Formula (9.6)]{Gaio/Gromov-Witten-invariants-of-symplectic-quotients-and-adiabatic-limits}.
\end{proof}

\begin{proposition}[{McDuff and Salamon \cite[Lemma 4.3.2]{McDuff/J-holomorphic-curves-and-symplectic-topology}}] \label{PropositionGeneralMeanValueInequality}
	Let $r>0$ and $c \ge 0$. If $\map{f}{B_{r}(0)}{\R}$ is a function of class $C^{2}$ that satisfies the inequalities
\[
\la f \ge -c \cdot f^{2}, \quad f\ge 0, \quad \int_{B_{r}(0)} f < \frac{\pi}{8 c},
\]
then
\[
f(0) \le \frac{8}{\pi r^{2}} \cdot \int_{B_{r}(0)} f.
\]
\end{proposition}

\medskip

\begin{proof}[Proof of Theorem~\ref{TheoremAPrioriEstimate}]
	Let $c$ be the constant from Proposition~\ref{PropositionPDIForLaplacianOfEnergyDensity}. Define $\de \deq \pi/8c$ and \mbox{$C \deq 8/\pi$}. Let $w_{0}\in B$ and $r > 0$ such that $B_{r}(w_{0}) \subset B$. Assume that $(\Phi,\Psi,u)$ is a vortex on $B_{r}(w_{0})$ and denote by $e \deq e(\Phi,\Psi,u)$ its Yang-Mills-Higgs energy density. Define a function $\map{f}{B_{r}(0)}{\R}$ by $f(w) \deq e(w + w_{0})$. Then Proposition~\ref{PropositionPDIForLaplacianOfEnergyDensity} implies that
\[
\la f \ge -c \cdot f^{2}, \quad f\ge 0.
\]
Hence it follows from Proposition~\ref{PropositionGeneralMeanValueInequality} that
\begin{eqnarray} \label{ImplicationMeanValueInequality}
	\int_{B_{r}(0)} f < \frac{\pi}{8 c} \quad \Longrightarrow \quad f(0) \le \frac{8}{\pi r^{2}} \cdot \int_{B_{r}(0)} f.
\end{eqnarray}
Since
\[
E\bigl( \Phi,\Psi,u;B_{r}(w_{0}) \bigr) = \int_{B_{r}(w_{0})} e = \int_{B_{r}(0)} f,
\]
the theorem follows from (\ref{ImplicationMeanValueInequality}).
\end{proof}

\smallskip

\begin{proof}[Proof of Corollary \ref{CorollaryAPrioriEstimateOnSurface}]
	Choose a finite collection of holomorphic disks
\[
\lmapsimeq{\vphi_{j}}{B}{\vphi_{j}(B) \subset \Sig}, \quad j=1,\ldots,N
\]
in such a way that the open sets $U_{j} \deq \vphi_{j}( B^{\circ} )$, where $B^{\circ}$ denotes the interior of $B$, form a covering of $\Sig$. By the Lebesgue number lemma, we find a constant $R>0$ such that for every~$z_{0}\in \Sig$ and every $0 < r < R$ there exists $j_{0}\in\{1,\ldots,N\}$ such that $B_{r}(z_{0}) \subset U_{j_{0}}$. The area form~$\dvol_{\Sig}$ defines smooth functions $\map{\lam_{j}}{B}{(0,\infty)}$ by the relation $\vphi_{j}^{\ast} \dvol_{\Sig} = \lam_{j}^{2} \, \dop\!s \wedge \dop\!t$. We denote by $\dop_{\Sig}$ the distance function on $\Sig$ defined by the metric $\l\cdot,\cdot\r_{\Sig}$ and by $\dop_{B}$ the distance function on $B$ corresponding to the Euclidean metric. By compactness of~$B$ there exist constants $c_{j} > 0$ such that
\begin{eqnarray} \label{InequalityLengthComparison}
	\dop_{\Sig} \bigl( \vphi_{j}(w_{1}),\vphi_{j}(w_{2}) \bigr) \le c_{j} \cdot \dop_{B}(w_{1},w_{2})
\end{eqnarray}
for all $w_{1},w_{2} \in B$. By Theorem~\ref{TheoremAPrioriEstimate} there exist constants~$\de_{j} > 0$ and $C_{j} > 0$, depending on the function~$\lam_{j}$, such that for all $w_{0}\in B$ and all $r > 0$ satisfying $B_{r}(w_{0}) \subset B$ the following holds. If $(\Phi,\Psi,u)$ is a vortex solving the equations
\[
\begin{split}
	\del_{s}\!u + X_{\Phi}(u) + J \bigl( \del_{t}\!u + X_{\Psi}(u) \bigr) &= 0, \\
	\del_{s}\!\Psi - \del_{t}\!\Phi + [\Phi,\Psi] + \lam_{j}^{2} \cdot \mu(u) &= 0
\end{split}
\]
on $B$, then
\begin{equation} \label{ImplicationAPrioriEstimateProofOfCorollary}
	E\bigl( \Phi,\Psi,u;B_{r}(w_{0}) \bigr) < \de_{j} \quad \Longrightarrow \quad e(\Phi,\Psi,u)(w_{0}) \le \frac{C_{j}}{r^{2}} \cdot E\bigl( \Phi,\Psi,u;B_{r}(w_{0}) \bigr).
\end{equation}
With all this understood, we define
\begin{eqnarray} \label{EqualityDefinitionOfConstantsAPrioriOnSurface}
	\hbar \deq \min_{1 \le j \le N} \bigl\{ \de_{j} \bigr\} \quad \text{and} \quad C \deq \max_{1 \le j \le N} \left\{ \frac{C_{j} \cdot c_{j}^{2}}{\norm{\lam_{j}}_{C^{0}(B)}^{2}} \right\}.
\end{eqnarray}

Let now $(A,u)$ be a vortex on $\Sig$. Let $z_{0} \in \Sig$ and $0 < r < R$, and assume that
\begin{eqnarray} \label{InequalityAssumptionOnEnergy}
	E \bigl( A,u;B_{r}(z_{0}) \bigr) < \hbar.
\end{eqnarray}
Since $r<R$ we have $B_{r}(z_{0}) \subset U_{j_{0}}$ for some $j_{0}\in \{1,\ldots,N\}$. By Remark~\ref{RemarkLocalFromGlobal}, locally in the chart $\map{\vphi_{j_{0}}}{B}{\Sig}$ the vortex $(A,u)$ is given by a triple $(\Phi,\Psi,u^{\loc})$ that solves the vortex equations
\[
\begin{split}
	\del_{s}\!u^{\loc} + X_{\Phi}\bigl( u^{\loc} \bigr) + J \bigl( \del_{t}\!u^{\loc} + X_{\Psi}\bigl( u^{\loc} \bigr) \bigr) &= 0, \\
	\del_{s}\!\Psi - \del_{t}\!\Phi + [\Phi,\Psi] + \lam_{j_{0}}^{2} \cdot \mu\bigl( u^{\loc} \bigr) &= 0.
\end{split}
\]
Define
\begin{eqnarray} \label{EquationDefinitionsProofOfCorollary}
	w_{0} \deq \vphi_{j_{0}}^{-1}(z_{0}) \quad \text{and} \quad \rho_{0} \deq \frac{r}{c_{j_{0}}}.
\end{eqnarray}
It follows from inequality (\ref{InequalityLengthComparison}) that $B_{\rho_{0}}(w_{0}) \subset \vphi_{j_{0}}^{-1} \bigl( B_{r}(z_{0}) \bigr) \subset B$. Hence
\begin{equation} \label{InequalityEnergyComparison}
	E\bigl( \Phi,\Psi,u^{\loc};B_{\rho_{0}}(w_{0}) \bigr) \le E\bigl( \Phi,\Psi,u^{\loc};\vphi_{j_{0}}^{-1}(B_{r}(z_{0}) \bigr) = E\bigl( A,u;B_{r}(z_{0}) \bigr)
\end{equation}
by formula~(\ref{EquationEnergiesLocalVsGlobal}). By assumption (\ref{InequalityAssumptionOnEnergy}) and the definition of $\hbar$ in (\ref{EqualityDefinitionOfConstantsAPrioriOnSurface}) it follows that
\[
E\bigl( \Phi,\Psi,u^{\loc};B_{\rho_{0}}(w_{0}) \bigr) < \hbar \le \de_{j_{0}}.
\]
Hence we may apply estimate (\ref{ImplicationAPrioriEstimateProofOfCorollary}) to the vortex $(\Phi,\Psi,u^{\loc})$, obtaining
\[
e\bigl( \Phi,\Psi,u^{\loc} \bigr)(w_{0}) \le \frac{C_{j_{0}}}{\rho_{0}^{2}} \cdot E\bigl( \Phi,\Psi,u^{\loc};B_{\rho_{0}}(w_{0}) \bigr).
\]
Using inequality~(\ref{InequalityEnergyComparison}) and the definition of $\rho_{0}$ in (\ref{EquationDefinitionsProofOfCorollary}), we further get
\[
e\bigl( \Phi,\Psi,u^{\loc} \bigr)(w_{0}) \le \frac{C_{j_{0}} \cdot c_{j_{0}}^{2}}{r^{2}} \cdot E\bigl( A,u;B_{r}(z_{0}) \bigr).
\]
Using the identity
\[
\frac{1}{2} \Abs{\dop_{A}\!u(z_{0})}^{2} + \Abs{\mu(u(z_{0}))}^{2} = e\bigl( \Phi,\Psi,u^{\loc} \bigr)(w_{0}) \cdot \lam_{j_{0}}^{-2}(w_{0})
\]
which holds by formula~(\ref{EqualityEnergyDensityLocalVsGlobal}), we arrive at
\[
\frac{1}{2} \Abs{\dop_{A}\!u(z_{0})}^{2} + \Abs{\mu(u(z_{0}))}^{2} \le \frac{C_{j_{0}} \cdot c_{j_{0}}^{2} \cdot \norm{\lam_{j_{0}}}_{C^{0}(B)}^{-2}}{r^{2}} \cdot E\bigl( A,u;B_{r}(z_{0}) \bigr).
\]
The a priori estimate now follows from the definition of $C$ in (\ref{EqualityDefinitionOfConstantsAPrioriOnSurface}). This proves Corollary \ref{CorollaryAPrioriEstimateOnSurface}.
\end{proof}

\section{Removal of singularities}
\label{SectionRemovalOfSingularities}

The goal of this section is to prove Theorem~\ref{TheoremRemovalOfSingularities}. We use Gromov's graph construction to reduce this problem to removal of singularities for certain punctured pseudoholomorphic curves. This will enable us to apply techniques from McDuff and Salamon \cite[Sec.\,4.5]{McDuff/J-holomorphic-curves-and-symplectic-topology}.

We keep the notation introduced in Section \ref{SectionIntroductionAndMainResults}. Let us fix a smooth function $\map{\lam}{B}{(0,\infty)}$, and let $(\Phi,\Psi,u)$ be a smooth vortex on the punctured disk $B \setminus \{0\}$ such that
\begin{description}[topsep=1ex, itemsep=0.5ex, itemindent=0cm, labelsep=1ex, leftmargin=0.5cm]
	\item[(R1)] $\Phi$ and $\Psi$ extend continuously to all of $B$;
	\item[(R2)] $(\Phi,\Psi,u)$ has finite Yang-Mills-Higgs energy $E(\Phi,\Psi,u;B) < \infty$.
\end{description}
It will be convenient to work with the smooth connection 1-form
\[
A \deq \Phi \dop\!s + \Psi \dop\!t
\]
on $B \setminus \{0\}$ that is determined by the functions $\Phi$ and $\Psi$. By Remark \ref{RemarkLocalFromGlobal}, the first vortex equation~(\ref{EquationLocalVortexEquations}) may then be written in the form
\begin{eqnarray} \label{EquationLocalVortexEquationsRemovalOfSingularities}
	\delbar_{J,A}(u) \deq \frac{1}{2} \, \bigl( \dop_{A}\!u + J(u) \circ \dop_{A}\!u \circ i \bigr) = 0,
\end{eqnarray}
and the Yang-Mills-Higgs energy density (\ref{EquationYMHEnergyDensityLocal}) of the vortex $(\Phi,\Psi,u)$ may be expressed in terms of $(A,u)$ by
\begin{eqnarray} \label{EqualityLocalYMHEnergyOfVortex}
	e(\Phi,\Psi,u) = \frac{1}{2} \Abs{\dop_{A}\!u}_{J}^{2} + \lam^{2} \cdot \Abs{\mu(u)}_{\gfr}^{2}.
\end{eqnarray}
Here the norm $\Abs{\dop_{A}\!u}_{J}$ is understood with respect to the metric $\l\cdot,\cdot\r_{J}$ on $\X$ and the Euclidean metric on $B$. Note that hypothesis (R1) above means that the connection form $A$ extends continuously to all of $B$.

We shall prove that the map $u$ is of class $W^{1,p}$ on $B$ for every $p>2$. We will proceed as follows. In Section \ref{SubSectionTheGraphConstruction}, we apply the graph construction in order to transform the vortex~$(A,u)$ into a punctured pseudoholomorphic section of the trivial fiber bundle $B \xop \X$ over $B$. In Section~\ref{SubSectionMeanValueInequality}, we obtain a mean value inequality for this section from the a priori estimate of Section~\ref{SectionAPrioriEstimate}. The actual proof of Theorem~\ref{TheoremRemovalOfSingularities} will then be given in Section~\ref{SubSectionProofOfTheoremRemoval}.

\subsection{The graph construction}
\label{SubSectionTheGraphConstruction}
	Let $\Xtilde \deq B \xop \X$ denote the total space of the trivial symplectic fiber bundle over $B$ with fiber the manifold $\X$. The map $\map{u}{B \setminus \{0\}}{\X}$ then gives rise to a section
\[
\map{\utilde}{B \setminus \{0\}}{\Xtilde}, \quad \utilde(z) \deq \bigl( z,u(z) \bigr),
\]
and the almost complex structure $J$ induces an almost complex structure $\Jtilde$ on $\Xtilde$ by
\begin{eqnarray} \label{EquationAlmostComplexStructureOnE}
	\Jtilde(v,w) \deq \bigl( i \, v,J w + J \, X_{A(v)}(x) - X_{A(i \, v)}(x) \bigr)
\end{eqnarray}
for all $(z,x)\in B \xop \X$ and $v\in T_{z}B$, $w\in T_{x}\X$. Here we use the identifications $T_{(z,x)} \Xtilde \cong T_{z}B \oplus T_{x}\X$ and $T_{z}B \cong \C$. In fact, a straightforward computation shows that $\Jtilde^{2} (v,w) = - (v,w)$. Note that the almost complex structure $\Jtilde$ will in general only be continuous, as follows from~(\ref{EquationAlmostComplexStructureOnE}) since~$A$ is only assumed to be continuous on $B$ by hypothesis (R1).

\begin{lemma} \label{LemmautildeIsJ_AHolomorphic}
	The section $\map{\utilde}{B \setminus \{0\}}{\Xtilde}$ is $(i,\Jtilde)$-holomorphic.
\end{lemma}

\begin{proof}
	The differential of $\utilde$ is given by $\dop\!\utilde(v) = ( v,\dop\!u(v) )$ for $v \in TB$. By the first vortex equation (\ref{EquationLocalVortexEquationsRemovalOfSingularities}), we have $J \, \dop_{A}\!u(v) = \dop_{A}\!u ( i\,v )$. Hence using formula (\ref{EquationTwistedDerivative}) we get
\[
\begin{split}
	\Jtilde \bigl( \dop\!\utilde(v) \bigr) &= \Jtilde \bigl( v,\dop\!u(v) \bigr) = \bigl( i\,v,J \dop\!u(v) + J X_{A(v)}(u) - X_{A(i\,v)}(u) \bigr) \\
	&= \bigl( i\,v,J \dop_{A}\!u(v) - X_{A(i\,v)}(u) \bigr) = \bigl( i\,v,\dop_{A}\!u(i\,v) - X_{A(i\,v)}(u) \bigr) \\
	&= \bigl( i\,v,\dop\!u(i\,v) \bigr) = \dop\!\utilde(i\,v).
\end{split}
\]
This implies that $\delbar_{\Jtilde}(\utilde) = \frac{1}{2} \bigl( \dop\!\utilde + \Jtilde(\utilde) \circ \dop\!\utilde \circ i \bigr) = 0.$
\end{proof}

Next we define a symplectic form $\omtilde$ on the manifold $\Xtilde$ that tames the almost complex structure $\Jtilde$. By compactness of $B$ and $\X$ we may fix a constant $c_{A}>1$ such that
\begin{eqnarray} \label{InequalityConstantForSymplecticFormOnMtilde}
	\Abs{X_{A(v)}(x)}_{J} \le \frac{1}{5} \, c_{A} \cdot \Abs{v}
\end{eqnarray}
for all tangent vectors $v\in TB$ and all points $x\in \X$, where $\abs{\,\cdot\,}_{J}$ and $\abs{\,\cdot\,}$ denote the norms associated to the metric $\l\cdot,\cdot\r_{J}$ on $\X$ and the Euclidean metric on $B$, respectively. We then define
\[
\omtilde \deq c_{A}^{2} \cdot \om_{0} \,\oplus\,\, \om,
\]
where $\om_{0} \deq \dop\!s \wedge \dop\!t$ denotes the standard symplectic form on $B$.

\begin{lemma} \label{LemmaTamingOnE}
	The symplectic form $\omtilde$ tames the almost complex structure $\Jtilde$.
\end{lemma}

\begin{proof}
	Let $(v,w)\in T\Xtilde$ be such that $(v,w) \neq (0,0)$. Using formula (\ref{EquationAlmostComplexStructureOnE}) and the definition of the metric $\l\cdot,\cdot\r_{J}$, we get
\begin{multline*}
	\omtilde\bigl( (v,w),\Jtilde(v,w) \bigr) = c_{A}^{2} \cdot \Abs{v}^{2} + \Abs{w + X_{A(v)}}_{J}^{2} - \bigl\l X_{A(v)},w + X_{A(v)} \bigr\r_{J} \\
	- \bigl\l J(w + X_{A(v)}),X_{A(i\,v)} \bigr\r_{J} + \bigl\l J X_{A(v)},X_{A(i\,v)} \bigr\r_{J}.	
\end{multline*}
Applying the inequalities of Cauchy-Schwarz and Young and using $J$-invariance of the norm $\Abs{\,\cdot\,}_{J}$ we may further estimate this from below by

\begin{eqnarray*}
	&& c_{A}^{2} \cdot \Abs{v}^{2} + \Abs{w + X_{A(v)}}_{J}^{2} - \Abs{X_{A(v)}}_{J} \cdot \Abs{w + X_{A(v)}}_{J} \\
	&& - \Abs{J(w + X_{A(v)})}_{J} \cdot \Abs{X_{A(i\,v)}}_{J} - \Abs{J X_{A(v)}}_{J} \cdot \Abs{X_{A(i\,v)}}_{J} \\
	&\ge& c_{A}^{2} \cdot \Abs{v}^{2} + \Abs{w + X_{A(v)}}_{J}^{2} - 4 \Abs{X_{A(v)}}_{J}^{2} - \frac{1}{4} \cdot \Abs{w + X_{A(v)}}_{J}^{2} \\
	&& - \frac{1}{4} \cdot \Abs{w + X_{A(v)}}_{J}^{2} - 4 \Abs{X_{A(i\,v)}}_{J}^{2} - \Abs{X_{A(v)}}_{J}^{2} - \Abs{X_{A(i\,v)}}_{J}^{2} \\
	&\ge& c_{A}^{2} \cdot \Abs{v}^{2} + \frac{1}{2} \Abs{w + X_{A(v)}}_{J}^{2} - 5 \Abs{X_{A(v)}}^{2}_{J} - 5 \Abs{X_{A(i\,v)}}_{J}^{2}.
\end{eqnarray*}
By inequality~(\ref{InequalityConstantForSymplecticFormOnMtilde}) this is not smaller than
\[
c_{A}^{2} \cdot \Abs{v}^{2} + \frac{1}{2} \Abs{w + X_{A(v)}}_{J}^{2} - \frac{1}{5} \, c_{A}^{2} \cdot \Abs{v}^{2} - \frac{1}{5} \, c_{A}^{2} \cdot \Abs{i\,v}^{2} \ge \frac{1}{2} \Bigl( c_{A}^{2} \cdot \Abs{v}^{2} + \Abs{w + X_{A(v)}}_{J}^{2} \Bigr) > 0,
\]
which proves the lemma.
\end{proof}

By Lemma \ref{LemmaTamingOnE}, the symplectic form $\omtilde$ and the almost complex structure~$\Jtilde$ determine a Riemannian metric $\l\cdot,\cdot\r_{\Jtilde}$ on $\Xtilde$ given by
\[
\bigl\l (v_{1},w_{1}),(v_{2},w_{2}) \bigr\r_{\Jtilde} \deq \frac{1}{2} \Bigl( \omtilde\bigl( (v_{1},w_{1}),\Jtilde(v_{2},w_{2}) \bigr) - \omtilde\bigl( \Jtilde(v_{1},w_{1}),(v_{2},w_{2}) \bigr) \Bigr)
\]
for all $(v_{1},w_{1}),(v_{2},w_{2}) \in T\Xtilde$. We will denote by $\Abs{\,\cdot\,}_{\Jtilde}$ the corresponding norm on $T \Xtilde$. Note that this norm will in general only be continuous, since~$\Jtilde$ has this property.

\begin{lemma} \label{LemmaComparisonOfMetricsOnE}
	The norm $\Abs{\,\cdot\,}_{\Jtilde}$ satisfies the inequalities
\[
\frac{1}{2} \, \Bigl( \Abs{v}^{2} + \Abs{w + X_{A(v)}}^{2}_{J} \Bigr) \le \Abs{(v,w)}_{\Jtilde}^{2} \le 3\,c_{A}^{2} \cdot \Bigl( \Abs{v}^{2} + \Abs{w + X_{A(v)}}^{2}_{J} \Bigr)
\]
for $(v,w)\in T\Xtilde$, where $c_{A}$ is the constant from inequality (\ref{InequalityConstantForSymplecticFormOnMtilde}).
\end{lemma}

\begin{proof}
	Recall that $c_{A}>1$. The computation in the proof of Lemma~\ref{LemmaTamingOnE} above then shows that
\[
\Abs{(v,w)}_{\Jtilde}^{2} = \omtilde\bigl( (v,w),\Jtilde(v,w) \bigr) \ge \frac{1}{2} \, \Bigl( \Abs{v}^{2} + \Abs{w + X_{A(v)}}^{2}_{J} \Bigr),
\]
which proves the first inequality. The second inequality follows in a similar way.
\end{proof}

\subsection{Mean value inequality}
\label{SubSectionMeanValueInequality}

We derive a mean value inequality for the $\Jtilde$-holomorphic section $\map{\utilde}{B\setminus\{0\}}{\Xtilde}$ from the a priori estimate for the vortex $(\Phi,\Psi,u)$ provided by Theorem~\ref{TheoremAPrioriEstimate}. Note that the mean value inequality from \cite[Lemma 4.3.1]{McDuff/J-holomorphic-curves-and-symplectic-topology} does not apply to the section $\utilde$ since the almost complex structure $\Jtilde$ will in general only be continuous.

\smallskip

To begin with, we recall from \cite[Sec.\,2.2]{McDuff/J-holomorphic-curves-and-symplectic-topology} that the energy of the section~$\utilde$ on an open subset~$U \subset B$ is given by
\[
E(\utilde;U) \deq \frac{1}{2} \int_{U} \Abs{\dop\!\utilde}_{\Jtilde}^{2},
\]
where the norm $\Abs{\dop\!\utilde}_{\Jtilde}$ is understood with respect to the metric $\l\cdot,\cdot\r_{\Jtilde}$ on $\Xtilde$ and the Euclidean metric on $B \subset \C$.

\begin{lemma} \label{LemmaMeanValueInequalityRemovalOfSingularities}
	The section $\map{\utilde}{B\setminus\{0\}}{\Xtilde}$ has finite energy $E(\utilde;B) < \infty$. Moreover, there exist constants $\de, C_{A}, r_{0}>0$ such that the following holds. For all $w_{0} \in B\setminus\{0\}$ and all $0 < r < r_{0}$ such that $B_{r}(w_{0}) \subset B\setminus\{0\}$, the section $\utilde$ satisfies the mean value inequality
\begin{eqnarray*} \label{APrioriEstimateForSectionRemovalOfSingualrities}
	E\bigl( \utilde;B_{r}(w_{0}) \bigr) < \de \quad \Longrightarrow \quad \Abs{\dop\!\utilde(w_{0})}^{2}_{\Jtilde} \le \frac{C_{A}}{r^{2}} \cdot E\bigl( \utilde;B_{r}(w_{0}) \bigr) + C_{A}.
\end{eqnarray*}
\end{lemma}

\begin{proof}
	For every $v \in TB$, using $\dop\!\utilde(v) = ( v,\dop\!u(v) )$ and formula (\ref{EquationTwistedDerivative}), we obtain from Lemma~\ref{LemmaComparisonOfMetricsOnE} the inequalities
\[
\frac{1}{2} \Bigl( \Abs{v}^{2} + \Abs{\dop_{A}\!u(v)}^{2}_{J} \Bigr) \le \Abs{\dop\!\utilde(v)}_{\Jtilde}^{2} \le 3\,c_{A}^{2} \cdot \Bigl( \Abs{v}^{2} + \Abs{\dop_{A}\!u(v)}^{2}_{J} \Bigr).
\]
This implies that
\begin{eqnarray} \label{InequalityComparisonDirichletDensities}
	\frac{1}{2} \Bigl( 2 + \Abs{\dop_{A}\!u}^{2}_{J} \Bigr) \le \Abs{\dop\!\utilde}_{\Jtilde}^{2} \le 3\,c_{A}^{2} \cdot \Bigl( 2 + \Abs{\dop_{A}\!u}^{2}_{J} \Bigr).
\end{eqnarray}
By formula (\ref{EqualityLocalYMHEnergyOfVortex}) we therefore obtain
\begin{eqnarray*}
	E(\utilde;B) &=& \frac{1}{2} \int_{B} \Abs{\dop\!\utilde}_{\Jtilde}^{2}\,\,\le\,\, \frac{3}{2} \, c_{A}^{2} \cdot \int_{B} \Bigl( 2 + \Abs{\dop_{A}\!u}^{2}_{J} \Bigr) \\
	&=& 3 \, c_{A}^{2} \cdot \int_{B} \left( \frac{1}{2} \Abs{\dop_{A}\!u}_{J}^{2} + \lam^{2} \cdot \Abs{\mu(u)}^{2} \right) + 3\pi \, c_{A}^{2} - 3 \, c_{A}^{2} \cdot \int_{B} \left( \lam^{2} \cdot \Abs{\mu(u)}^{2} \right) \\
	&\le& 3\,c_{A}^{2} \cdot E(\Phi,\Psi,u;B) + 3\pi \, c_{A}^{2}.
\end{eqnarray*}
The first term on the right-hand side of this inequality is finite by hypothesis (R2). Hence we have $E(\utilde;B) < \infty$, which proves the first assertion of the lemma.

By Theorem~\ref{TheoremAPrioriEstimate}, there exist constants $\de^{\prime},C>0$ such that the following holds. For all $w_{0}\in B\setminus\{0\}$ and $r>0$ such that $B_{r}(w_{0}) \subset B\setminus\{0\}$ the vortex $(\Phi,\Psi,u)$ satisfies the a priori estimate
\begin{equation} \label{InequalityAPrioriEstimateProofOfMeanValueInequality}
	E\bigl( \Phi,\Psi,u;B_{r}(w_{0}) \bigr) < \de^{\prime} \quad \Longrightarrow \quad e(\Phi,\Psi,u)(w_{0}) \le \frac{C}{r^{2}} \cdot E\bigl( \Phi,\Psi,u;B_{r}(w_{0}) \bigr).
\end{equation}
Define constants
\[
K \deq \pi \cdot \norm{\lam}_{C^{0}(B)}^{2} \cdot \norm{\mu}_{C^{0}(\X)}^{2}
\]
and
\[
\de \deq \frac{\de^{\prime}}{4}, \quad C_{A} \deq 12\,c_{A}^{2} \, \bigl( C (K + 1) + 1 \bigr), \quad r_{0} \deq \min \left\{ \sqrt{\frac{\de^{\prime}}{2K}},1 \right\},
\]
where $c_{A}$ is the constant from inequality (\ref{InequalityConstantForSymplecticFormOnMtilde}). Assume now that
\begin{eqnarray} \label{InequalityAssumptionBoundOnDirichletEnergy}
	r < r_{0} \quad \text{and} \quad E\bigl( \utilde;B_{r}(w_{0}) \bigr) < \de.
\end{eqnarray}
Using the first inequality in (\ref{InequalityComparisonDirichletDensities}) and formula (\ref{EqualityLocalYMHEnergyOfVortex}) we then obtain
\begin{eqnarray*}
	E\bigl( \utilde;B_{r}(w_{0}) \bigr) &=& \frac{1}{2} \int_{B_{r}(w_{0})} \Abs{\dop\!\utilde}_{\Jtilde}^{2} \,\,\ge\,\, \frac{1}{4} \int_{B_{r}(w_{0})} \Bigl( 2 + \Abs{\dop_{A}\!u}^{2}_{J} \Bigr) \\
	&=& \frac{1}{2} \int_{B_{r}(w_{0})} \left( \frac{1}{2} \Abs{\dop_{A}\!u}_{J}^{2} + \lam^{2} \cdot \Abs{\mu(u)}^{2} \right) + \frac{\pi r^{2}}{2} - \frac{1}{2} \int_{B_{r}(w_{0})} \lam^{2} \cdot \Abs{\mu(u)}^{2} \\
	&\ge& \frac{1}{2} \, E\bigl( \Phi,\Psi,u;B_{r}(w_{0}) \bigr) - \frac{1}{2} \, K \, r^{2},
\end{eqnarray*}
whence
\begin{eqnarray} \label{InequalityComparisonYMHAndDirichletEnergy}
	E\bigl( \Phi,\Psi,u;B_{r}(w_{0}) \bigr) \le 2 \, E\bigl( \utilde;B_{r}(w_{0}) \bigr) + K \, r^{2}.
\end{eqnarray}
By assumption (\ref{InequalityAssumptionBoundOnDirichletEnergy}) and the definition of $r_{0}$ above, it follows from this that
\begin{eqnarray*}
	E\bigl( \Phi,\Psi,u;B_{r}(w_{0}) \bigr) < \frac{\de^{\prime}}{2} + K \, r^{2} < \de^{\prime}.
\end{eqnarray*}
Thus the a priori estimate (\ref{InequalityAPrioriEstimateProofOfMeanValueInequality}) implies that
\begin{eqnarray*} \label{InequalityYMHEnergyDensityRemovalOfSingularities}
	e(\Phi,\Psi,u)(w_{0}) \le \frac{C}{r^{2}} \cdot E\bigl( \Phi,\Psi,u;B_{r}(w_{0}) \bigr).
\end{eqnarray*}
Hence, using the second inequality in (\ref{InequalityComparisonDirichletDensities}) and formula (\ref{EqualityLocalYMHEnergyOfVortex}), we further obtain
\begin{eqnarray*}
	\Abs{\dop\!\utilde(w_{0})}_{\Jtilde}^{2} &\le& 3\,c_{A}^{2} \cdot \Bigl( 2 + \Abs{\dop_{A}\!u(w_{0})}^{2}_{J} \Bigr) \\
	&\le& 6\,c_{A}^{2} \cdot \Bigl( \frac{1}{2} \Abs{\dop_{A}\!u(w_{0})}^{2}_{J} + \lam^{2} \cdot \Abs{\mu\bigl( u(w_{0}) \bigr)}^{2} \Bigr) + 6\,c_{A}^{2} \\
	&=& 6\,c_{A}^{2} \cdot e(\Phi,\Psi,u)(w_{0}) + 6\,c_{A}^{2} \\
	&\le& \frac{6\,c_{A}^{2}\,C}{r^{2}} \cdot E\bigl( \Phi,\Psi,u;B_{r}(w_{0}) \bigr) + 6\,c_{A}^{2}.
\end{eqnarray*}
Applying inequality~(\ref{InequalityComparisonYMHAndDirichletEnergy}) again and using $c_{A}>1$, we finally have
\[
\Abs{\dop\!\utilde(w_{0})}_{\Jtilde}^{2} \le \frac{12\,c_{A}^{2}\,C}{r^{2}} \cdot E\bigl( \utilde;B_{r}(w_{0}) \bigr) + 6\,c_{A}^{2} \cdot \bigl( C K + 1 \bigr) \le \frac{C_{A}}{r^{2}} \cdot E\bigl( \utilde;B_{r}(w_{0}) \bigr) + C_{A}.
\]
This proves Lemma \ref{LemmaMeanValueInequalityRemovalOfSingularities}.
\end{proof}

\subsection{Proof of Theorem~\ref{TheoremRemovalOfSingularities}}
\label{SubSectionProofOfTheoremRemoval}

Our proof is adapted from the proof of \cite[Thm.\,4.1.2]{McDuff/J-holomorphic-curves-and-symplectic-topology}. As a first step, we shall apply the isoperimetric inequality from \cite[Thm.\,4.4.1]{McDuff/J-holomorphic-curves-and-symplectic-topology} to estimate the energy of the section $\map{\utilde}{B\setminus\{0\}}{\Xtilde}$ on small neighborhoods around the puncture.

\smallskip

We begin by recalling some notation from \cite[Sec.\,4.4]{McDuff/J-holomorphic-curves-and-symplectic-topology}. For any smooth loop $\map{\ga}{\del\!B}{\Xtilde}$ we denote by $\ell(\ga)$ its length with respect to the metric $\l\cdot,\cdot\r_{\Jtilde}$. If $\ell(\ga)$ is smaller than the injectivity radius of $\Xtilde$, then $\ga$ admits a smooth local extension $\map{u_{\ga}}{B}{\Xtilde}$ such that $u_{\ga}(e^{i\th}) = \ga(\th)$ for all $\th\in[0,2\pi]$ and the image of $u_{\ga}$ is contained in a geodesic ball of radius not greater than half the injectivity radius. The local symplectic action of $\ga$ is then defined as
\[
a(\ga) \deq - \int_{B} u_{\ga}^{\ast}\,\omtilde.
\]
Note that it does not depend on the choice of the extension $u_{\ga}$. Since $\omtilde$ tames $\Jtilde$ by Lemma~\ref{LemmaTamingOnE}, the isoperimetric inequality from \cite[Thm.\,4.4.1]{McDuff/J-holomorphic-curves-and-symplectic-topology} applies to $\Xtilde$; in fact, a careful analysis of the proof of said theorem reveals that the isoperimetric inequality holds in the present situation even though the almost complex structure $\Jtilde$ will in general only be continuous. Thus we have:

\begin{lemma}[{McDuff and Salamon \cite[Thm.\,4.4.1]{McDuff/J-holomorphic-curves-and-symplectic-topology}}] \label{LemmaIsoperimetricInequality}
	For every constant $c > 1/4\pi$ there exists a constant $\ell_{0}>0$ such that
\[
\ell(\ga) < \ell_{0} \quad \Longrightarrow \quad \Abs{a(\ga)} \le c \cdot \ell(\ga)^{2}
\]
for every smooth loop $\map{\ga}{\del\!B}{\Xtilde}$.
\end{lemma}

We may now prove a variant of \cite[Lemma 4.5.1]{McDuff/J-holomorphic-curves-and-symplectic-topology}. For that purpose, we define a function $\map{\veps}{(0,1]}{\R}$ by
\begin{eqnarray} \label{MapEnergyEpsilon}
	\veps(r) \deq E\bigl( \utilde;B_{r}(0) \bigr) = \frac{1}{2} \int_{0}^{r} \rho \int_{0}^{2\pi} \Abs{\dop\!\utilde \bigl( \rho e^{i\th} \bigr)}_{\Jtilde}^{2} \dop\!\th\dop\!\rho
\end{eqnarray}
that assigns to every $0 < r \le 1$ the energy of the curve $\map{\utilde}{B\setminus\{0\}}{\Xtilde}$ on the punctured disk~$B_{r}(0) \setminus \{0\}$. We see from formula (\ref{MapEnergyEpsilon}) that the function $\veps$ is of class $C^{1}$. Let $\map{\ga_{r}}{\del\!B}{\Xtilde}$ denote the loop defined by $\ga_{r}(\th) \deq \utilde\bigl( r e^{i\th} \bigr)$ for $\th\in[0,2\pi]$.

\begin{lemma} \label{LemmaSpecialIsoperimetricInequality}
	For every constant $c > 1/4\pi$ there exists a constant $r_{1}>0$ such that
\[
0 < r < r_{1} \quad \Longrightarrow \quad \veps(r) \le c \cdot \ell(\ga_{r})^{2}.
\]
\end{lemma}

\begin{proof}
	Our proof is adapted from the proof of \cite[Lemma 4.5.1]{McDuff/J-holomorphic-curves-and-symplectic-topology}. Let $c > 1/4\pi$, let $\ell_{0}$ be the constant from Lemma~\ref{LemmaIsoperimetricInequality}, and let $\de$, $C_{A}$ and $r_{0}$ be the constants from Lemma~\ref{LemmaMeanValueInequalityRemovalOfSingularities}. Fix a constant~$r_{1}>0$ such that
\begin{eqnarray} \label{InequalityEnergyBoundForr_1RemovalOfSingularities}
	r_{1} < \min\left\{ r_{0},\frac{1}{2} \right\} \quad \text{and} \quad \veps(2r_{1}) < \min\left\{ \de,\frac{\ell_{0}^{2} - 4\pi^{2}\,C_{A}\,r_{1}^{2}}{8\pi^{2}\,C_{A}} \right\}.
\end{eqnarray}
Such $r_{1}$ exists since $\veps(1)=E(\utilde;B) < \infty$ by Lemma~\ref{LemmaMeanValueInequalityRemovalOfSingularities} and the function $\veps$ is nonnegative and nondecreasing with $\lim_{r\to 0}\veps(r) = 0$.

Let now $0 < r < r_{1}$. Then $E(\utilde;B_{r/2}(r e^{i\th})) \le E(\utilde;B_{2r}(0)) = \veps(2r) < \de$ by the second inequality in (\ref{InequalityEnergyBoundForr_1RemovalOfSingularities}). Hence the mean value inequality of Lemma~\ref{LemmaMeanValueInequalityRemovalOfSingularities}, applied to the disk \mbox{$B_{r/2}(r e^{i\th}) \subset B \setminus \{0\}$}, yields
\[
\Abs{\dop\!\utilde\bigl( re^{i\th} \bigr)}_{\Jtilde}^{2} \le \frac{4 \, C_{A}}{r^{2}} \cdot E\bigl( \utilde;B_{r/2}(r e^{i\th}) \bigr) + C_{A} \le \frac{4 \, C_{A}}{r^{2}} \cdot \veps(2r) + C_{A}.
\]
It follows that the derivative of $\ga_{r}$ in the direction of $\th$ satisfies an estimate
\[
\Abs{\dot{\ga_{r}}(\th)}_{\Jtilde} = \frac{r}{\sqrt{2}} \cdot \Abs{\dop\!\utilde\bigl( re^{i\th} \bigr)}_{\Jtilde} \le \sqrt{2 \, C_{A} \cdot \veps(2r) + C_{A} \cdot r^{2}}.
\]
By the second inequality in (\ref{InequalityEnergyBoundForr_1RemovalOfSingularities}) this implies that
\begin{eqnarray} \label{InequalityEstimateForLengthOfLoopRemovalOfSingularities}
	\ell(\ga_{r}) = \int_{0}^{2\pi} \Abs{\dot{\ga_{r}}(\th)} \dop\!\th \le \sqrt{8\pi^{2}\,C_{A} \cdot \veps(2r) + 4\pi^{2}\,C_{A} \cdot r^{2}} < \ell_{0}.
\end{eqnarray}
We now proceed exactly as in the proof of \cite[Lemma 4.5.1]{McDuff/J-holomorphic-curves-and-symplectic-topology}, obtaining $\veps(r) = -a(\ga_{r})$. By (\ref{InequalityEstimateForLengthOfLoopRemovalOfSingularities}) the isoperimetric inequality of Lemma~\ref{LemmaIsoperimetricInequality} applies, so we finally arrive at $\veps(r) \le c \cdot \ell(\ga_{r})^{2}$.
\end{proof}

We are now ready for the actual proof of Theorem~\ref{TheoremRemovalOfSingularities}.

\begin{proof}[Proof of Theorem~\ref{TheoremRemovalOfSingularities}]
	Our proof follows the proof of \cite[Thm.\,4.1.2]{McDuff/J-holomorphic-curves-and-symplectic-topology}. Applying the isoperimetric inequality of Lemma~\ref{LemmaSpecialIsoperimetricInequality} we conclude as in said proof that there exist constants $c > 1/4\pi$ and $c_{1}>0$ such that, for $r>0$ sufficiently small, the function (\ref{MapEnergyEpsilon}) satisfies an inequality
\[
\veps(r) \le c_{1} \cdot r^{2\al},
\]
where $\al \deq 1/4\pi c < 1$. In fact, this argument only requires the function $\veps$ to be of class $C^{1}$. For~$r>0$ sufficiently small, combining this with the mean value inequality of Lemma~\ref{LemmaMeanValueInequalityRemovalOfSingularities} applied to the disk $B_{r/2}(r e^{i\th}) \subset B \setminus \{0\}$, we hence obtain
\begin{eqnarray} \label{InequalityPointwiseEstimateForNormSquareRemovalOfSingularities}
	\Abs{\dop\!\utilde\bigl( r e^{i\th} \bigr)}_{\Jtilde}^{2} \,\le\, \frac{4\,C_{A}}{r^{2}} \cdot \veps(2r) + C_{A} \le c_{2} \cdot r^{-2(1-\al)} + C_{A},
\end{eqnarray}
where $c_{2}>0$ is some constant not depending on $r$ and $\th$. We now proceed exactly as in the proof of \cite[Thm.\,4.1.2]{McDuff/J-holomorphic-curves-and-symplectic-topology}, replacing inequality (4.5.2) in that proof by inequality~(\ref{InequalityPointwiseEstimateForNormSquareRemovalOfSingularities}). Note that we have to keep track of the additive constant $C_{A}$ on the right-hand side of inequality~(\ref{InequalityPointwiseEstimateForNormSquareRemovalOfSingularities}).
\end{proof}

\section{Convergence modulo bubbling}
\label{SectionConvergenceModuloBubbling}

The purpose of this section is to prove Theorem~\ref{TheoremConvergenceModuloBubbling} below, which establishes a compactness result for vortices, ignoring any bubbling phenomena. This theorem constitutes the first part of the proof of Theorem~\ref{TheoremGromovCompactness}.

\smallskip

Similar compactness results were proved by Mundet i Riera \cite{Mundet-i-Riera/Hamiltonian-Gromov-Witten-invariants} in the case of $G = S^{1}$, using a different approach that relies on the compactness results for pseudoholomorphic curves due to Ivashkovich and Shevchishin \cite{Ivashkovich/Gromov-compactness-theorem-for-J-complex-curves-with-boundary}, and by \cite{Cieliebak/The-symplectic-vortex-equations-and-invariants-of-Hamiltonian-group-actions} for arbitrary compact Lie groups~$G$ under the assumption that $\X$ is symplectically aspherical. We shall now prove a generalization of these results that holds for arbitrary compact Lie groups $G$ and arbitrary closed Hamiltonian $G$-manifolds $\X$. Our strategy is to combine the above-mentioned approach of \cite{Cieliebak/The-symplectic-vortex-equations-and-invariants-of-Hamiltonian-group-actions} with the methods that were applied by McDuff and Salamon in proving a similar compactness result for pseudoholomorphic curves, see \cite[Thm.\,4.6.1]{McDuff/J-holomorphic-curves-and-symplectic-topology}. The proof crucially relies on removal of singularities for vortices provided by Theorem \ref{TheoremRemovalOfSingularities}.

We keep the notation introduced in Section \ref{SectionIntroductionAndMainResults}. Moreover, for $p>2$ we denote by $\A^{1,p}(P)$ the space of connections on the bundle $P$ of Sobolev class $W^{1,p}$, and by $W^{1,p}(P,\X)^{G}$ the space of $G$-equivariant maps \mbox{$\map{u}{P}{\X}$} of class $W^{1,p}$. Note that the vortex equations (\ref{EquationVortexEquations}) and the Yang-Mills-Higgs energy (\ref{EquationYMHEnergyGlobal}) are well-defined for pairs $(A,u) \in \A^{1,p}(P) \xop W^{1,p}(P,\X)^{G}$. The action~(\ref{EquationActionOfGaugeTransformation}) of the group of smooth gauge transformations of $P$ then naturally extends to an action of the group $\G^{2,p}(P) \deq W^{2,p}(P,G)^{G}$ of gauge transformations of $P$ of class $W^{2,p}$ on the space $\A^{1,p}(P) \xop W^{1,p}(P,\X)^{G}$ (see \cite[App.\,A]{Wehrheim/Uhlenbeck-compactness} for details on this). The vortex equations (\ref{EquationVortexEquations}) and the Yang-Mills-Higgs energy (\ref{EquationYMHEnergyGlobal}) remain invariant under this action.

\begin{theorem}[Convergence modulo bubbling] \label{TheoremConvergenceModuloBubbling}
	Let $(A_{\nu},u_{\nu})$ be a sequence of vortices whose Yang-Mills-Higgs energy satisfies a uniform bound
\[
\sup_{\nu} E\bigl( A_{\nu},u_{\nu} \bigr) < \infty.
\]
Then there exist a smooth vortex $(A,u)$, a sequence of smooth gauge transformations $g_{\nu} \in \G(P)$, a real number $p>2$, and a finite set $Z = \{z_{1},\ldots,z_{N}\}$ of distinct points on $\Sig$ such that, after passing to a subsequence, the following holds.
\begin{enumerate}[itemsep=0.3ex, itemindent=0cm, labelsep=1ex, leftmargin=1cm]
	\item The sequence $g^{\ast}_{\nu} A_{\nu}$ converges to $A$ weakly in $W^{1,p}$ and strongly in $C^{0}$ on $\Sig$;
	\item the sequence $(g^{\ast}_{\nu} A_{\nu},g^{-1}_{\nu} u_{\nu})$ converges to $(A,u)$ in $C^{\infty}$ on compact subsets of $\Sig \setminus Z$;
	\item for every $j\in\{1,\ldots,N\}$ and every $\veps>0$ such that $B_{\veps}(z_{j}) \cap Z = \{z_{j}\}$, the limit
\[
m_{\veps}(z_{j}) \deq \lim_{\nu\to\infty} E\bigl( g_{\nu}^{\ast} A_{\nu},g_{\nu}^{-1} u_{\nu};B_{\veps}(z_{j}) \bigr)
\]
exists and is a continuous function of $\veps$, and
\[
m(z_{j}) \deq \lim_{\veps\to 0}m_{\veps}(z_{j}) \ge \hbar,
\]
where $\hbar$ is the constant of Corollary \ref{CorollaryAPrioriEstimateOnSurface};
	\item for every compact subset $K \subset \Sig$ such that $Z$ is contained in the interior of $K$,
\[
E\bigl( A,u;K \bigr) + \sum_{j=1}^{N} m(z_{j}) = \lim_{\nu\to\infty} E\bigl( g_{\nu}^{\ast} A_{\nu},g_{\nu}^{-1} u_{\nu};K \bigr).
\]
\end{enumerate}
\end{theorem}

The proof of Theorem~\ref{TheoremConvergenceModuloBubbling} will occupy the remainder of this section. It is much inspired by the proofs of \cite[Thm.\,4.6.1]{McDuff/J-holomorphic-curves-and-symplectic-topology} and \cite[Thm.\,3.2]{Cieliebak/The-symplectic-vortex-equations-and-invariants-of-Hamiltonian-group-actions}. We shall proceed in several steps. First, in Section \ref{SubSectionSingularPoints} we prove that bubbling may occur at only finitely many points. We then apply weak Uhlenbeck compactness and a local slice theorem for the action of the group of gauge transformations in order to construct a limit connection, in Section~\ref{SubSectionUhlenbeckCompactnessAndCoulombGauge}. Section \ref{SubSectionTheLimitEquations} is concerned with the study of the limit vortex equations on the complement of the bubbling points. Next, in Section \ref{SubSectionRemovalOfSingularities} we apply removal of singularities in order to obtain the limit vortex. Finally, in Section \ref{SubSectionProofOfTheoremConvergenceModuloBubbling} we combine the previous results in order to prove Theorem~\ref{TheoremConvergenceModuloBubbling}. 

\smallskip

Throughout this section, let $(A_{\nu},u_{\nu})$ be a sequence of vortices whose Yang-Mills-Higgs energy satisfies a uniform bound
\[
\sup_{\nu} E\bigl( A_{\nu},u_{\nu} \bigr) < \infty.
\]

\subsection{Singular points}
\label{SubSectionSingularPoints}

Following the terminology in \cite[Sec.\,4.6]{McDuff/J-holomorphic-curves-and-symplectic-topology} a point $z \in \Sig$ is called \emph{singular} for the sequence $(A_{\nu},u_{\nu})$ if there exists a sequence $z^{\nu}$ of points in $\Sig$ converging to $z$ such that $\Abs{\dop_{A_{\nu}}\!u_{\nu}(z^{\nu})}_{J} \to \infty$. The main result of this subsection is that the sequence $(A_{\nu},u_{\nu})$ can have only finitely many singular points. Basically, this is an immediate consequence of quantization of energy for pseudoholomorphic spheres \cite[Thm.\,3.4]{Cieliebak/The-symplectic-vortex-equations-and-invariants-of-Hamiltonian-group-actions}. However, we prefer to give a much shorter alternative proof of this fact using an indirect argument due to Wehrheim \cite{Wehrheim/Energy-quantization-and-mean-value-inequalities-for-nonlinear-boundary-value-problems} that is based on the a priori estimate of Corollary~\ref{CorollaryAPrioriEstimateOnSurface} and avoids an explicit construction of bubbles.

\begin{lemma} \label{LemmaLowerEnergyBoundForBubbles}
	Let $z$ be a singular point of the sequence $(A_{\nu},u_{\nu})$. Then
\[
\liminf_{\nu\to\infty} E\bigl( A_{\nu},u_{\nu};B_{\veps}(z) \bigr) \ge \hbar
\]
for every $0 < \veps < R$, where $\hbar$ and $R$ are the constants from Corollary \ref{CorollaryAPrioriEstimateOnSurface}.
\end{lemma}

\begin{proof}
	Our proof is adapted from the proof of \cite[Thm.\,2.1]{Wehrheim/Energy-quantization-and-mean-value-inequalities-for-nonlinear-boundary-value-problems}. Let $\hbar$, $R$ and $C$ be the constants from Corollary \ref{CorollaryAPrioriEstimateOnSurface}. Let $z$ be a singular point of the sequence $(A_{\nu},u_{\nu})$, and assume for contradiction that
\[
\liminf_{\nu\to\infty} E\bigl( A_{\nu},u_{\nu};B_{\veps}(z) \bigr) < \hbar
\]
for some $0 < \veps < R$. Since $z$ is singular there exists a sequence $z^{\nu}$ converging to $z$ such that $\abs{\dop_{A_{\nu}}\!u_{\nu}(z^{\nu})}_{J} \to \infty$. Hence there exists $\nu_{0}$ such that
\begin{equation} \label{InequalityBadPoint}
	z^{\nu_{0}}\in B_{\veps/2}(z), \quad E\bigl( A_{\nu},u_{\nu};B_{\veps/2}(z^{\nu_{0}}) \bigr) < \hbar, \quad \Abs{\dop_{A_{\nu}}\!u_{\nu}(z^{\nu_{0}})}_{J} > \frac{8\,C \hbar}{\veps^{2}}.
\end{equation}
We may therefore apply the a priori estimate from Corollary \ref{CorollaryAPrioriEstimateOnSurface} to the vortex $(A_{\nu},u_{\nu})$ on the disk $B_{\veps/2}(z^{\nu_{0}})$, obtaining
\[
\frac{1}{2} \Abs{\dop_{A_{\nu}}\!u_{\nu}(z^{\nu_{0}})}_{J}^{2} \le \frac{4\,C}{\veps^{2}} \cdot E\bigl( A_{\nu},u_{\nu};B_{\veps/2}(z^{\nu_{0}}) \bigr).
\]
Using the second inequality in (\ref{InequalityBadPoint}), we further infer
\[
\Abs{\dop_{A_{\nu}}\!u_{\nu}(z^{\nu_{0}})}_{J}^{2} < \frac{8\,C \hbar}{\veps^{2}},
\]
which contradicts the third inequality in (\ref{InequalityBadPoint}).
\end{proof}

Since $\sup_{\nu} E(A_{\nu},u_{\nu}) < \infty$ by assumption, it follows from Lemma~\ref{LemmaLowerEnergyBoundForBubbles} that the sequence $(A_{\nu},u_{\nu})$ has finitely many singular points. More specifically, we have the following result.

\begin{lemma} \label{LemmaSingularPoints}
	After passing to a subsequence, the sequence $(A_{\nu},u_{\nu})$ has a finite set
\[
Z = \{z_{1},\ldots,z_{N}\}
\]
of singular points in $\Sig$ and satisfies
\begin{eqnarray*}
	\sup_{\nu} \Norm{\dop_{A_{\nu}}\!u_{\nu}}_{L^{\infty}(K)} < \infty
\end{eqnarray*}
for every compact subset $K \subset \Sig \setminus Z$.
\end{lemma}

\begin{proof}
	The proof of this lemma is word by word the same as that of the Claim in the proof of \cite[Thm.\,4.6.1]{McDuff/J-holomorphic-curves-and-symplectic-topology}.
\end{proof}

By Lemma \ref{LemmaSingularPoints} we may henceforth assume that the sequence $(A_{\nu},u_{\nu})$ has finitely many singular points $Z \deq \{z_{1},\ldots,z_{N}\}$ and satisfies
\begin{eqnarray} \label{EstimateUniformL^inftyBoundOnTwistedDerivatives}
	\sup_{\nu} \Norm{\dop_{A_{\nu}}\!u_{\nu}}_{L^{\infty}(K)} < \infty
\end{eqnarray}
for every compact subset $K \subset \Sig \setminus Z$.

\subsection{Uhlenbeck compactness and Coulomb gauge}
\label{SubSectionUhlenbeckCompactnessAndCoulombGauge}

We investigate the convergence properties of the sequence $(A_{\nu},u_{\nu})$ more closely from the gauge-theoretic point of view.

\begin{lemma} \label{LemmaWeakLimitCompactification}
	Fix $p>2$. There exist a pair $(A,u)$ consisting of a connection $A\in\A^{1,p}(P)$ on $P$ and a section $u\in W^{1,p}_{\mathrm{loc}}(\Sig \setminus Z,P(\X))$ of the bundle $P(\X)$ defined on $\Sig \setminus Z$, a smooth reference connection $A_{0}\in\A(P)$, and a sequence of gauge transformations $g_{\nu} \in \G^{2,p}(P)$ such that the following holds.
\begin{enumerate}[itemsep=0.3ex, itemindent=0cm, labelsep=1ex, leftmargin=1cm]
	\item The connection $A$ is in Coulomb gauge relative to $A_{0}$ on $\Sig$, that is,
\[
\dop_{A_{0}}^{\ast}(A - A_{0}) = 0.
\]
	\item After passing to a subsequence, the sequence $(g_{\nu}^{\ast}A_{\nu},g_{\nu}^{-1}u_{\nu})$ converges to $(A,u)$ in the following sense.
\begin{enumerate}[topsep=1ex, itemsep=0.3ex, itemindent=0cm, labelsep=1ex, leftmargin=0.9cm]
	\item The sequence $g^{\ast}_{\nu} A_{\nu}$ converges to $A$ weakly in $W^{1,p}$ and strongly in $C^{0}$ on $\Sig$;
	\item the sequence $g^{-1}_{\nu} u_{\nu}$ converges to $u$ weakly in $W^{1,p}$ and strongly in $C^{0}$ on compact subsets of $\Sig \setminus Z$;
	\item every $g_{\nu}^{\ast} A_{\nu}$ is in Coulomb gauge relative to $A$ on $\Sig$, that is,
\[
\dop_{A}^{\ast}\bigl( g^{\ast}_{\nu} A_{\nu} - A \bigr) = 0.
\]
\end{enumerate}
\end{enumerate}
\end{lemma}

\begin{proof}
	Our proof is a variant of the arguments in the proofs of \cite[Thm.\,3.1 and Thm.\,3.2]{Cieliebak/The-symplectic-vortex-equations-and-invariants-of-Hamiltonian-group-actions}. Since $(A_{\nu},u_{\nu})$ solves the second vortex equation
\[
F_{A_{\nu}} = - \mu(u_{\nu}) \dvol_{\Sig},
\]
compactness of $\X$ yields a uniform $L^{p}$-bound for the sequence $F_{A_{\nu}}$. Hence by weak Uhlenbeck compactness (see \cite{Uhlenbeck/Connections-with-Lp-bounds-on-curvature} and \cite[Thm.\,A]{Wehrheim/Uhlenbeck-compactness}) there exists a sequence of gauge transformations \mbox{$h_{\nu}\in\G^{2,p}(P)$} such that the sequence $h_{\nu}^{\ast} A_{\nu}$ is uniformly bounded in $W^{1,p}$. It follows from the Banach-Alaoglu theorem that there exists a connection $\Atilde \in \A^{1,p}(P)$ such that, after passing to a subsequence, $h_{\nu}^{\ast} A_{\nu}$ converges to $\Atilde$ weakly in $W^{1,p}$.

Now we apply the local slice theorem \cite[Thm.\,F]{Wehrheim/Uhlenbeck-compactness}. We take $\Atilde$ as reference connection and choose a smooth connection $A_{0} \in \A(P)$ such that $\norm{\Atilde - A_{0}}_{W^{1,p}(\Sig)}$ (and hence also $\norm{\Atilde - A_{0}}_{L^{p}(\Sig)}$) is sufficiently small. Then the local slice theorem (taking $q=p$) asserts the existence of a gauge transformation $h\in\G^{2,p}(P)$ such that
\[
\dop^{\ast}_{\Atilde}\bigl( h_{\ast}A_{0} - \Atilde \bigr) = 0.
\]
By \cite[Lemma 8.4\,(iv)]{Wehrheim/Uhlenbeck-compactness}) this implies that
\[
\dop^{\ast}_{A_{0}}\bigl( h^{\ast} \Atilde - A_{0} \bigr) = 0.
\]
Define
\[
A \deq h^{\ast}\Atilde \in \A^{1,p}(P).
\]
Then $A$ is in Coulomb gauge relative to $A_{0}$ on $\Sig$. This proves (i).

Since $h_{\nu}^{\ast} A_{\nu}$ converges to $\Atilde$ weakly in $W^{1,p}$ as we have seen above, the sequence $h^{\ast} h_{\nu}^{\ast} A_{\nu}$ converges to $A = h^{\ast}\Atilde$ weakly in~$W^{1,p}$. In particular, it follows by the Sobolev embedding theorem and Rellich's theorem that, after passing to a subsequence, $h^{\ast} h_{\nu}^{\ast} A_{\nu}$ converges to $A$ strongly in $C^{0}$. Now we apply the local slice theorem a second time, taking~$A$ as reference connection. By what we just proved, we have
\[
\lim_{\nu\to\infty}\Norm{h^{\ast} h_{\nu}^{\ast} A_{\nu} - A}_{L^{p}(\Sig)} = 0, \quad \sup_{\nu}\Norm{h^{\ast} h_{\nu}^{\ast} A_{\nu} - A}_{W^{1,p}(\Sig)} < \infty.
\]
Hence by the local slice theorem there exist $\hat{h}_{\nu}\in\G^{2,p}(P)$ such that
\begin{eqnarray} \label{EqualityCoulombGauge2CompactificationWithBoundedDerivatives}
	\dop^{\ast}_{A}\bigl( \hat{h}_{\nu}^{\ast} h^{\ast} h_{\nu}^{\ast} A_{\nu} - A \bigr) = 0,
\end{eqnarray}
\begin{eqnarray} \label{ExpressionBoundsForGaugedSequenceCompactificationFirst}
	\lim_{\nu\to\infty}\Norm{\hat{h}_{\nu}^{\ast} h^{\ast} h_{\nu}^{\ast} A_{\nu} - A}_{L^{p}(\Sig)} = 0
\end{eqnarray}
and
\begin{eqnarray} \label{ExpressionBoundsForGaugedSequenceCompactificationSecond}
	\sup_{\nu}\Norm{\hat{h}_{\nu}^{\ast} h^{\ast} h_{\nu}^{\ast} A_{\nu} - A}_{W^{1,p}(\Sig)} < \infty.
\end{eqnarray}
We now define $g_{\nu} \deq h_{\nu} h \hat{h}_{\nu}$. Then (\ref{EqualityCoulombGauge2CompactificationWithBoundedDerivatives}) proves (c) in (ii). Furthermore, by (\ref{ExpressionBoundsForGaugedSequenceCompactificationSecond}) the sequence $g^{\ast}_{\nu} A_{\nu}$ is uniformly bounded in $W^{1,p}$. Thus, by the Banach-Alaoglu theorem, the Sobolev embedding theorem and Rellich's theorem it follows that, after passing to a subsequence, $g^{\ast}_{\nu} A_{\nu}$ converges to some connection $A^{\prime}$ weakly in $W^{1,p}$ and strongly in $C^{0}$. By (\ref{ExpressionBoundsForGaugedSequenceCompactificationFirst}) we conclude that $A^{\prime}=A$. This proves (a) in (ii). It remains to consider the sequence of sections $g_{\nu}^{-1}u_{\nu}$. By (\ref{EstimateUniformL^inftyBoundOnTwistedDerivatives}) we have
\[
\sup_{\nu} \Norm{\dop_{g^{\ast}_{\nu} A_{\nu}}\!\bigl( g_{\nu}^{-1} u_{\nu} \bigr)}_{L^{p}(K)} = \sup_{\nu} \Norm{\dop_{A_{\nu}}\!u_{\nu}}_{L^{p}(K)} \le \sup_{\nu} \Norm{\dop_{A_{\nu}}\!u_{\nu}}_{L^{\infty}(K)} < \infty
\]
for every compact subset $K \subset \Sig \setminus Z$. Hence, by compactness of $\X$ it follows from formula (\ref{EquationTwistedDerivative}) that $g_{\nu}^{-1} u_{\nu}$ is uniformly bounded in $W^{1,p}$ on compact subsets of $\Sig \setminus Z$. Hence there exists a section $u\in W^{1,p}_{\mathrm{loc}}(\Sig \setminus Z,P(\X))$ such that, after passing to a subsequence, $g^{-1}_{\nu} u_{\nu}$ converges to~$u$ weakly in $W^{1,p}$ and strongly in $C^{0}$ on compact subsets of $\Sig \setminus Z$. This proves (b) in (ii) and completes the proof of Lemma \ref{LemmaWeakLimitCompactification}.
\end{proof}

For the rest of this section, let us fix $p>2$. To simplify notation, we abbreviate
\[
\Ahat_{\nu} \deq g^{\ast}_{\nu} A_{\nu} \quad \text{and} \quad \uhat_{\nu} \deq g_{\nu}^{-1} u_{\nu},
\]
where $g_{\nu} \in \G^{2,p}(P)$ are the gauge transformations from Lemma~\ref{LemmaWeakLimitCompactification}. We restate the assertion of Lemma~\ref{LemmaWeakLimitCompactification} using this notation: There exists a pair $(A,u)$ consisting of a connection $A\in\A^{1,p}(P)$ on $P$ and a section $u\in W^{1,p}_{\mathrm{loc}}(\Sig \setminus Z,P(\X))$ of the bundle $P(\X)$ that is defined on $\Sig \setminus Z$, and a smooth reference connection $A_{0}\in\A(P)$ such that the following holds.
\begin{description}[topsep=1ex, itemsep=0.5ex, itemindent=0cm, labelsep=1ex, leftmargin=0.5cm]
	\item[(C1)] \label{EnumerationCoulombGaugeForA} The connection $A$ is in Coulomb gauge relative to $A_{0}$ on $\Sig$, that is,
\[
\dop_{A_{0}}^{\ast}(A - A_{0}) = 0.
\]
	\item[(C2)] \label{EnumerationConvergenceFor(A,u)} The sequence $(\Ahat_{\nu},\uhat_{\nu})$ converges to $(A,u)$ in the following sense.
\begin{enumerate}[topsep=1ex, itemsep=0.3ex, itemindent=0cm, labelsep=1ex, leftmargin=1cm]
	\item[(a)] The sequence $\Ahat_{\nu}$ converges to $A$ weakly in $W^{1,p}$ and strongly in $C^{0}$ on $\Sig$;
	\item[(b)] the sequence $\uhat_{\nu}$ converges to $u$ weakly in $W^{1,p}$ and strongly in $C^{0}$ on compact subsets of $\Sig \setminus Z$;
	\item[(c)] every $\Ahat_{\nu}$ is in Coulomb gauge relative to $A$ on $\Sig$, that is,
\begin{eqnarray*}
	\dop_{A}^{\ast}\bigl( \Ahat_{\nu} - A \bigr) = 0.
\end{eqnarray*}
\end{enumerate}
	\item[(C3)] \label{EnumerationUniformBoundForEnergy} The Yang-Mills-Higgs energy of the sequence $(\Ahat_{\nu},\uhat_{\nu})$ satisfies a uniform bound
\[
\sup_{\nu} E(\Ahat_{\nu},\uhat_{\nu}) < \infty.
\]
\end{description}

\renewcommand{\labelenumi}{(\roman{enumi})}

\subsection{The limit equations}
\label{SubSectionTheLimitEquations}

We consider the vortex equations
\[
\delbar_{J,\Ahat_{\nu}}(\uhat_{\nu}) = 0, \quad F_{\Ahat_{\nu}} + \mu(\uhat_{\nu}) \dvol_{\Sig} = 0
\]
in the limit $\nu \to \infty$ in order to obtain equations for the limit pair $(A,u)$. Since $u_{\nu}$ only converges on compact subsets of $\Sig \setminus Z$, the limit equations will only be defined on $\Sig \setminus Z$.

\begin{lemma} \label{LemmaLimitVortexEquationsOnSigma}
	The pair $(A,u)$ is a solution of class $W^{1,p}_{\mathrm{loc}}$ of the vortex equations (\ref{EquationVortexEquations}) on the complement $\Sig \setminus Z$ of the singular points.
\end{lemma}

\begin{proof}
	By condition (C2,\,a--b) the sequences $\delbar_{J,\Ahat_{\nu}}(\uhat_{\nu})$ and $F_{\Ahat_{\nu}} + \mu(\uhat_{\nu}) \dvol_{\Sig}$ converge to $\delbar_{J,A}(u)$ and $F_{A} + \mu(u) \dvol_{\Sig}$, respectively, weakly in~$L^{p}$ on compact subsets of $\Sig \setminus Z$. Since $(\Ahat_{\nu},\uhat_{\nu})$ satisfies the vortex equations
\[
\delbar_{J,\Ahat_{\nu}}(\uhat_{\nu}) = 0, \quad F_{\Ahat_{\nu}} + \mu(\uhat_{\nu}) \dvol_{\Sig} = 0
\]
for every $\nu$, it follows that $(A,u)$ is a solution of class $W^{1,p}_{\mathrm{loc}}$ of the vortex equations (\ref{EquationVortexEquations}) on $\Sig \setminus Z$.
\end{proof}

\subsection{Removal of singularities}
\label{SubSectionRemovalOfSingularities}

We apply Theorem~\ref{TheoremRemovalOfSingularities} in order to obtain limit equations for the pair $(A,u)$ that hold on all of $\Sig$. We begin by verifying that $(A,u)$ satisfies the assumptions of Theorem~\ref{TheoremRemovalOfSingularities}.

\begin{lemma} \label{LemmaSmoothnessAndFiniteEnergy}
	The limit pair $(A,u)$ has the following properties.
\begin{enumerate}[itemsep=0.3ex, itemindent=0cm, labelsep=1ex, leftmargin=1cm]
	\item $(A,u)$ is smooth on $\Sig \setminus Z$ and, after passing to a subsequence, $(\Ahat_{\nu},\uhat_{\nu})$ converges to $(A,u)$ in $C^{\infty}$ on compact subsets of $\Sig \setminus Z$.
	\item $(A,u)$ has finite Yang-Mills-Higgs energy $E(A,u) < \infty$.
\end{enumerate}
\end{lemma}

\begin{proof}
	Proof of (i): The proof is by elliptic bootstrapping and is similar to the proofs of \cite[Thm.\,3.1 and Thm.\,3.2]{Cieliebak/The-symplectic-vortex-equations-and-invariants-of-Hamiltonian-group-actions}, so we will be very brief on this. We prove that $(A,u)$ is of class $W^{k,p}$ for all $k \ge 1$, on any compact subset~$K \subset \Sig \setminus Z$. Since $p>2$, smoothness of $(A,u)$ on $\Sig \setminus Z$ will then follow by the Sobolev embedding theorem.

For $k=1$ this is true since $(A,u)$ is of class $W^{1,p}$ on $K$. Suppose now that $(A,u)$ is of class $W^{k,p}$ on $K$ for some $k \ge 1$. By Lemma \ref{LemmaLimitVortexEquationsOnSigma} the pair $(A,u)$ solves the vortex equations (\ref{EquationVortexEquations}) on the subset $K$.

Let $A_{0}$ be the smooth reference connection from Lemma~\ref{LemmaWeakLimitCompactification}, and write $\al \deq A - A_{0}$. Combining the second vortex equation (\ref{EquationVortexEquations}) with the Coulomb gauge condition (C1) we obtain an elliptic system
\[
\dop_{A_{0}}\!\al = - F_{A_{0}} - \frac{1}{2}\,[\al \wedge \al] - \mu(u)\,\dvol_{\Sig}, \quad \dop_{A_{0}}^{\ast}\!\al = 0.
\]
Since $(A,u)$ is of class $W^{k,p}$ on $K$, it follows that the right-hand sides of these equations are of class $W^{k,p}$ as well. Hence, by elliptic regularity, we conclude that $\al$, whence $A$, is of class $W^{k+1,p}$ on $K$.

Consider a holomorphic coordinate chart $\C \supset D \to \Sig$. By Remark \ref{RemarkLocalFromGlobal}, locally on $D$ the first vortex equation (\ref{EquationVortexEquations}) may be written in the form
\begin{eqnarray*} \label{EquationLinearPDERewrittenLimitGeneralRegularityCompactification}
	\del_{s}\!u  + J(u) \del_{t}\!u = - X_{\Phi}(u) - J(u) X_{\Psi}(u),
\end{eqnarray*}
where $A = \Phi \dop\!s + \Psi \dop\!t$. Since $(A,u)$ is of class $W^{k,p}$ on $K$, the right-hand side of this equation is of class $W^{k,p}$ as well. Hence elliptic regularity implies that~$u$ is of class $W^{k+1,p}$ on $K$ (see \cite[App.\,B.4]{McDuff/J-holomorphic-curves-and-symplectic-topology}).

This proves the first part of (i). The proof of the second part is similar and will be omitted. Note that it relies on the Coulomb gauge condition (C2,\,c) together with the fact that $A$ is smooth on $\Sig \setminus Z$.

\smallskip

\noindent Proof of (ii): Let $K_{\mu} \subset \Sig \setminus Z$ be an exhausting sequence of compact subsets such that
\[
K_{\mu} \subset K_{\mu+1} \quad \text{and} \quad \bigcup_{\mu} K_{\mu} = \Sig \setminus Z.
\]
By (i) above it follows that the sequence of functions
\[
e(\Ahat_{\nu},\uhat_{\nu}) = \frac{1}{2} \Abs{\dop_{\Ahat_{\nu}}\!\uhat_{\nu}}^{2}_{J} + \Abs{\mu(\uhat_{\nu})}^{2}
\]
converges to
\[
e(A,u) = \frac{1}{2} \Abs{\dop_{A}\!u}^{2}_{J} + \Abs{\mu(u)}^{2}
\]
in $C^{\infty}$ on every compact set $K_{\mu}$, whence
\[
\int_{K_{\mu}} e(A,u) \dvol_{\Sig} = \lim_{\nu\to\infty} \int_{K_{\mu}} e(\Ahat_{\nu},\uhat_{\nu}) \dvol_{\Sig}
\]
for every $\mu$. Moreover, we have
\[
\lim_{\mu\to\infty} \int_{K_{\mu}} e(\Ahat_{\nu},\uhat_{\nu}) \dvol_{\Sig} = \int_{\Sig \setminus Z} e(\Ahat_{\nu},\uhat_{\nu}) \dvol_{\Sig} = E(\Ahat_{\nu},\uhat_{\nu})
\]
for every $\nu$. By Fatou's lemma we therefore obtain
\begin{eqnarray*}
	E(A,u) &=& \int_{\Sig \setminus Z} e(A,u) \dvol_{\Sig} \,\le\, \liminf_{\mu\to\infty} \int_{K_{\mu}} e(A,u) \dvol_{\Sig} \\
	&=& \liminf_{\mu\to\infty} \left( \lim_{\nu\to\infty} \int_{K_{\mu}} e(\Ahat_{\nu},\uhat_{\nu}) \dvol_{\Sig} \right) \\
	&\le& \sup_{\nu} \left( \lim_{\mu\to\infty} \int_{K_{\mu}} e(\Ahat_{\nu},\uhat_{\nu}) \dvol_{\Sig} \right) = \, \sup_{\nu} E(\Ahat_{\nu},\uhat_{\nu}).
\end{eqnarray*}
In the last inequality we used that the sequence $\int_{K_{\mu}} e(\Ahat_{\nu},\uhat_{\nu}) \dvol_{\Sig}$ is nondecreasing for fixed~$\nu$. Since $\sup_{\nu} E(\Ahat_{\nu},\uhat_{\nu}) < \infty$ by condition (C3), assertion (ii) follows. This finishes the proof of Lemma \ref{LemmaSmoothnessAndFiniteEnergy}.
\end{proof}

\pagebreak

We are now in a position to apply Theorem \ref{TheoremRemovalOfSingularities} to the limit pair $(A,u)$.

\begin{lemma} \label{LemmaLimitVortexEquationsOnSigma_0}
	The limit pair $(A,u)$ is a solution of class $W^{1,p}$ of the vortex equations (\ref{EquationVortexEquations}) on all of $\Sig$.
\end{lemma}

\begin{proof}
	We apply Theorem \ref{TheoremRemovalOfSingularities} to each of the finitely many singular points in~$Z$. Let $z_{j}\in Z$ and choose a holomorphic chart $\C \supset B \to \Sig$ such that the origin is mapped to $z_{j}$. By Remark \ref{RemarkLocalFromGlobal}, locally in this chart the vortex $(A,u)$ gets identified with a triple $(\Phi,\Psi,u)$ satisfying the vortex equations
\[
\begin{split}
	\del_{s}\!u + X_{\Phi}(u) + J \bigl( \del_{t}\!u + X_{\Psi}(u) \bigr) &= 0, \\
	\del_{s}\!\Psi - \del_{t}\!\Phi + [\Phi,\Psi] + \lam^{2} \cdot \mu(u) &= 0
\end{split}
\]
on the punctured disk $B \setminus \{0\}$, where $\map{\Phi,\Psi}{B}{\gfr}$ are defined by $A = \Phi \dop\!s + \Psi \dop\!t$ and $\map{\lam}{B}{(0,\infty)}$ is defined by $\dvol_{\Sig} = \lam^{2}\,\dop\!s \wedge \dop\!t$. Since~$A$ is of class $W^{1,p}$ on $\Sig$ and $p>2$, it follows by the Sobolev embedding theorem that $\Phi$ and $\Psi$ are continuous on all of $B$. Moreover, by Lemma~\ref{LemmaSmoothnessAndFiniteEnergy}\,(i), $A$ is smooth on $\Sig \setminus Z$, whence $\Phi$ and $\Psi$ are smooth on $B \setminus \{0\}$. Lastly, $E(\Phi,\Psi,u;B) < \infty$ by Remark \ref{RemarkLocalFromGlobal} and Lemma~\ref{LemmaSmoothnessAndFiniteEnergy}\,(ii). Hence Theorem~\ref{TheoremRemovalOfSingularities} implies that the map~$u$ is of class $W^{1,p}$ on all of $B$. The lemma now follows from Lemma~\ref{LemmaLimitVortexEquationsOnSigma}.
\end{proof}

We close with two results concerning the regularity of the limit pair $(A,u)$ and of the pairs $(\Ahat_{\nu},\uhat_{\nu})$.

\begin{lemma} \label{LemmaEllipticRegularity}
\begin{enumerate}[itemsep=0.3ex, itemindent=0cm, labelsep=1ex, leftmargin=1cm]
	\item The limit pair $(A,u)$ is smooth on $\Sig$.
	\item For every $\nu$, the pair $(\Ahat_{\nu},\uhat_{\nu})$ is smooth on $\Sig$.
\end{enumerate}
\end{lemma}

\begin{proof}
	Proof of (i): By Lemma \ref{LemmaLimitVortexEquationsOnSigma_0}, the pair $(A,u)$ is a $W^{1,p}$-solution of the vortex equations~(\ref{EquationVortexEquations}) on all of $\Sig$. Moreover, by (C1) the connection $A$ is in Coulomb gauge relative to the smooth connection $A_{0}$. Hence assertion (i) follows by elliptic bootstrapping as in the proof of Lemma~\ref{LemmaSmoothnessAndFiniteEnergy}\,(i) above (see also the proof of \cite[Thm.\,3.1]{Cieliebak/The-symplectic-vortex-equations-and-invariants-of-Hamiltonian-group-actions}).

\smallskip

\noindent Proof of (ii): By gauge invariance of the vortex equations, for every $\nu$ the pair $(\Ahat_{\nu},\uhat_{\nu})$ is a $W^{1,p}$-solution of the vortex equations (\ref{EquationVortexEquations}) on $\Sig$. Moreover, by (C2,\,c) the connection $\Ahat_{\nu}$ is in Coulomb gauge relative to the connection~$A$. Since $A$ is smooth on $\Sig$ by part (i) above, assertion~(ii) now follows by elliptic bootstrapping as in (i).
\end{proof}

\subsection{Proof of Theorem \ref{TheoremConvergenceModuloBubbling}}
\label{SubSectionProofOfTheoremConvergenceModuloBubbling}

Let $Z = \{z_{1},\ldots,z_{N}\}$ be as in Lemma \ref{LemmaSingularPoints}. Fix $p>2$, and let the pair $(A,u)$ and the sequence of gauge transformations $g_{\nu} \in \G^{2,p}(P)$ be as in Lemma \ref{LemmaWeakLimitCompactification}. We will see below that the gauge transformations $g_{\nu}$ are actually smooth.

By Lemma~\ref{LemmaLimitVortexEquationsOnSigma_0} and Lemma \ref{LemmaEllipticRegularity}\,(i), the pair $(A,u)$ is a smooth vortex.

Recall that we abbreviated
\[
\Ahat_{\nu} \deq g^{\ast}_{\nu} A_{\nu} \quad \text{and} \quad \uhat_{\nu} \deq g_{\nu}^{-1} u_{\nu}.
\]
Assertion (i) of Theorem \ref{TheoremConvergenceModuloBubbling} then holds by (C2,\,a), while assertion (ii) follows from Lemma~\ref{LemmaSmoothnessAndFiniteEnergy}\,(i). Moreover, by Lemma \ref{LemmaEllipticRegularity}\,(ii), for every $\nu$ the connection
\[
g_{\nu}^{\ast} A_{\nu} = g_{\nu}^{-1} A_{\nu} g_{\nu} + g_{\nu}^{-1} \dop\!g_{\nu}
\]
is smooth. Since $A_{\nu}$ is smooth by assumption, a standard bootstrapping argument shows that the gauge transformations $g_{\nu}$ are actually smooth for every $\nu$.

It remains to prove assertions (iii) and (iv). Following the proof of \cite[Thm.\,4.6.1]{McDuff/J-holomorphic-curves-and-symplectic-topology}, we fix numbers $\veps_{j}>0$ for $j=1,\ldots,N$ such that the disks $B_{\veps_{j}}(z_{j})$ are pairwise disjoint. Then, after passing to a subsequence, the limits
\[
m_{\veps_{j}}(z_{j}) \deq \lim_{\nu\to\infty} E\bigl( A_{\nu},u_{\nu};B_{\veps_{j}}(z_{j}) \bigr)
\]
exist, and the function $\veps \mapsto m_{\veps}(z_{j})$ is continuous for $0<\veps \le \veps_{j}$. By Lemma~\ref{LemmaLowerEnergyBoundForBubbles},
\[
\liminf_{\nu\to\infty} E\bigl( A_{\nu},u_{\nu};B_{\veps_{j}}(z_{j}) \bigr) \ge \hbar,
\]
whence
\[
m(z_{j}) \deq \lim_{\veps\to 0} m_{\veps}(z_{j}) \ge \hbar.
\]
This proves (iii). To prove (iv), fix a number $\veps\le\min_{j}\veps_{j}$ and note that
\[
\begin{split}
	E\!\left( A,u;K \setminus \bigcup_{j=1}^{N} B_{\veps}(z_{j}) \right) &= \lim_{\nu\to\infty} E\bigl( A_{\nu},u_{\nu};K \bigr) - \sum_{j=1}^{N} \lim_{\nu\to\infty} E\bigl( A_{\nu},u_{\nu};B_{\veps}(z_{j}) \bigr) \\
	&= \lim_{\nu\to\infty} E\bigl( A_{\nu},u_{\nu};K \bigr) - \sum_{j=1}^{N} m_{\veps}(z_{j}).
\end{split}
\]
Taking the limit $\veps\to 0$, we get
\[
E(A,u;K) = \lim_{\nu\to\infty} E\bigl( A_{\nu},u_{\nu};K \bigr) - \sum_{j=1}^{N} m(z_{j}).
\]
This proves (iv) and completes the proof of Theorem~\ref{TheoremConvergenceModuloBubbling}.

\section{Gromov compactness}
\label{SectionGromovCompactness}

The aim of this section is to prove Theorem~\ref{TheoremGromovCompactness}. Following the approach of Mundet i Riera \cite{Mundet-i-Riera/Hamiltonian-Gromov-Witten-invariants}, our strategy will be to reduce the compactification problem for vortices to Gromov compactness for pseudoholomorphic curves. To this end, we shall apply Gromov's graph construction in order to transform vortices into pseudoholomorphic sections of the associated symplectic fiber bundle $P(\X) = P \xop_{G} \X$ over $\Sig$. We then deduce a mean value inequality for these sections from the a priori estimate for vortices proved in Section~\ref{SectionAPrioriEstimate}. All this will be explained in Section~\ref{SubSectionVorticesVersusPseudoholomorphicMaps}. In Section \ref{SubSectionBubblesConnectRevisited}, we generalize the bubbling analysis from McDuff and Salamon \cite[Sec.\,4.7]{McDuff/J-holomorphic-curves-and-symplectic-topology} in such a way that it also applies to pseudoholomorphic sections of $P(\X)$ induced by vortices. The actual proof of Theorem~\ref{TheoremGromovCompactness} is then given in Section~\ref{SubSectionProofOfGromovCompactness}, where we assemble the results previously obtained in Sections \ref{SectionConvergenceModuloBubbling} and \ref{SectionGromovCompactness}. We keep the notation introduced in Section~\ref{SectionIntroductionAndMainResults}.

\subsection{Vortices vs.\,pseudoholomorphic curves}
\label{SubSectionVorticesVersusPseudoholomorphicMaps}

We explain how vortices naturally occur as pseudoholomorphic sections of the bundle $P(\X) \to \Sig$, and prove a mean value inequality for such sections. This may be regarded as a global version of the graph construction from Section~\ref{SectionRemovalOfSingularities}. The main results are collected in Lemma \ref{LemmaVorticesAsPseudoholomorphicMaps} and Proposition \ref{PropositionVorticesAsPseudoholomorphicMaps} below.

\smallskip

Let us begin by explaining how the total space $P(\X) = P \xop_{G} \X$ naturally inherits the structure of an almost complex symplectic manifold. Fix an arbitrary smooth connection~$A$ on the $G$-bundle $P \to \Sig$. It is a well-known fact (see \cite{Mundet-i-Riera/Hamiltonian-Gromov-Witten-invariants, Cieliebak/J-holomorphic-curves-moment-maps-and-invariants-of-Hamiltonian-group-actions, Gaio/Gromov-Witten-invariants-of-symplectic-quotients-and-adiabatic-limits, Guillemin/Supersymmetry-and-equivariant-de-Rham-theory}) that $A$, together with the symplectic form $\om$ on $\X$, the almost complex structure $J$ on $\X$, and the complex structure~$j_{\Sig}$ on~$\Sig$, gives rise to a symplectic form and an almost complex structure on $P(\X)$. For later reference, we briefly review these constructions.

First, we define a symplectic form $\om_{A}$ on $P(\X)$. Let us denote by $\map{p_{1}}{P \xop \X}{P}$ and $\map{p_{2}}{P \xop \X}{\X}$ the canonical projections, and consider the 2-form
\[
\sigtilde_{A} \deq \om - \dop \l A,\mu \r = p_{2}^{\ast}\,\om - \dop \bigl\l p_{1}^{\ast}A,\mu \circ p_{2} \bigr\r_{\gfr}
\]
on $P \xop \X$. It descends to a closed 2-form $\sig_{A}$ on $P(\X)$, called the coupling form (see \cite{Guillemin/Supersymmetry-and-equivariant-de-Rham-theory}). Note that $\sig_{A}$ may be degenerate in the horizontal direction. We make it into a symplectic form by adding on a sufficiently large multiple of the pull-back of the area form $\dvol_{\Sig}$ along the bundle projection $\map{p}{P(\X)}{\Sig}$. This leads us to define the symplectic form~$\om_{A}$ by
\begin{eqnarray} \label{EquationSymplecticFormOnP(X)}
	\om_{A} \deq (1 + c_{A,\mu}) \cdot p^{\ast}\!\dvol_{\Sig} + \, \sig_{A},
\end{eqnarray}
where $c_{A,\mu} > 0$ is a sufficiently large constant. It will later be convenient to choose $c_{A,\mu}$ in such a way that
\begin{eqnarray} \label{EstimateForTwoFormOnP(X)}
	\Abs{\l F_{A}(v_{1},v_{2}),\mu \r_{\gfr}} \le c_{A,\mu} \cdot \Abs{\dop\!\pi(v_{1})} \cdot \Abs{\dop\!\pi(v_{2})}
\end{eqnarray}
for all $v_{1},v_{2} \in TP$, where $\map{\pi}{P}{\Sig}$ denotes the bundle projection. Note that such a constant~$c_{A,\mu}$ exists since $F_{A}$ is horizontal and $\X$ is compact.

Second, we define an almost complex structure $J_{A}$ on $P(\X)$. For that purpose, we consider the splitting of the tangent bundle $TP(\X)$ induced by the connection $A$. More precisely, recall that we denote the points of $P(\X)$ by $[p,x]$, where $p\in P$ and $x\in\X$. The tangent space $T_{[p,x]}P(\X)$ is given by
\[
T_{[p,x]}P(\X) = \bigl( T_{p}P \xop T_{x}\X \bigr) \big/ \bigl\{ (p.\xi,-X_{\xi}(x)) \,|\, \xi\in\gfr \bigr\},
\]
where $p.\xi$ and $X_{\xi}(x)$ denote the infinitesimal action of $\xi \in \gfr$ on $P$ at $p$ and on $\X$ at~$x$, respectively. Its elements will be denoted by $[v,w]$, where $v\in T_{p}P$ and $w\in T_{x}\X$. The connection~$A$ then gives rise to a splitting $TP \cong TP^{\,\hor} \oplus TP^{\,\ver}$ into horizontal and vertical subbundles, denoted by $v = v^{\hor} + p.A_{p}(v)$ for $v \in T_{p}P$. It further induces a splitting $TP(\X) \cong TP(\X)^{\hor} \oplus TP(\X)^{\ver}$, and any tangent vector $[v,w] \in T_{[p,x]}P(\X)$ may then be written as
\begin{eqnarray} \label{EquationHorizontalDecompositionInTP(X)}
	[v,w] = \bigl[ v^{\hor}, w + X_{A_{p}(v)}(x) \bigr].
\end{eqnarray}
The almost complex structure $J_{A}$ is now defined in terms of the complex structure $j_{\Sig}$ on $\Sig$ and the almost complex structure $J$ on $\X$ by the formula
\begin{eqnarray} \label{EquationAlmostComplexStructureOnP(X)}
	J_{A} [v,w] \deq \Bigl[ (\pi^{\ast}j_{\Sig})_{p} \, v^{\hor},J \bigl( w + X_{A_{p}(v)}(x) \bigr) \Bigr],
\end{eqnarray}
where we denote by $\pi^{\ast}j_{\Sig}$ the $G$-equivariant lift of $j_{\Sig}$ to $TP^{\hor}$. A straightforward computation shows that $J_{A}$ satisfies $J_{A}^{2}[v,w] = -[v,w]$.

\medskip

We are now in a position to state the main results of this subsection. The key observation is the next lemma, see \cite{Mundet-i-Riera/Hamiltonian-Gromov-Witten-invariants, Cieliebak/J-holomorphic-curves-moment-maps-and-invariants-of-Hamiltonian-group-actions}. It explains how vortices give rise to pseudoholomorphic sections of the bundle $P(\X)$.

\pagebreak

\begin{lemma} \label{LemmaVorticesAsPseudoholomorphicMaps}
	Fix a smooth connection $A$ on $P$, with corresponding almost complex structure~$J_{A}$ on $P(\X)$ defined by formula (\ref{EquationAlmostComplexStructureOnP(X)}). Let $\map{u}{P}{\X}$ be a smooth $G$-equivariant map, and denote by $\map{\utilde}{\Sig}{P(\X)}$ the corresponding section as in Remark \ref{RemarkMapVersusSection}. Then $u$ satisfies the first vortex equation
\[
\delbar_{J,A}(u) = \frac{1}{2} \, \bigl( \dop_{A}\!u + J(u) \circ \dop_{A}\!u \circ j_{\Sig} \bigr) = 0
\]
if and only if $\utilde$ is $(j_{\Sig},J_{A})$-holomorphic, that is,
\[
\delbar_{J_{A}}(\utilde) = \frac{1}{2} \bigl( \dop\!\utilde + J_{A}(\utilde) \circ \dop\!\utilde \circ j_{\Sig} \bigr) = 0.
\]
\end{lemma}

\begin{proof}
	Recall from Remark \ref{RemarkMapVersusSection} the definition of the section $\map{\utilde}{\Sig}{P(\X)}$. Let $z\in \Sig$ and $v\in T_{z}\Sig$. By formula (\ref{EquationTwistedDerivative}) we have
\begin{eqnarray} \label{EquationDerivativeOfutilde}
	\dop\!\utilde (v) = \bigl[ \vtilde,\dop\!u (\vtilde) \bigr] = \bigl[ \vtilde,\dop_{A}\!u (\vtilde) \bigr],
\end{eqnarray}
where $\vtilde\in T_{p}P$ denotes the $A$-horizontal lift of $v$, for some $p\in P$ such that $\pi(p)=z$. Then we have
\[
\begin{split}
	\delbar_{J_{A}}(\utilde)(v) &= \frac{1}{2} \Big( \dop\!\utilde(v) + \bigl( J_{A}(\utilde) \circ \dop\!\utilde \circ j_{\Sig} \bigr)(v) \Big) \\
	&= \frac{1}{2} \Big( \bigl[ \vtilde,\dop_{A}\!u(\vtilde) \bigr] + J_{A}(\utilde) \bigl[ (\pi^{\ast}j_{\Sig}) \vtilde,\dop_{A}\!u \bigl( (\pi^{\ast}j_{\Sig}) \vtilde \bigr) \bigr] \Big) \\
	&= \frac{1}{2} \Big( \bigl[ \vtilde,\dop_{A}\!u(\vtilde) \bigr] + \bigl[ (\pi^{\ast}j_{\Sig})^{2} \vtilde,\bigl( J(u) \circ \dop_{A}\!u  \circ \pi^{\ast}j_{\Sig} \bigr) (\vtilde) \bigr] \Big) \\
	&= \frac{1}{2} \Big( \bigl[ \vtilde,\dop_{A}\!u(\vtilde) \bigr] + \bigl[ -\vtilde,\bigl( J(u) \circ \dop_{A}\!u  \circ \pi^{\ast}j_{\Sig} \bigr) (\vtilde) \bigr] \Big) \\
	&= \left[ 0, \frac{1}{2} \Big( \dop_{A}\!u(\vtilde)  + \bigl( J(u) \circ \dop_{A}\!u  \circ \pi^{\ast}j_{\Sig} \bigr) (\vtilde) \Big) \right] = \bigl[ 0,\delbar_{J,A}(u)(\vtilde) \bigr],
\end{split}
\]
and the lemma follows.
\end{proof}

The next proposition provides a mean value inequality for the pseudoholomorphic sections that are associated to vortices as in the previous lemma.

\begin{proposition} \label{PropositionVorticesAsPseudoholomorphicMaps}
	Fix a smooth reference connection $A_{0}$ on $P$, with corresponding symplectic form $\om_{A_{0}}$ and almost complex structure~$J_{A_{0}}$ on $P(\X)$ defined by formulas (\ref{EquationSymplecticFormOnP(X)}) and (\ref{EquationAlmostComplexStructureOnP(X)}), respectively. Then there exist constants $c, c^{\prime} > 0$, $r_{0} > 0$, and $\de, C > 0$ such that for all connections $A$ satisfying
\[
\Norm{A-A_{0}}_{C^{0}(\Sig)} \le c
\]
the following holds.
\begin{enumerate}[topsep=1ex, itemsep=0.3ex, itemindent=0cm, labelsep=1ex, leftmargin=1cm]
	\item The almost complex structure $J_{A}$ on $P(\X)$ defined by formula (\ref{EquationAlmostComplexStructureOnP(X)}) is tamed by the symplectic form $\om_{A_{0}}$.
\end{enumerate}
Let moreover $\map{u}{P}{\X}$ be a smooth $G$-equivariant map. Suppose that $(A,u)$ is a vortex, and denote by $\map{\utilde}{\Sig}{P(\X)}$ the $J_{A}$-holomorphic section of $P(\X)$ induced by $u$ as in Lemma \ref{LemmaVorticesAsPseudoholomorphicMaps}. Denote by
\[
\l\cdot\,,\cdot\r_{J_{A}} \deq \frac{1}{2} \bigl( \om_{A_{0}}(\cdot\,,J_{A}\,\cdot) - \om_{A_{0}}(J_{A}\,\cdot,\,\cdot) \bigr)
\]
the Riemannian metric on $P(\X)$ determined by $\om_{A_{0}}$ and $J_{A}$, which is well-defined by (i) above. Recall from \cite[Sec.\,2.2]{McDuff/J-holomorphic-curves-and-symplectic-topology} that the energy of $\utilde$ is given by
\[
E_{J_{A}}(\utilde) \deq \frac{1}{2} \int_{\Sig} \Abs{\dop\!\utilde}_{J_{A}}^{2} \, \dvol_{\Sig},
\]
where the norm $\Abs{\dop\!\utilde}_{J_{A}}$ is understood with respect to the metric $\l\cdot\,,\cdot\r_{J_{A}}$ on $P(\X)$ and the metric~$\l\cdot\,,\cdot\r_{\Sig}$ on $\Sig$.
\begin{enumerate}[topsep=1ex, itemsep=0.3ex, itemindent=0cm, labelsep=1ex, leftmargin=1cm]
	\item[(ii)] The energy of the section $\utilde$ and the Yang-Mills-Higgs energy of the vortex $(A,u)$ are related by
\[
E_{J_{A}}(\utilde) \le c^{\prime} \cdot \bigl( E(A,u) + \volume(\Sig) \bigr),
\]
where $\volume(\Sig)$ denotes the area of $\Sig$ with respect to $\dvol_{\Sig}$.
	\item[(iii)] For all $z_{0} \in \Sig$ and all $0 < r < r_{0}$, the section $\utilde$ satisfies a mean value inequality
\[
\hspace{1cm} E_{J_{A}} \bigl( \utilde;B_{r}(z_{0}) \bigr) < \de \quad \Longrightarrow \quad \Abs{\dop\!\utilde(z_{0})}_{J_{A}}^{2} \le \frac{C}{r^{2}} \cdot E_{J_{A}}\bigl( \utilde;B_{r}(z_{0}) \bigr) + C.
\]
\end{enumerate}
\end{proposition}

\bigskip

The proof of Proposition~\ref{PropositionVorticesAsPseudoholomorphicMaps} will occupy the remainder of this subsection.

\smallskip

Let us fix a smooth reference connection $A_{0}$ on $P$. It gives rise to a symplectic form $\om_{A_{0}}$ on the total space $P(\X)$ by formula (\ref{EquationSymplecticFormOnP(X)}). Let $A$ be a smooth connection on~$P$.

\medskip

\noindent{\textbf{Proof of (i):}}
	Let $[v,w]\in TP(\X)$ such that $[v,w] \neq [0,0]$. By formula~(\ref{EquationHorizontalDecompositionInTP(X)}) we may without loss of generality assume that $v$ is $A$-horizontal. Then formula (\ref{EquationAlmostComplexStructureOnP(X)}) becomes
\[
J_{A} [v,w] = \bigl[ (\pi^{\ast}j_{\Sig}) v,Jw \bigr].
\]
Combining this with formula (\ref{EquationSymplecticFormOnP(X)}) we obtain
\begin{multline} \label{EquationFirstCalculationTaming}
	\om_{A_{0}}\bigl( [v,w],J_{A} [v,w] \bigr) = \om\bigl( w + X_{A_{0}(v)},Jw + X_{A_{0}((\pi^{\ast}j_{\Sig}) v)} \bigr) \\ - \, \bigl\l F_{A_{0}}\bigl( v,(\pi^{\ast}j_{\Sig}) v \bigr),\mu \bigr\r + \, (1 + c_{A_{0},\mu}) \cdot \dvol_{\Sig}\bigl( \dop\!\pi(v),j_{\Sig} \dop\!\pi(v) \bigr).
\end{multline}
In order to estimate the first term on the right-hand side, we write it as
\begin{multline*}
	\om\bigl( w + X_{A_{0}(v)},Jw + X_{A_{0}((\pi^{\ast}j_{\Sig}) v)} \bigr) = \om\bigl( w,Jw \bigr) + \om\bigl( X_{(A - A_{0})(\dop\!\pi(v))},J w \bigr) \\ + \om\bigl( w,X_{(A-A_{0})(j_{\Sig} \dop\!\pi(v))} \bigr) + \om\bigl( X_{(A - A_{0})(\dop\!\pi(v))},X_{(A-A_{0})(j_{\Sig} \dop\!\pi(v))} \bigr).
\end{multline*}
Here we used that $A-A_{0}$ is horizontal and hence descends to $\Sig$, and that $A(v)=0$ since $v$ is $A$-horizontal by assumption. Applying the inequalities of Cauchy-Schwarz and Young it follows that there exists a constant $c_{1}>0$, not depending on $A$, such that
\[
\Abs{\om\bigl( w + X_{A_{0}(v)},Jw + X_{A_{0}((\pi^{\ast}j_{\Sig}) v)} \bigr)} \ge \frac{1}{2} \Abs{w}_{J}^{2} - c_{1} \cdot \Norm{A-A_{0}}_{C^{0}(\Sig)}^{2} \cdot \abs{\dop\!\pi(v)}^{2}.
\]
Furthermore, by inequality~(\ref{EstimateForTwoFormOnP(X)}) the last two terms on the right-hand side of (\ref{EquationFirstCalculationTaming}) may be estimated by
\[
- \bigl\l F_{A_{0}}\bigl( v,(\pi^{\ast}j_{\Sig}) v \bigr),\mu \bigr\r + (1 + c_{A_{0},\mu}) \cdot \dvol_{\Sig}\bigl( \dop\!\pi(v),j_{\Sig} \dop\!\pi(v) \bigr) \\ \ge \Abs{\dop\!\pi(v)}^{2}.
\]
Hence we conclude that
\[
\om_{A_{0}}\bigl( [v,w],J_{A} [v,w] \bigr) \ge \frac{1}{2} \, \abs{w}_{J}^{2} + \left( 1 - c_{1} \cdot \Norm{A-A_{0}}_{C^{0}(\Sig)}^{2} \right) \cdot \Abs{\dop\!\pi(v)}^{2} > 0
\]
whenever $\norm{A-A_{0}}_{C^{0}(\Sig)}$ is sufficiently small. This proves (i).

\bigskip

Let now $\map{u}{P}{\X}$ be a smooth $G$-equivariant map such that $(A,u)$ is a vortex. Denote by $\map{\utilde}{\Sig}{P(\X)}$ the $J_{A}$-holomorphic section of $P(\X)$ induced by $u$ as in Lemma \ref{LemmaVorticesAsPseudoholomorphicMaps}. Write
\[
\l\cdot\,,\cdot\r_{J_{A}} \deq \frac{1}{2} \bigl( \om_{A_{0}}(\cdot\,,J_{A}\,\cdot) - \om_{A_{0}}(J_{A}\,\cdot,\,\cdot) \bigr)
\]
for the Riemannian metric on $P(\X)$ determined by $\om_{A_{0}}$ and $J_{A}$, which is well-defined by (i) above. Recall further that the energy of $\utilde$ is given by
\[
E_{J_{A}}(\utilde) \deq \frac{1}{2} \int_{\Sig} \Abs{\dop\!\utilde}_{J_{A}}^{2} \, \dvol_{\Sig}.
\]
Before we turn to the proof of assertions (ii) and (iii) of Proposition~\ref{PropositionVorticesAsPseudoholomorphicMaps}, we prove the following technical lemma.

\begin{lemma} \label{LemmaComparisonOfMetricsOnP(X)}
	There exists a constant $C_{A_{0}} > 0$, not depending on $(A,u)$, such that the following holds. Whenever $\norm{A-A_{0}}_{C^{0}(\Sig)}$ is sufficiently small, we have
\[
\frac{1}{2} \Bigl( 2 + \Abs{\dop_{A}\!u}^{2}_{J} \Bigr) \le \Abs{\dop\!\utilde}^{2}_{J_{A}} \le C_{A_{0}} \cdot \Bigl( 2 + \Abs{\dop_{A}\!u}^{2}_{J} \Bigr).
\]
\end{lemma}

\begin{proof}
	Let $[v,w]\in TP(\X)$, and note that
\[
\Abs{[v,w]}_{J_{A}}^{2} = \om_{A_{0}}\bigl( [v,w],J_{A} [v,w] \bigr).
\]
The computation in the proof of part (i) of Proposition \ref{PropositionVorticesAsPseudoholomorphicMaps} above hence shows that
\[
\Abs{[v,w]}_{J_{A}}^{2} \ge \frac{1}{2} \Abs{w}_{J}^{2} + \Bigl( 1 - c_{1} \cdot \Norm{A-A_{0}}_{C^{0}(\Sig)}^{2} \Bigr) \cdot \Abs{\dop\!\pi(v)}^{2}
\]
for some constant $c_{1}>0$, not depending on $A$. A similar computation yields
\[
\Abs{[v,w]}_{J_{A}}^{2} \le 2 \Abs{w}_{J}^{2} + c_{2} \cdot \Bigl( 1 + \Norm{A-A_{0}}_{C^{0}(\Sig)}^{2} \Bigr) \cdot \Abs{\dop\!\pi(v)}^{2}
\]
for some constant $c_{2}>0$, not depending on $A$. Hence there exists a constant $C_{A_{0}}>0$, not depending on $A$, such that
\[
\frac{1}{2} \, \Bigl( \Abs{\dop\!\pi(v)}^{2} + \Abs{w}^{2}_{J} \Bigr) \le \Abs{[v,w]}_{J_{A}}^{2} \le C_{A_{0}} \cdot \Bigl( \Abs{\dop\!\pi(v)}^{2} + \Abs{w}^{2}_{J} \Bigr)
\]
whenever $\norm{A-A_{0}}_{C^{0}(\Sig)}$ is sufficiently small. By formula (\ref{EquationDerivativeOfutilde}), the claimed inequality follows.
\end{proof}

\bigskip

\noindent{\textbf{Proof of (ii):}}
	Using Lemma \ref{LemmaComparisonOfMetricsOnP(X)} and formula (\ref{EquationYMHEnergyGlobal}), we obtain
\begin{eqnarray*}
	E_{J_{A}}(\utilde) &=& \frac{1}{2} \, \int_{\Sig} \Abs{\dop\!\utilde}_{J_{A}}^{2} \dvol_{\Sig} \,\le\, \frac{C_{A_{0}}}{2} \cdot \int_{\Sig} \Bigl( 2 + \Abs{\dop_{A}\!u}^{2}_{J} \Bigr) \dvol_{\Sig} \\
	&=& C_{A_{0}} \cdot \int_{\Sig} \left( \frac{1}{2} \Abs{\dop_{A}\!u}^{2}_{J} + \Abs{\mu(u)}^{2} \right) \dvol_{\Sig} \,+\, C_{A_{0}} \cdot \volume(\Sig) - C_{A_{0}} \cdot \int_{\Sig} \Abs{\mu(u)}^{2} \dvol_{\Sig} \\
	&\le& C_{A_{0}} \cdot \Bigl( E(A,u) + \volume(\Sig) \Bigr)
\end{eqnarray*}
whenever $\norm{A-A_{0}}_{C^{0}(\Sig)}$ is sufficiently small. This proves (ii).

\medskip

\noindent{\textbf{Proof of (iii):}} Let $z_{0} \in \Sig$. By Corollary~\ref{CorollaryAPrioriEstimateOnSurface} there exist constants $\hbar > 0$, $C^{\prime} > 0$ and $R > 0$, not depending on $(A,u)$, such that for all $0 < r < R$
\begin{equation} \label{InequalityAPrioriEstimateVorticesVsPseudoholomorphicCurves}
	E\bigl( A,u;B_{r}(z_{0}) \bigr) < \hbar \quad \Longrightarrow \quad
\frac{1}{2} \Abs{\dop_{A}\!u(z_{0})}_{J}^{2} + \Abs{\mu(u(z_{0}))}^{2} \le \frac{C^{\prime}}{r^{2}} \cdot E\bigl( A,u;B_{r}(z_{0}) \bigr).
\end{equation}
Define constants
\[
K \deq \norm{\mu}_{C^{0}(M)}^{2}
\]
and
\[
\de \deq \frac{\hbar}{4}, \quad C \deq 4\,C_{A_{0}} \cdot \bigl( 1 + C^{\prime} \bigr) \cdot \left( 1 + K \cdot \sup_{0<r<R} \frac{\volume(B_{r}(z_{0}))}{r^{2}} \right),
\]
where $C_{A_{0}}$ is the constant from Lemma~\ref{LemmaComparisonOfMetricsOnP(X)} and $\volume( B_{r}(z_{0}) )$ denotes the area of $B_{r}(z_{0})$ with respect to $\dvol_{\Sig}$. Choose a positive constant $r_{0} < R$ such that
\begin{eqnarray} \label{InequalityDefinitionOfr_0VorticesVsPseudoholomorphicCurves}
	\volume\bigl( B_{r}(z_{0}) \bigr) \le \frac{\hbar}{2K}
\end{eqnarray}
for all $0 < r < r_{0}$. Assume now that
\begin{eqnarray} \label{InequalityAssumptionBoundOnDirichletEnergyGromorvConvergence}
	r < r_{0} \quad \text{and} \quad E_{J_{A}}\bigl( \utilde;B_{r}(z_{0}) \bigr) < \de.
\end{eqnarray}
Using Lemma~\ref{LemmaComparisonOfMetricsOnP(X)} and formula (\ref{EquationYMHEnergyGlobal}) we then obtain
\begin{eqnarray*}
	&& E_{J_{A}}\bigl( \utilde;B_{r}(z_{0}) \bigr) \\
	&=& \frac{1}{2} \int_{B_{r}(z_{0})} \Abs{\dop\!\utilde}_{J_{A}}^{2} \dvol_{\Sig} \,\ge\, \frac{1}{4} \int_{B_{r}(z_{0})} \Bigl( 2 + \Abs{\dop_{A}\!u}^{2}_{J} \Bigr) \dvol_{\Sig} \\
	&=& \frac{1}{2} \int_{B_{r}(z_{0})} \left( \frac{1}{2} \Abs{\dop_{A}\!u}^{2}_{J} + \Abs{\mu(u)}^{2} \right) \dvol_{\Sig} \,-\, \frac{1}{2} \int_{B_{r}(z_{0})} \Abs{\mu(u)}^{2} \dvol_{\Sig} + \frac{1}{2} \volume\bigl( B_{r}(z_{0}) \bigr) \\
	&\ge& \frac{1}{2} \, E\bigl( A,u;B_{r}(z_{0}) \bigr) - \frac{1}{2} \, K \cdot \volume\bigl( B_{r}(z_{0}) \bigr),
\end{eqnarray*}
whence
\begin{eqnarray} \label{InequalityComparisonYMHAndDirichletEnergyGromovConvergence}
	E\bigl( A,u;B_{r}(z_{0}) \bigr) \le 2 \, E_{J_{A}}\bigl( \utilde;B_{r}(z_{0}) \bigr) + K \cdot \volume\bigl( B_{r}(z_{0}) \bigr).
\end{eqnarray}
Using inequalities (\ref{InequalityDefinitionOfr_0VorticesVsPseudoholomorphicCurves}) and (\ref{InequalityAssumptionBoundOnDirichletEnergyGromorvConvergence}), it follows from this that
\begin{eqnarray*}
	E\bigl( A,u;B_{r}(z_{0}) \bigr) < \frac{\hbar}{2} + \frac{\hbar}{2} = \hbar.
\end{eqnarray*}
Hence the a priori estimate (\ref{InequalityAPrioriEstimateVorticesVsPseudoholomorphicCurves}) implies that
\begin{eqnarray} \label{InequalityYMHEnergyDensityGromovConvergence}
	\frac{1}{2} \Abs{\dop_{A}\!u(z_{0})}_{J}^{2} + \Abs{\mu(u(z_{0}))}^{2} \le \frac{C^{\prime}}{r^{2}} \cdot E\bigl( A,u;B_{r}(z_{0}) \bigr).
\end{eqnarray}
Using Lemma~\ref{LemmaComparisonOfMetricsOnP(X)} and formula (\ref{EquationYMHEnergyGlobal}), it follows that
\begin{eqnarray*}
	\Abs{\dop\!\utilde(z_{0})}_{J_{A}}^{2}
	&\le& 2\,C_{A_{0}} \cdot \left( \frac{1}{2} \Abs{\dop_{A}\!u(z_{0})}^{2}_{J} + \Abs{\mu(u(z_{0}))}^{2} \right) + 2 \, C_{A_{0}} \\
	&\le& \frac{2\,C^{\prime}\,C_{A_{0}}}{r^{2}} \cdot E\bigl( A,u;B_{r}(z_{0}) \bigr) + 2 \, C_{A_{0}}.
\end{eqnarray*}
Applying inequality~(\ref{InequalityComparisonYMHAndDirichletEnergyGromovConvergence}) again, we finally obtain
\begin{eqnarray*}
	\Abs{\dop\!\utilde(z_{0})}_{J_{A}}^{2} &\le& \frac{4\,C^{\prime}\,C_{A_{0}}}{r^{2}} \cdot E_{J_{A}}\bigl( \utilde;B_{r}(z_{0}) \bigr) + 2\,C_{A_{0}} \cdot \left( C^{\prime} \cdot K \cdot \frac{\volume(B_{r}(z_{0}))}{r^{2}} + 1 \right) \\
	&\le& \frac{C}{r^{2}} \cdot E_{J_{A}}\bigl( \utilde;B_{r}(z_{0}) \bigr) + C.
\end{eqnarray*}
This proves (iii), and completes the proof of Proposition~\ref{PropositionVorticesAsPseudoholomorphicMaps}.

\subsection{Bubbles connect revisited}
\label{SubSectionBubblesConnectRevisited}

We prove Proposition~\ref{PropositionBubblesConnect} below, which provides preliminary results that will be needed to carry out the bubbling analysis in the proof of Theorem~\ref{TheoremGromovCompactness} in Section~\ref{SubSectionProofOfGromovCompactness}. It is adapted from \mbox{McDuff} and Salamon \cite[Prop.\,4.7.1 and Prop.\,4.7.2]{McDuff/J-holomorphic-curves-and-symplectic-topology}. We keep the notation introduced in Section \ref{SubSectionVorticesVersusPseudoholomorphicMaps}. Before stating the proposition, let us explain the set-up and fix some more notation.

We shall consider a sequence $A_{\nu}$ of smooth connections on $P$ that converges to a smooth connection $A$ on $P$ weakly in $W^{1,p}$ on $\Sig$, for some fixed $p>2$. As we have seen in Section~\ref{SubSectionVorticesVersusPseudoholomorphicMaps}, $A$ gives rise to a symplectic form $\om_{A}$ on the total space $P(\X) = P \xop_{G} \X$, defined by formula (\ref{EquationSymplecticFormOnP(X)}); moreover, $A$ and $A_{\nu}$ give rise to almost complex structures $J_{A}$ and $J_{A_{\nu}}$ on $P(\X)$, defined by formula (\ref{EquationAlmostComplexStructureOnP(X)}). Now by the Sobolev embedding theorem and by Rellich's theorem it follows that, after passing to a subsequence, $A_{\nu}$ converges to $A$ strongly in $C^{0}$ on $\Sig$. Hence we infer from Proposition~\ref{PropositionVorticesAsPseudoholomorphicMaps}\,(i), taking $A$ as reference connection, that both $J_{A}$ and $J_{A_{\nu}}$, for $\nu$ sufficiently large, are tamed by $\om_{A}$.

More generally, for any $\om_{A}$-tame almost complex structure $\Jtilde$ on $P(\X)$ we denote by
\[
\l\cdot,\cdot\r_{\Jtilde} \deq \frac{1}{2} \bigl( \om_{A}(\cdot,\Jtilde\cdot) - \om_{A}(\Jtilde\cdot,\cdot) \bigr)
\]
the Riemannian metric on $\X$ determined by $\om_{A}$ and $\Jtilde$. For $z_{0} \in \C$ and $r>0$ we denote by $B_{r}(z_{0}) \subset \C$ the closed disk of radius $r$ centered at~$z_{0}$. Recall from \cite[Sec.\,2.2]{McDuff/J-holomorphic-curves-and-symplectic-topology} that the energy of a $\Jtilde$-holomorphic curve $\map{\utilde}{B_{r}(z_{0})}{P(\X)}$ is then given by
\[
E_{\Jtilde}\bigl( \utilde,B_{r}(z_{0}) \bigr) \deq \frac{1}{2} \int_{B_{r}(z_{0})} \Abs{\dop\!\utilde}_{\Jtilde}^{2},
\]
where the norm $\Abs{\dop\!\utilde}_{\Jtilde}$ is understood with respect to the metric $\l\cdot,\cdot\r_{\Jtilde}$ on $P(\X)$ and the Euclidean metric on $\C$. Let $B \subset \C$ denote the closed unit disk.

\medskip

The main result of this subsection is the following proposition.

\begin{proposition} \label{PropositionBubblesConnect}
	Fix a holomorphic coordinate chart $\map{\vphi}{B}{\Sig}$, a point $z_{0}\in\C$, and a real number $r_{0} > 0$. Let $A$ be a smooth connection on $P$, and let $A_{\nu}$ be a sequence of smooth connections on $P$ that converges to $A$ weakly in $W^{1,p}$ on $\Sig$, for some fixed $p>2$. Suppose moreover that
\begin{itemize}[itemsep=0.3ex, itemindent=0cm, labelsep=0.3cm, leftmargin=0.7cm]
	\item $\map{u_{\nu}}{\Sig}{P(\X)}$ is a sequence of $J_{A_{\nu}}$-holomorphic sections;
	\item $\mapin{\phi_{\nu}}{B_{r_{0}}(z_{0})}{B}$ is a sequence of injective holomorphic maps;
	\item $\map{\utilde}{B_{r_{0}}(z_{0})}{P(\X)}$ is a $J_{A}$-holomorphic curve
\end{itemize}
such that the following holds.
\begin{enumerate}[topsep=1ex, itemsep=0.3ex, itemindent=0cm, labelsep=1ex, leftmargin=1cm]
	\item[(a)] The sequence $\phi_{\nu}$ is uniformly bounded in $W^{2,\infty}$ on $B_{r_{0}}(z_{0})$.
	\item[(b)] The sequence $\utilde_{\nu} \deq u_{\nu} \circ \vphi \circ \phi_{\nu}$ converges to $\utilde$ in $C^{1}$ on compact subsets of $B_{r_{0}}(z_{0}) \setminus \{z_{0}\}$.
	\item[(c)] The limit
\[
m_{0} \deq \lim_{\veps\to 0}\lim_{\nu\to\infty} E_{J_{A_{\nu}}}\bigl( \utilde_{\nu}; B_{\veps}(z_{0}) \bigr)
\]
exists and is positive.
	\item[(d)] There exist constants $r_{0} > 0$ and $\de,C>0$ such that for every~$\nu$ the section $u_{\nu}$ satisfies a mean value inequality of the following form: For all $z_{0} \in \pc$ and all $0 < r < r_{0}$,
\[
E_{J_{A_{\nu}}} \bigl( u_{\nu};B_{r}(z_{0}) \bigr) < \de \quad \Longrightarrow \quad \Abs{\dop\!u_{\nu}(z_{0})}_{J_{A_{\nu}}}^{2} \le \frac{C}{r^{2}} \cdot E_{J_{A_{\nu}}}\bigl( u_{\nu};B_{r}(z_{0}) \bigr) + C.
\]
\end{enumerate}
Then there exist a sequence of M\"obius transformations $\psi_{\nu}\in\Aut(\proj{1}) \cong \PSL{2}{\C}$, a $J$-holomorphic sphere $\map{v}{\proj{1}}{P(\X)_{\vphi(0)}} \cong \X$ in the fiber of $P(\X)$ over the point $\vphi(0)$, and finitely many distinct points $z_{1},\ldots,z_{\ell}$, $z_{\infty}$ on $\proj{1}$ such that, after passing to a subsequence, the following holds.
\begin{enumerate}[topsep=1ex, itemsep=0.3ex, itemindent=0cm, labelsep=1ex, leftmargin=1cm]
	\item The sequence $\psi_{\nu}$ converges to $z_{0}$ in $C^{\infty}$ on compact subsets of $\proj{1}\setminus\{z_{\infty}\} \cong \C$.
	\item The sequence $v_{\nu} \deq \utilde_{\nu} \circ \psi_{\nu}$ converges to $v$ in $C^{1}$ on compact subsets of
\[
\proj{1}\setminus\{z_{1},\ldots,z_{\ell},z_{\infty}\} \subset \C,\]
and the limits
\[
m_{j} \deq \lim_{\veps\to 0}\lim_{\nu\to\infty} E_{J_{A_{\nu}}}\bigl( v_{\nu};B_{\veps}(z_{j}) \bigr)
\]
exist and are positive for $j=1,\ldots,\ell$.
	\item No energy gets lost in the limit, that is,
\[
E_{J}(v) + \sum_{j=1}^{\ell} m_{j} = m_{0}.
\]
	\item If $v$ is constant then $\ell \ge 2$.
\end{enumerate}
Moreover, bubbles connect in the sense that
\[
\utilde(z_{0}) = v(z_{\infty}),
\]
and, for every $\eps>0$, there exist constants $\de_{0}>0$ and $\nu_{0}$ such that
\[
\dop(z,z_{0}) + \dop\bigl( (\psi_{\nu})^{-1}(z),z_{\infty} \bigr) < \de_{0} \quad \Longrightarrow \quad \dop_{J_{A}}\bigl( \utilde_{\nu}(z),\utilde(z_{0}) \bigr) < \eps
\]
for every $\nu \ge \nu_{0}$ and every $z\in\proj{1}$.
\end{proposition}

\medskip

The remainder of this subsection is devoted to the proof of Proposition~\ref{PropositionBubblesConnect}. It is largely the same as the proof of \cite[Prop.\,4.7.1 and Prop.\,4.7.2]{McDuff/J-holomorphic-curves-and-symplectic-topology} except for certain modifications resulting from the fact that the assumption in Proposition~\ref{PropositionBubblesConnect} on convergence of the almost complex structures is weaker than the respective assumption in \cite[Prop.\,4.7.1 and Prop.\,4.7.2]{McDuff/J-holomorphic-curves-and-symplectic-topology}.

More concretely, we are assuming that the sequence of connections $A_{\nu}$ converges to $A$ weakly in~$W^{1,p}$ for some $p>2$. Hence, after passing to a subsequence, $A_{\nu}$ converges to $A$ strongly in~$C^{0}$. Therefore, we see from formula (\ref{EquationAlmostComplexStructureOnP(X)}) that the almost complex structures $J_{A_{\nu}}$ will in general converge to $J_{A}$ only in $C^{0}$, in contrast to \cite[Prop.\,4.7.1 and Prop.\,4.7.2]{McDuff/J-holomorphic-curves-and-symplectic-topology}, where the sequence~$J_{\nu}$ is assumed to converge to $J$ in $C^{\infty}$. We thus conclude that those arguments in the proofs of \cite[Prop.\,4.7.1 and Prop.\,4.7.2]{McDuff/J-holomorphic-curves-and-symplectic-topology} that rely on uniform estimates involving the derivatives of the almost complex structures~$J_{\nu}$ will not carry over to our situation without modification. There are basically two types of such arguments: elliptic bootstrapping for rescaled $J_{\nu}$-holomorphic curves on the one hand, and any argument involving a uniform mean value inequality for sequences of $J_{\nu}$-holomorphic curves based on \cite[Lemma 4.3.1]{McDuff/J-holomorphic-curves-and-symplectic-topology} on the other hand. In fact, a careful examination of the proof of \cite[Lemma~4.3.1]{McDuff/J-holomorphic-curves-and-symplectic-topology} reveals that the constant $\de$ in the statement of this lemma depends on the first and second derivatives of the almost complex structure (see also the comments after \cite[Lemma 4.7.3]{McDuff/J-holomorphic-curves-and-symplectic-topology}).

We now discuss in detail how to modify those critical arguments in order to make them work under our assumptions as well.

\smallskip

First, we note that elliptic bootstrapping for rescaled $J_{\nu}$-holomorphic curves enters precisely into Steps 2 and 3 of the proof of \cite[Prop.\,4.7.1]{McDuff/J-holomorphic-curves-and-symplectic-topology}. More precisely, it enters via \cite[Lemma~4.6.5 and Thm.\,4.6.1]{McDuff/J-holomorphic-curves-and-symplectic-topology}, the proofs of which are in turn based on the basic compactness theorem \cite[Thm.\,4.1.1]{McDuff/J-holomorphic-curves-and-symplectic-topology}. When adapted to our situation this argument essentially boils down to elliptic bootstrapping for the sequence of rescaled $J_{A_{\nu}}$-holomorphic curves
\[
v_{\nu} \deq \utilde_{\nu} \circ \psi_{\nu}, \quad \psi_{\nu}(z) = \de^{\nu} z, \quad \de^{\nu} \to 0
\]
in Step 2 and
\[
w_{\nu}(z) \deq \utilde_{\nu}(\veps^{\nu} z), \quad \veps^{\nu} \to 0
\]
in Step 3 of the proof of \cite[Prop.\,4.7.1]{McDuff/J-holomorphic-curves-and-symplectic-topology}---here we assume that $z_{0}=0$ by Step 1 of that proof. The key idea now is to exploit the fact that the curves $\utilde_{\nu} = u_{\nu} \circ \vphi \circ \phi_{\nu}$ factor through the $J_{A_{\nu}}$-holomorphic sections $\map{u_{\nu}}{\Sig}{P(\X)}$. This will eventually provide us with certain perturbed $J$-holomorphic curve equations for the curves $v_{\nu}$ and $w_{\nu}$ to which standard elliptic bootstrapping arguments apply. Basically, we will follow the bubbling argument for vortices from the proof of \cite[Thm.\,3.4]{Cieliebak/The-symplectic-vortex-equations-and-invariants-of-Hamiltonian-group-actions}.

To start with, we note that by Lemma \ref{LemmaVorticesAsPseudoholomorphicMaps} the $J_{A_{\nu}}$-holomorphic sections~$u_{\nu}$ of the bundle~$P(\X)$ satisfy the first vortex equation
\[
\delbar_{J,A_{\nu}}(u_{\nu}) = \frac{1}{2} \, \bigl( \dop_{A_{\nu}}\!u_{\nu} + J(u_{\nu}) \circ \dop_{A_{\nu}}\!u_{\nu} \circ j_{\Sig} \bigr) = 0
\]
when considered as $G$-equivariant maps $\map{u_{\nu}}{P}{\X}$ as in Remark \ref{RemarkMapVersusSection}. By Remark \ref{RemarkLocalFromGlobal}, locally in the chart $\map{\vphi}{B}{\Sig}$ this equation takes the form
\begin{eqnarray*} \label{EquationFirstVortexEquationBubbling}
	\del_{s}\!u_{\nu}^{\loc}  + J\bigl( u_{\nu}^{\loc} \bigr) \del_{t}\!u_{\nu}^{\loc} = - X_{\Phi_{\nu}}\bigl( u_{\nu}^{\loc} \bigr) - J\bigl( u_{\nu}^{\loc} \bigr) X_{\Psi_{\nu}}\bigl( u_{\nu}^{\loc} \bigr).
\end{eqnarray*}
Here $\vphitilde^{\,\ast}A_{\nu} = \Phi_{\nu} \dop\!s + \Psi_{\nu} \dop\!t$ and $u_{\nu}^{\loc} \deq u_{\nu} \circ \vphitilde$, where $\map{\vphitilde}{B}{P}$ is some lift of $\vphi$. A straightforward calculation as in \cite[App.\,B.2]{Ziltener/Symplectic-vortices-on-the-complex-plane-and-quantum-cohomology} then shows that the rescaled curves $v_{\nu} = \utilde_{\nu} \circ \psi_{\nu}$ satisfy the equation
\begin{equation} \label{EquationRescaledFirstVortexEquation}
	\hspace{1cm} \del_{s}\!v_{\nu} + J(v_{\nu}) \del_{t}\!v_{\nu} \\ = - \de^{\nu} \, \overline{\phi_{\nu}^{\prime}} \cdot \bigl( X_{\Phi_{\nu} \circ \phi_{\nu} \circ \psi_{\nu}}(v_{\nu}) + J(v_{\nu}) X_{\Psi_{\nu} \circ \phi_{\nu} \circ \psi_{\nu}}(v_{\nu}) \bigr).
\end{equation}
Here $\overline{\phi_{\nu}^{\prime}}$ denotes the complex conjugate of the derivative $\partial_{z}\phi_{\nu}$ of the holomorphic map $\phi_{\nu}$, and the product on the right-hand side is defined by
\[
(s + it) \cdot w \deq s \cdot w + t \cdot J(x)w
\]
for all numbers $s + it \in \C$ and all tangent vectors $w \in T_{x}\X$ for $x \in \X$.

The elliptic bootstrapping for the sequence $v_{\nu}$ is then based on equation (\ref{EquationRescaledFirstVortexEquation}), as follows. Assumption (a) of Proposition~\ref{PropositionBubblesConnect} provides a uniform $W^{2,\infty}$-bound for the sequence $\phi_{\nu}$, and $\de^{\nu} \to 0$, so the first factor on the right-hand side of equation (\ref{EquationRescaledFirstVortexEquation}) is uniformly bounded in~$W^{1,\infty}$. Likewise, the assumption of Proposition~\ref{PropositionBubblesConnect} provides a uniform $W^{1,p}$-bound, $p>2$, for the sequence of connections $A_{\nu}$, whence the functions $\Phi_{\nu}$ and $\Psi_{\nu}$ are both uniformly bounded in~$W^{1,p}$ on $B$. Again by assumption (a) of Proposition~\ref{PropositionBubblesConnect} it follows that the functions $\Phi_{\nu} \circ \phi_{\nu} \circ \psi_{\nu}$ and $\Psi_{\nu} \circ \phi_{\nu} \circ \psi_{\nu}$ are uniformly bounded in $W^{1,p}$ on compact subsets of $\C$. By construction, the sequence $v_{\nu}$ is uniformly bounded in $W^{1,\infty}$ on a certain compact subset $K \subset \C$ depending on whether we are considering the proof of \cite[Lemma 4.6.5]{McDuff/J-holomorphic-curves-and-symplectic-topology} or \cite[Thm.\,4.6.1]{McDuff/J-holomorphic-curves-and-symplectic-topology}. Hence the second factor on the right-hand side of equation (\ref{EquationRescaledFirstVortexEquation}) satisfies a uniform $W^{1,p}$-bound on $K$. We conclude that the right-hand side of equation (\ref{EquationRescaledFirstVortexEquation}) is uniformly bounded in $W^{1,p}$ on $K$. Hence elliptic regularity implies that the sequence $v_{\nu}$ is uniformly bounded in $W^{2,p}$ on $K$ (see \cite[App.\,B.4]{McDuff/J-holomorphic-curves-and-symplectic-topology}). By the Banach-Alaoglu theorem, the Sobolev embedding theorem, and Rellich's theorem it follows that, after passing to a subsequence, the sequence $v_{\nu}$ converges weakly in~$W^{2,p}$ and strongly in $C^{1}$ on $K$ to a $J$-holomorphic curve $\map{v}{K}{\X}$ satisfying the equation
\[
\del_{s}\!v + J(v) \del_{t}\!v = 0,
\]
which is obtained from (\ref{EquationRescaledFirstVortexEquation}) in the limit $\nu \to \infty$.

The argument for the curves $w_{\nu}$ is identical.

Observe that, in general, we cannot expect better convergence than $C^{1}$ for $v_{\nu}$ and $w_{\nu}$ since the sequence $A_{\nu}$ is only assumed to be bounded in $W^{1,p}$ and hence the bootstrapping terminates after just one step. However, this is sufficient for all subsequent arguments in the proof of \cite[Prop.\,4.7.1]{McDuff/J-holomorphic-curves-and-symplectic-topology}.

Note also that the above argument shows that the sequence of curves $v_{\nu} = \utilde_{\nu} \circ \psi_{\nu}$ converges to a $J$-holomorphic sphere $\map{v}{\proj{1}}{P(\X)_{\vphi(0)} \cong \X}$ in the fiber of $P(\X)$ over the bubbling point $\vphi(0)$.

\smallskip

Second, we investigate all arguments in the proof of \cite[Prop.\,4.7.1 and Prop.\,4.7.2]{McDuff/J-holomorphic-curves-and-symplectic-topology} that rely on a uniform mean value inequality for sequences of $J_{\nu}$-holomorphic curves based on \cite[Lemma 4.3.1]{McDuff/J-holomorphic-curves-and-symplectic-topology}. Our strategy will be to deduce all those mean value inequalities not from the mean value inequality provided by \cite[Lemma 4.3.1]{McDuff/J-holomorphic-curves-and-symplectic-topology} but from the mean value inequality for the sections~$u_{\nu}$ provided by assumption (d) of Proposition~\ref{PropositionBubblesConnect}. Note that this mean value inequality is slightly weaker than the mean value inequality of \cite[Lemma~4.3.1]{McDuff/J-holomorphic-curves-and-symplectic-topology} since it contains an additive constant~$C$.

Now the only step in the proof of \cite[Prop.\,4.7.1 and Prop.\,4.7.2]{McDuff/J-holomorphic-curves-and-symplectic-topology} where \cite[Lemma~4.3.1]{McDuff/J-holomorphic-curves-and-symplectic-topology} is used is in the proof of \cite[Lemma~4.7.3]{McDuff/J-holomorphic-curves-and-symplectic-topology}. In order to complete the proof of Proposition~\ref{PropositionBubblesConnect} we will therefore first prove a variant of \cite[Lemma~4.7.3]{McDuff/J-holomorphic-curves-and-symplectic-topology}, see Lemma \ref{LemmaAnnulusLemma} below, that relies on the mean value inequality from assumption (d) instead of the mean value inequality from \cite[Lemma~4.3.1]{McDuff/J-holomorphic-curves-and-symplectic-topology}. We will state this lemma more generally for any closed symplectic manifold, which we will denote by $(\X,\om)$ by abuse of notation. Moreover, for $r < R$ we denote by $A(r,R) \deq \{ z\in\C \,|\, r \le \Abs{z} \le R \}$ the closed annulus in $\C$ of inner radius $r$ and outer radius $R$ centered at the origin.

\begin{lemma} \label{LemmaAnnulusLemma}
	Let $(\X,\om)$ be a closed symplectic manifold and assume that $J$ is an $\om$-tame almost complex structure on $\X$. Fix constants $\de,C>0$. Then, for every $0<\mu<1$, there exist constants $R_{0}>0$, $\de_{0} \deq \de_{0}(\de,C,\mu) > 0$ and $c \deq c(C,\mu) > 0$ such that the following holds.

Suppose that $0 < r < R < R_{0}$ with $R/r \ge 4 e^{2}$, and that $\map{u}{A(r,R)}{\X}$ is a $J$-holomorphic curve that satisfies a mean value inequality of the following form: For all $z\in A(r,R)$ and all~$\rho>0$ such that $B_{\rho}(z) \subset A(r,R)$,
\begin{equation} \label{InequalityMeanValueInequalityAnnulusLemma}
	E_{J}\bigl( u;B_{\rho}(z) \bigr) < \de \quad \Longrightarrow \quad \frac{1}{2} \Abs{\dop\!u(z)}^{2}_{J} \le \frac{C}{\rho^{2}} \cdot E_{J}\bigl( u;B_{\rho}(z) \bigr) + C.
\end{equation}
Here the norm $\Abs{\dop\!u}_{J}$ is understood with respect to the Riemannian metric $\l\cdot,\cdot\r_{J}$ on $\X$ determined by $\om$ and $J$, and the Euclidean metric on $\C$. Then, if the energy of $u$ is sufficiently small in the sense that
\[
E_{J}(u) \deq E_{J}\bigl( u;A(r,R) \bigr) < \de_{0},
\]
we have estimates
\begin{eqnarray} \label{InequalityFirstEstimateAnnulusLemma}
	E_{J}\bigl( u;A(e^{T}r,e^{-T}R) \bigr) \le c \cdot e^{-2\mu T} \cdot E_{J}(u)
\end{eqnarray}
and
\begin{eqnarray} \label{InequalitySecondEstimateAnnulusLemma}
	\sup_{z_{1},z_{2}\in A(e^{T}r,e^{-T}R)} \dop_{J}\bigl( u(z_{1}),u(z_{2}) \bigr) \le c \cdot \Bigl( e^{-\mu T} \cdot \sqrt{E_{J}(u)} + R \Bigr)
\end{eqnarray}
for all $T$ such that $\log2 \le T \le \log\sqrt{R/r}$. Here $\dop_{J}$ denotes the distance function on $\X$ induced by the metric $\l\cdot,\cdot\r_{J}$.
\end{lemma}

\begin{proof}
	The proof is adapted from the proof of \cite[Lemma 4.7.3]{McDuff/J-holomorphic-curves-and-symplectic-topology}. Fix constants $\de,C > 0$ and $0<\mu<1$, and define $c^{\prime} \deq c^{\prime}(\mu) \deq 1/4\pi\mu$.

We first recall some notation from \cite[Sec.\,4.4]{McDuff/J-holomorphic-curves-and-symplectic-topology}. For any smooth loop $\map{\ga}{\del\!B}{\X}$, where~$B$ denotes the closed unit disk in $\C$, we denote by $\ell(\ga)$ its length with respect to the metric $\l\cdot,\cdot\r_{J}$. If $\ell(\ga)$ is smaller than the injectivity radius of $\X$, then $\ga$ admits a smooth local extension $\map{u_{\ga}}{B}{\X}$ such that $u_{\ga}\bigl( e^{i\th} \bigr) = \ga(\th)$ for every $\th\in[0,2\pi]$ and the image of $u_{\ga}$ is contained in a geodesic ball of radius not greater than half the injectivity radius. In this case the local symplectic action of $\ga$ is given by
\begin{eqnarray} \label{EqualityLocalSymplecticActionLemma}
	a(\ga) \deq - \int_{B} u_{\ga}^{\ast}\,\om.
\end{eqnarray}
Note that it does not depend on the choice of the extension $u_{\ga}$. Recall that $c^{\prime}>1/4\pi$ by assumption. Hence by the isoperimetric inequality from \cite[Thm.\,4.4.1]{McDuff/J-holomorphic-curves-and-symplectic-topology} there exists a constant $\de_{0} \deq \de_{0}(\de,C,\mu) > 0$ such that
\begin{eqnarray} \label{InequalityAssumptionOnDelta_0}
	\de_{0} \le \de
\end{eqnarray}
and
\begin{eqnarray} \label{InequalityIsoperimetricInequalityBubblesConnect}
	\ell(\ga) < 4\pi \sqrt{2\,C \de_{0}} \quad \Longrightarrow \quad \Abs{a(\ga)} \le c^{\prime} \cdot \ell(\ga)^{2}
\end{eqnarray}
for every smooth loop $\map{\ga}{\del\!B}{\X}$. We are now ready for the actual proof of the lemma.

\smallskip

Assume that $E_{J}(u) \deq E_{J}\bigl( u;A(r,R) \bigr) < \de_{0}$.

\smallskip

\noindent{Proof of (\ref{InequalityFirstEstimateAnnulusLemma}):} For $r \le \rho \le R$ let $\map{\ga_{\rho}}{\del\!B}{\X}$ be the loop defined by $\ga_{\rho}(\th) \deq u(\rho\,e^{i\th})$ for $\th\in[0,2\pi]$. Furthermore, for $\log2 \le t \le \log\sqrt{R/r}$ we define a smooth function $t \mapsto \veps(t)$ by
\begin{eqnarray} \label{EqualityFunctionEpsilon}
	\veps(t) \deq E_{J}\Bigl( u;A\bigl( e^{t}r,e^{-t}R \bigr) \Bigr) = \frac{1}{2} \int_{A(e^{t}r,e^{-t}R)} \Abs{\dop\!u}_{J}^{2}.
\end{eqnarray}
It will be useful to keep in mind that the condition $\log2 \le t \le \log\sqrt{R/r}$ is equivalent to $2r \le e^{t}r \le e^{-t}R \le R/2$. Fix a number $T$ such that
\[
\log2 \le T \le \log\sqrt{R/r}
\]
and let $\rho$ such that $2r \le e^{T}r \le \rho \le e^{-T}R \le R/2$. Then, for any $\th\in[0,2\pi]$, the disk $B_{\rho/2}(\rho e^{i\th})$ is contained in the annulus $A(r,R)$. Since $E_{J}(u) < \de_{0} \le \de$ by the assumption of the lemma and by (\ref{InequalityAssumptionOnDelta_0}), the mean value inequality (\ref{InequalityMeanValueInequalityAnnulusLemma}) yields
\begin{eqnarray}\label{InequalityAPrioriEstimateAnnulusLemmaBubblesConnectRevisited}
	\hspace{1cm} \frac{1}{2} \Abs{\dop\!u\bigl( \rho e^{i\th} \bigr)}^{2}_{J} \le \frac{4\,C}{\rho^{2}} \cdot E_{J}(u;B_{\rho/2}\bigl(\rho e^{i\th}\bigr) \bigr) + C \le \frac{4\,C}{\rho^{2}} \cdot E_{J}(u) + C.
\end{eqnarray}
Hence
\[
\Abs{\dot{\ga_{\rho}}(\th)}_{J} = \frac{\rho}{\sqrt{2}} \cdot \Abs{\dop\!u\bigl( \rho e^{i\th} \bigr)}_{J} \le 2 \sqrt{C (E_{J}(u) + \rho^{2})} < 2 \sqrt{C (\de_{0} + \rho^{2})}.
\]
Now define $R_{0} \deq \sqrt{\de_{0}}$ and assume for the remainder of this proof that $R < R_{0}$. Then $\rho^{2} < \de_{0}$, and the previous estimate implies that
\begin{eqnarray*}
	\ell(\ga_{\rho}) = \int_{0}^{2\pi} \Abs{\dot{\ga_{\rho}}(\th)}_{J} \dop\!\th < 4\pi \sqrt{2\,C \de_{0}}.
\end{eqnarray*}
It then follows from the isoperimetric inequality~(\ref{InequalityIsoperimetricInequalityBubblesConnect}) that
\begin{eqnarray} \label{InequalityConclusionOfIsoperimetricInequality}
	\Abs{a(\ga_{\rho})} \le c^{\prime} \cdot \ell(\ga_{\rho})^{2}.
\end{eqnarray}
As in \cite[Rmk.\,4.4.2]{McDuff/J-holomorphic-curves-and-symplectic-topology} we denote by $\map{u_{\rho}}{B}{\X}$ the local extension of the loop $\ga_{\rho}$ defined by the formula $u_{\rho}(\rho^{\prime} e^{i\th}) \deq \exp_{\ga_{\rho}(0)}(\rho^{\prime} \, \xi(\th))$ for $0 < \rho^{\prime} < \rho$ and $\th\in[0,2\pi]$, where the map $\map{\xi}{[0,2\pi]}{T_{\ga_{\rho}(0)}\X}$ is determined by the condition $\exp_{\ga_{\rho}(0)}(\xi(\th)) = \ga_{\rho}(\th)$. For $\log2 \le t \le \log\sqrt{R/r}$ consider the sphere $\map{v_{t}}{S^{2}}{\X}$ that is obtained from the restriction of the map~$u$ to the annulus $A(e^{t}r,e^{-t}R)$ by filling in the boundary circles $\ga_{e^{t}r}$ and $\ga_{e^{-t}R}$ with the local extensions $u_{e^{t}r}$ and $u_{e^{-t}R}$. The sphere $\map{v_{t}}{S^{2}}{\X}$ is contractible because it is the boundary of the 3-ball consisting of the union of the disks $\map{u_{\rho^{\prime}}}{B}{\X}$ for $e^{t}r \le \rho^{\prime} \le e^{-t}R$, whence
\[
0 = \int_{S^{2}} v_{t}^{\ast}\om = \int_{A(e^{t}r,e^{-t}R)} u^{\ast}\om - \int_{B} u_{e^{t}r}^{\ast}\,\om + \int_{B} u_{e^{-t}R}^{\ast}\,\om.
\]
Using the energy identity \cite[Lemma 2.2.1]{McDuff/J-holomorphic-curves-and-symplectic-topology}, we may write this equality in terms of the function~(\ref{EqualityFunctionEpsilon}) and the local symplectic action (\ref{EqualityLocalSymplecticActionLemma}) as
\[
\veps(t) = - a(\ga_{e^{t}r}) + a(\ga_{e^{-t}R}).
\]

\noindent Thus, applying inequality~(\ref{InequalityConclusionOfIsoperimetricInequality}) and H\"older's inequality we obtain
\begin{eqnarray*}
	\veps(t) &\le& c^{\prime} \cdot \ell(\ga_{e^{t}r})^{2} + c^{\prime} \cdot \ell(\ga_{e^{-t}R})^{2} \\
	&=& c^{\prime} \cdot \left( \int_{0}^{2\pi} \Abs{\dot{\ga}_{e^{t}r}(\th)}_{J} \dop\!\th \right)^{2} + c^{\prime} \cdot \left( \int_{0}^{2\pi} \Abs{\dot{\ga}_{e^{-t}R}(\th)}_{J} \dop\!\th \right)^{2} \\
	&=& \frac{c^{\prime} \bigl(e^{t}r\bigr)^{2}}{2} \cdot \left( \int_{0}^{2\pi} \Abs{\dop\!u\bigl(e^{t}r e^{i\th}\bigr)}_{J} \dop\!\th \right)^{2} + \frac{c^{\prime} \bigl(e^{-t}R\bigr)^{2}}{2} \cdot \left( \int_{0}^{2\pi} \Abs{\dop\!u\bigl(e^{-t}R e^{i\th}\bigr)}_{J} \dop\!\th \right)^{2} \\
	&\le& 2\pi c^{\prime} \cdot \Bigg( \frac{1}{2} \bigl(e^{t}r\bigr)^{2} \cdot \int_{0}^{2\pi} \Abs{\dop\!u\bigl(e^{t}r e^{i\th}\bigr)}_{J}^{2} \dop\!\th + \frac{1}{2} \bigl(e^{-t}R\bigr)^{2} \cdot \int_{0}^{2\pi} \Abs{\dop\!u\bigl(e^{-t}R e^{i\th}\bigr)}_{J}^{2} \dop\!\th \Bigg).
\end{eqnarray*}
To estimate this further, recall that
\[
\veps(t) = \frac{1}{2} \int_{A(e^{t}r,e^{-t}R)} \Abs{\dop\!u}_{J}^{2} = \frac{1}{2} \int_{e^{t}r}^{e^{-t}R} \rho \int_{0}^{2\pi} \Abs{\dop\!u\bigl( \rho e^{i\th} \bigr)}_{J}^{2} \dop\!\th \dop\!\rho,
\]
whence
\[
\dot\veps(t) = -\frac{1}{2} (e^{t}r)^{2} \int_{0}^{2\pi} \Abs{\dop\!u\bigl( e^{t}r e^{i\th} \bigr)}_{J}^{2} \dop\!\th - \frac{1}{2} (e^{-t}R)^{2} \int_{0}^{2\pi} \Abs{\dop\!u\bigl( e^{-t}R e^{i\th} \bigr)}_{J}^{2} \dop\!\th.
\]
We conclude that
\[
\veps(t) \le -2\pi c^{\prime} \cdot \dot{\veps}(t),
\]
which implies
\[
\dot{\veps}(t) \le -2\mu \cdot \veps(t) < 0.
\]
Integrating this differential inequality from $\log2$ to $T$ yields
\begin{eqnarray} \label{InequalityFirstInequalityAnnulusLemma}
	\veps(T) \le e^{-2\mu (T-\log2)} \cdot \veps(\log2) \le e^{-2\mu T} \cdot e^{2\mu} \cdot E_{J}(u).
\end{eqnarray}
Since $\mu > 0$, inequality (\ref{InequalityFirstEstimateAnnulusLemma}) follows.

\smallskip

\noindent{Proof of (\ref{InequalitySecondEstimateAnnulusLemma}):} Let us denote $\rho_{0} \deq \sqrt{rR}$. The assumption $R/r \ge 4 e^{2}$ then implies that $2r \le \rho_{0} \le R/2$. We begin with the following observation.

\begin{claim*}
	The map $u$ satisfies the following estimates, where $\rho_{0} \deq \sqrt{rR}$.
\begin{enumerate}[topsep=1ex, itemsep=0.3ex, itemindent=0cm, labelsep=1ex, leftmargin=1cm]
	\item If $2r \le \rho \le \rho_{0}$, then
\begin{eqnarray*}
	\frac{1}{2} \Abs{\dop\!u\bigl( \rho e^{i\th} \bigr)}^{2}_{J} \le \frac{36\,C\,e^{2\mu}}{\rho^{2}} \cdot \left( \frac{r}{\rho} \right)^{2\mu} \cdot E_{J}(u) + C.
\end{eqnarray*}
	\item If $\rho_{0} \le \rho \le R/2$, then
\begin{eqnarray*}
	\frac{1}{2} \Abs{\dop\!u\bigl( \rho e^{i\th} \bigr)}^{2}_{J} \le \frac{36\,C\,e^{2\mu}}{\rho^{2}} \cdot \left( \frac{\rho}{R} \right)^{2\mu} \cdot E_{J}(u) + C.
\end{eqnarray*}
\end{enumerate}
\end{claim*}

\begin{proof}[Proof of Claim]
	First of all, note that the assumption $R/r \ge 4 e^{2}$ yields the finer estimate $2r \le 2e \, r \le \rho_{0} \le R/2e \le R/2$.

In order to prove (i) assume that $2r \le \rho \le \rho_{0}$.  We may then distinguish two cases.

\smallskip

\fall{1}{$2r \le \rho \le 2e \, r$} Then $B_{\rho/2}(\rho e^{i\th}) \subset A(r,R)$. Since $E_{J}(u) < \de_{0} \le \de$ by the assumption of the lemma and by (\ref{InequalityAssumptionOnDelta_0}), the mean value inequality (\ref{InequalityMeanValueInequalityAnnulusLemma}) yields
\[
\frac{1}{2} \Abs{\dop\!u\bigl( \rho e^{i\th} \bigr)}^{2}_{J} \le \frac{4\,C}{\rho^{2}} \cdot \left( \frac{\rho}{r} \right)^{2\mu} \cdot \left( \frac{r}{\rho} \right)^{2\mu} \cdot E_{J}(u) + C \le \frac{16\,C\,e^{2\mu}}{\rho^{2}} \cdot \left( \frac{r}{\rho} \right)^{2\mu} \cdot E_{J}(u) + C.
\]
In the second inequality we used that $\rho/r \le 2e$ and hence $(\rho/r)^{2\mu} \le 4e^{2\mu}$.

\smallskip

\fall{2}{$2e \, r \le \rho \le \rho_{0}$} Then $B_{\rho/2}(\rho e^{i\th}) \subset A(\rho/e,e\rho) \subset A(r,R)$. Again inequality (\ref{InequalityMeanValueInequalityAnnulusLemma}) yields
\[
\frac{1}{2} \Abs{\dop\!u\bigl( \rho e^{i\th} \bigr)}^{2}_{J} \le \frac{4\,C}{\rho^{2}} \cdot E_{J}\bigl( u; A(\rho/e,e\rho) \bigr) + C.
\]
Applying inequality~(\ref{InequalityFirstInequalityAnnulusLemma}) to the annulus $A(e^{\log(\rho/r)-1} \cdot r, e^{-\log(\rho/r)+1} \cdot R) \supset A(\rho/e,e\rho)$ we get
\begin{multline*}
	E_{J}\bigl( u; A(\rho/e,e\rho) \bigr) \le E_{J}\Bigl( u;A\bigl( e^{\log(\rho/r)-1} \cdot r, e^{-\log(\rho/r)+1} \cdot R \bigr) \Bigr) \\ \le e^{2\mu} \cdot e^{-2\mu (\log(\rho/r)-1)} \cdot E_{J}(u) \le 9 \, e^{2\mu} \cdot \left( \frac{r}{\rho} \right)^{2\mu} \cdot E_{J}(u).
\end{multline*}
Hence
\[
\frac{1}{2} \Abs{\dop\!u\bigl( \rho e^{i\th} \bigr)}^{2}_{J} \le \frac{36\,C\,e^{2\mu}}{\rho^{2}} \cdot \left( \frac{r}{\rho} \right)^{2\mu} \cdot E_{J}(u) + C.
\]
This proves (i). The proof of (ii) is similar and will therefore be omitted.
\end{proof}

\medskip

Continuing with the proof of inequality (\ref{InequalitySecondEstimateAnnulusLemma}), fix $T$ such that $\log2 \le T \le \log\sqrt{R/r}$. Note that this implies $2r \le e^{T}r \le \rho_{0} \le e^{-T}R \le R/2$. Suppose that $z_{1},z_{2}\in A(e^{T}r,e^{-T}R)$. Then we have
\begin{eqnarray} \label{InequalityLengthEstimateAnnulusLemma}
	\dop_{J}\bigl( u(z_{1}),u(z_{2}) \bigr) \le \dop_{J}\bigl( u(z_{1}),u(\rho_{0}) \bigr) + \dop_{J}\bigl( u(\rho_{0})),u(z_{2}) \bigr).
\end{eqnarray}
In order to estimate the terms on the right-hand side of this inequality we write $z_{j} = \rho_{j} e^{i \th_{j}}$, $j=1,2$, with $e^{T}r \le \rho_{j} \le e^{-T}R$ and $\th_{j}\in[0,2\pi]$. We distinguish two cases.

\smallskip

\fall{1}{$e^{T}r \le \rho_{j} \le \rho_{0}$} To start with we estimate
\[
\dop_{J}\bigl( u(z_{j}),u(\rho_{0}) \bigr) \le \int_{\rho_{j}}^{\rho_{0}} \Abs{\del_{\rho}\!u(\rho)}_{J} \dop\!\rho \,+\, \int_{0}^{\th_{j}} \Abs{\del_{\th}\!u\bigl( \rho_{j}e^{i \th} \bigr)}_{J} \dop\!\th.
\]
Note that
\[
\Abs{\del_{\rho}\!u(\rho)}_{J} = \frac{1}{\sqrt{2}} \Abs{\dop\!u(\rho)}_{J} \quad \text{and} \quad \Abs{\del_{\th}\!u\bigl( \rho_{j}e^{i \th} \bigr)}_{J} = \frac{\rho_{j}}{\sqrt{2}} \cdot \Abs{\dop\!u\bigl( \rho_{j}e^{i \th} \bigr)}_{J}.
\]
Furthermore, since $\mu < 1$ we have $e^{2\mu} \le 9$, so it follows from assertion (i) of the Claim above that
\[
\frac{1}{\sqrt{2}} \Abs{\dop\!u\bigl( \rho e^{i\th} \bigr)}_{J} \le 18\,\sqrt{C} \cdot \frac{r^{\mu}}{\rho^{\,\mu+1}} \cdot \sqrt{E_{J}(u)} + \sqrt{C}
\]
for $2r \le \rho \le \rho_{0}$. Hence we obtain
\begin{multline*}
	\dop_{J}\bigl( u(z_{j}),u(\rho_{0}) \bigr) \le 18 \, \sqrt{C} \cdot \left( \int_{\rho_{j}}^{\rho_{0}} \frac{r^{\mu}}{\rho^{\,\mu+1}} \dop\!\rho + \int_{0}^{\th_{j}} \left( \frac{r}{\rho_{j}} \right)^{\mu} \dop\!\th \right) \cdot \sqrt{E_{J}(u)} \\ + \bigl( \rho_{0} - \rho_{j} + \rho_{j}\th_{j} \bigr) \cdot \sqrt{C}.
\end{multline*}
Using $2r \le e^{T}r \le \rho_{j} \le \rho_{0} \le R/2$ we may estimate the terms on the right-hand side by
\[
\int_{\rho_{j}}^{\rho_{0}} \frac{r^{\mu}}{\rho^{\,\mu+1}} \dop\!\rho \le \int_{e^{T}r}^{\rho_{0}} \frac{r^{\mu}}{\rho^{\,\mu+1}} \dop\!\rho = -\frac{1}{\mu} \cdot \left( \frac{r}{\rho_{0}} \right)^{\mu} + \frac{r^{\mu}}{\mu \cdot (e^{T}r)^{\,\mu}} \le \frac{1}{\mu} \cdot e^{-\mu T}
\]
and
\[
\int_{0}^{\th_{j}} \left( \frac{r}{\rho_{j}} \right)^{\mu} \dop\!\th \le \int_{0}^{2\pi} \left( \frac{r}{e^{T}r} \right)^{\mu} \dop\!\th = 2\pi \cdot e^{-\mu T}
\]
and
\[
\rho_{0} - \rho_{j} + \rho_{j}\th_{j} \le (1 + 2\pi) \, R,
\]
which finally yields
\begin{equation} \label{InequalityFirstTermSecondPartOfAnnulusLemma}
	\dop_{J}\bigl( u(z_{j}),u(\rho_{0}) \bigr) \le 18\,\sqrt{C} \left( \frac{1}{\mu} + 2\pi \right) \cdot e^{-\mu T} \cdot \sqrt{E_{J}(u)} + \sqrt{C} \, (1 + 2\pi) \, R.
\end{equation}

\smallskip

\fall{2}{$\rho_{0} \le \rho_{j} \le e^{-T}R$} A similar argument, which uses assertion (ii) of the Claim above, leads to the same estimate as in (\ref{InequalityFirstTermSecondPartOfAnnulusLemma}).

\smallskip

Plugging inequality (\ref{InequalityFirstTermSecondPartOfAnnulusLemma}) into estimate (\ref{InequalityLengthEstimateAnnulusLemma}), inequality (\ref{InequalitySecondEstimateAnnulusLemma}) follows. This completes the proof of Lemma \ref{LemmaAnnulusLemma}.
\end{proof}

Note that Lemma~\ref{LemmaAnnulusLemma} differs from \cite[Lemma 4.7.3]{McDuff/J-holomorphic-curves-and-symplectic-topology}. In particular, the constant $\de_{0}$ appearing in Lemma~\ref{LemmaAnnulusLemma} does not depend on the almost complex structure $J$. This will be crucial in the subsequent applications.

\smallskip

In order to complete the proof of Proposition~\ref{PropositionBubblesConnect} we shall now modify those parts of the proofs of \cite[Prop.\,4.7.1 and Prop.\,4.7.2]{McDuff/J-holomorphic-curves-and-symplectic-topology} that rely on \cite[Lemma 4.7.3]{McDuff/J-holomorphic-curves-and-symplectic-topology}. We revert to the notation we were using in the proof of Proposition~\ref{PropositionBubblesConnect} before stating Lemma~\ref{LemmaAnnulusLemma}. Throughout we will use without explicit mention that $J_{A_{\nu}}$ converges to $J_{A}$ in $C^{0}$ on $P(\X)$, as follows from formula (\ref{EquationAlmostComplexStructureOnP(X)}) since $A_{\nu}$ converges to $A$ in $C^{0}$.

\smallskip

Let us examine the proof of \cite[Prop.\,4.7.2]{McDuff/J-holomorphic-curves-and-symplectic-topology} first. This proof is based on \cite[Lemma~4.7.4]{McDuff/J-holomorphic-curves-and-symplectic-topology}, and \cite[Lemma 4.7.3]{McDuff/J-holomorphic-curves-and-symplectic-topology} enters via the proof of \cite[Lemma~4.7.4]{McDuff/J-holomorphic-curves-and-symplectic-topology}. We will therefore explain the modifications to the proof of \cite[Lemma~4.7.4]{McDuff/J-holomorphic-curves-and-symplectic-topology} that are required when \cite[Lemma 4.7.3]{McDuff/J-holomorphic-curves-and-symplectic-topology} is replaced by Lemma~\ref{LemmaAnnulusLemma}.

We need to apply Lemma~\ref{LemmaAnnulusLemma} to the $J_{A_{\nu}}$-holomorphic curves
\begin{eqnarray} \label{MapRescaledCurvesOnAnnulus}
	\map{w_{\nu} \deq \utilde_{\nu} \circ \bigl( \rho_{0}^{\nu} \bigr)^{-1}}{A(\de^{\nu}/\rho,\rho)}{P(\X)},
\end{eqnarray}
where $\utilde_{\nu} = u_{\nu} \circ \vphi \circ \phi_{\nu}$ and $\rho_{0}^{\nu}$ denotes the sequence of conformal maps introduced in Step 2 in the proof of \cite[Prop.\,4.7.2]{McDuff/J-holomorphic-curves-and-symplectic-topology}. Before explaining the required modifications to the proof of \cite[Lemma~4.7.4]{McDuff/J-holomorphic-curves-and-symplectic-topology} we check that the assumptions of Lemma~\ref{LemmaAnnulusLemma} are satisfied.

In fact, we are considering the limit $\rho \to 0$, so we may without loss of generality assume that~$\rho < R_{0}$ for some constant $R_{0}>0$. Moreover, for fixed $\rho$ the ratio $\rho / (\de^{\nu}/\rho)$ gets arbitrarily large since $\de^{\nu} \to 0$. It remains to show that the curves $w_{\nu}$ satisfy a uniform mean value inequality of the following form:

There exist constants $\de^{\prime},C^{\prime}>0$ such that for all $\nu$ the following holds. For all $z\in A(\de^{\nu}/\rho,\rho)$ and all $\rho^{\prime}>0$ such that $B_{\rho^{\prime}}(z) \subset A(\de^{\nu}/\rho,\rho)$,
\begin{equation} \label{InequalityMeanValueInequalityFirstApplicationOfAnnulusLemma}
	E_{J_{A_{\nu}}}\bigl( w_{\nu};B_{\rho^{\prime}}(z) \bigr) < \de^{\prime} \quad \Longrightarrow \quad \frac{1}{2} \Abs{\dop\!w_{\nu}(z)}^{2}_{J_{A_{\nu}}} \le \frac{C^{\prime}}{{\rho^{\prime}}^{2}} \cdot E_{J_{A_{\nu}}}\bigl( w_{\nu};B_{\rho^{\prime}}(z) \bigr) + C^{\prime}.
\end{equation}
This uniform mean value inequality, however, follows from the mean value inequality contained in assumption (d) of Proposition~\ref{PropositionBubblesConnect} by conformal invariance of the energy, by assumption (a) of Proposition~\ref{PropositionBubblesConnect}, and since the sequence of conformal maps $\rho_{0}^{\nu}$ appearing in the definition of $w_{\nu}$ converges uniformly to the identity by Step 2 in the proof of \cite[Prop.\,4.7.2]{McDuff/J-holomorphic-curves-and-symplectic-topology}. Thus Lemma~\ref{LemmaAnnulusLemma} applies to the curves (\ref{MapRescaledCurvesOnAnnulus}).

The modifications to the proof of \cite[Lemma~4.7.4]{McDuff/J-holomorphic-curves-and-symplectic-topology} are as follows. Write
\[
E^{\nu}(\rho) \deq E_{J_{A_{\nu}}}\bigl( w_{\nu}; A(\de^{\nu}/\rho,\rho) \bigr) \quad \text{and} \quad E(\rho) \deq \lim_{\nu \to \infty} E^{\nu}(\rho).
\]
Note that if $2\rho \le r$ then $A(\de^{\nu}/2\rho,2\rho) \subset A(\de^{\nu}/r,r)$. We may now apply Lemma~\ref{LemmaAnnulusLemma} to the curves (\ref{MapRescaledCurvesOnAnnulus}) for some fixed number $0 < \mu < 1$, obtaining constants $\de_{0}^{\prime} \deq \de_{0}^{\prime}(\de^{\prime},C^{\prime},\mu)>0$ and $c^{\prime} \deq c^{\prime}(C^{\prime},\mu)>0$. Here $\de^{\prime}$ and $C^{\prime}$ are the constants from the uniform mean value inequality~(\ref{InequalityMeanValueInequalityFirstApplicationOfAnnulusLemma}). Then by inequality~(\ref{InequalitySecondEstimateAnnulusLemma}) with $T=\log2$ we further obtain
\[
E^{\nu}(2\rho) < \de_{0}^{\prime} \quad \Longrightarrow \quad \sup_{z_{1},z_{2}\in A(\de^{\nu}/\rho,\rho)} \dop_{J_{A_{\nu}}}\bigl( w_{\nu}(z_{1}),w_{\nu}(z_{2}) \bigr) \le c^{\prime} \cdot \left( \sqrt{E^{\nu}(2\rho)} + 2\rho \right).
\]
Taking the limit $\nu\to\infty$ we therefore get
\[
E(2\rho) < \de_{0}^{\prime} \quad \Longrightarrow \quad \dop_{J_{A}}\bigl( \utilde(\rho),v(1/\rho) \bigr) = \lim_{\nu\to\infty} \dop_{J_{A_{\nu}}}\bigl( w_{\nu}(\rho),w_{\nu}(\de^{\nu}/\rho) \bigr) \le c^{\prime} \cdot \left( \sqrt{E(2\rho)} + 2\rho \right).
\]
Letting $\rho\to 0$ we obtain $\utilde(0) = v(\infty)$. The remaining parts of the proof of \cite[Lemma~4.7.4]{McDuff/J-holomorphic-curves-and-symplectic-topology} carry over to our situation without modification.

\smallskip

Lastly, we discuss the proof of \cite[Prop.\,4.7.1]{McDuff/J-holomorphic-curves-and-symplectic-topology}. In Step 3 of this proof we need to apply Lemma~\ref{LemmaAnnulusLemma} to the $J_{A_{\nu}}$-holomorphic curves
\begin{eqnarray} \label{MapCurvesOnAnnulus}
	\map{\utilde_{\nu}}{A(\de^{\nu},\veps^{\nu})}{P(\X)},
\end{eqnarray}
where $\de^{\nu}$ and $\veps^{\nu}$ both converge to zero such that $\de^{\nu}/\veps^{\nu} \to 0$. Again, before explaining the required modifications to this proof we verify that the assumptions of Lemma~\ref{LemmaAnnulusLemma} are satisfied.

Since $\veps^{\nu} \to 0$ we may without loss of generality assume that $\veps^{\nu} < R_{0}$ for some constant~$R_{0}>0$. Moreover, since $\de^{\nu}/\veps^{\nu} \to 0$ the ratio $\veps^{\nu} / \de^{\nu}$ gets arbitrarily large. It remains to show that the curves $\utilde_{\nu}$ satisfy a uniform mean value inequality of the following form:

There exist constants $\de^{\prime},C^{\prime}>0$ such that for all $\nu$ the following holds. For all $z\in A(\de^{\nu},\veps^{\nu})$ and all $\rho^{\prime}>0$ such that $B_{\rho^{\prime}}(z) \subset A(\de^{\nu},\veps^{\nu})$,
\begin{equation} \label{InequalityMeanValueInequalitySecondApplicationOfAnnulusLemma}
	E_{J_{A_{\nu}}}\bigl( \utilde_{\nu};B_{\rho^{\prime}}(z) \bigr) < \de^{\prime} \quad \Longrightarrow \quad \frac{1}{2} \Abs{\dop\!\utilde_{\nu}(z)}^{2}_{J_{A_{\nu}}} \le \frac{C^{\prime}}{{\rho^{\prime}}^{2}} \cdot E_{J_{A_{\nu}}}\bigl( \utilde_{\nu};B_{\rho^{\prime}}(z) \bigr) + C^{\prime}.
\end{equation}
As above, this uniform mean value inequality follows from the mean value inequality contained in assumption (d) of Proposition~\ref{PropositionBubblesConnect} by conformal invariance of the energy and by assumption~(a) of Proposition~\ref{PropositionBubblesConnect}. Thus Lemma~\ref{LemmaAnnulusLemma} applies to the curves (\ref{MapCurvesOnAnnulus}).

The modifications to Step 3 in the proof of \cite[Prop.\,4.7.1]{McDuff/J-holomorphic-curves-and-symplectic-topology} are as follows. We apply Lemma~\ref{LemmaAnnulusLemma} to the curves (\ref{MapCurvesOnAnnulus}) for $\mu \deq 1/2$, obtaining constants $\de_{0}^{\prime} \deq \de_{0}^{\prime}(\de^{\prime},C^{\prime},\mu)>0$ and $c^{\prime} \deq c^{\prime}(C^{\prime},\mu)>0$. Here $\de^{\prime}$ and $C^{\prime}$ are the constants from the uniform mean value inequality (\ref{InequalityMeanValueInequalitySecondApplicationOfAnnulusLemma}). By inequality~(\ref{InequalityFirstEstimateAnnulusLemma}) we finally get
\[
E_{J_{A_{\nu}}}\bigl( \utilde_{\nu}; A\bigl( e^{T}\de^{\nu},e^{-T}\veps^{\nu} \bigr) \bigr) \le c^{\prime} \cdot e^{-T} \cdot E_{J_{A_{\nu}}}\bigl( \utilde_{\nu}; A(\de^{\nu},\veps^{\nu}) \bigr).
\]
The remaining parts of the proof of \cite[Prop.\,4.7.1]{McDuff/J-holomorphic-curves-and-symplectic-topology} then carry over to our situation without modification.

\smallskip

This finishes the proof of Proposition~\ref{PropositionBubblesConnect}.

\subsection{Proof of Gromov compactness}
\label{SubSectionProofOfGromovCompactness}

We are now ready to prove Theorem~\ref{TheoremGromovCompactness}. Our strategy is to adapt the proof of \cite[Thm.\,5.3.1]{McDuff/J-holomorphic-curves-and-symplectic-topology} on Gromov compactness for pseudoholomorphic curves, replacing the statements of \cite[Thm.\,4.6.1, Prop.\,4.7.1 and Prop.\,4.7.2]{McDuff/J-holomorphic-curves-and-symplectic-topology} with the corresponding statements of Theorem~\ref{TheoremConvergenceModuloBubbling} and Proposition~\ref{PropositionBubblesConnect}.

\smallskip

Fix a nonnegative integer $n$, a $G$-invariant $\om$-compatible almost complex structure~$J$ on $\X$, and a complex structure $j_{\Sig}$ and an area form $\dvol_{\Sig}$ on~$\Sig$. Consider a sequence $(A_{\nu},u_{\nu},\textbf{z}_{\nu})$ of $n$-marked vortices whose Yang-Mills-Higgs energy satisfies a uniform bound
\[
\sup_{\nu} E(A_{\nu},u_{\nu}) < \infty.
\]
Our goal is to construct a rooted $n$-labeled tree $T=(V=\{\rv\} \sqcup V_{S},E,\Lam)$ and a polystable vortex
\[
\big( \sv \big) = \big( (A,u_{\rv}),\{u_{\al}\}_{\al\,\in\,V_{S}},\{z_{\al\be}\}_{\edge{\al}{\be}},\{\al_{i},z_{i}\}_{1 \le i \le n} \big)
\]
of combinatorial type $T$ such that the sequence $(A_{\nu},u_{\nu},\textbf{z}_{\nu})$ Gromov converges to $(\sv)$ in the sense of Definition~\ref{DefinitionGromovConvergence}.

\smallskip

We shall proceed in eight steps.

\medskip

\schritt{1} We fix a root vertex $0$ and assign to it the principal component $\pc \deq \Sig$. In this way,~$\pc$ inherits a fixed complex structure $j_{\pc} \deq j_{\Sig}$ and a fixed area form $\dvol_{\pc} \deq \dvol_{\Sig}$, with corresponding K\"ahler metric $\l\cdot,\cdot\r_{\pc} \deq \dvol_{\pc}(\,\cdot, j_{\pc}\,\cdot)$.

\medskip

\schritt{2} We apply Theorem~\ref{TheoremConvergenceModuloBubbling} to the sequence of vortices $(A_{\nu},u_{\nu})$. The conclusion is that there exists a smooth vortex $(A,u_{\rv})$, a sequence of smooth gauge transformations $g_{\nu} \in \G(P)$, a real number $p>2$, and a finite set $Z_{\rv} = \{\zeta_{1},\ldots,\zeta_{N}\}$ of distinct points on $\pc$ such that, after passing to a subsequence,
\begin{enumerate}[itemsep=0.3ex, itemindent=0cm, labelsep=1ex, leftmargin=1.1cm]
	\item the sequence $g^{\ast}_{\nu} A_{\nu}$ converges to $A$ weakly in $W^{1,p}$ and strongly in $C^{0}$ on $\pc$;
	\item the sequence $g^{-1}_{\nu} u_{\nu}$ converges to $u_{\rv}$ in $C^{\infty}$ on compact subsets of $\pc \setminus Z_{\rv}$;
	\item for every $j\in\{1,\ldots,N\}$ and every $\veps>0$ such that $B_{\veps}(\zeta_{j}) \cap Z_{\rv} = \{\zeta_{j}\}$ the limit
\[
m_{\veps}(\zeta_{j}) \deq \lim_{\nu\to\infty} E\bigl( g_{\nu}^{\ast} A_{\nu},g_{\nu}^{-1} u_{\nu};B_{\veps}(\zeta_{j}) \bigr)
\]
exists and is a continuous function of $\veps$, and
\[
m(\zeta_{j}) \deq \lim_{\veps\to 0}m_{\veps}(\zeta_{j}) \ge \hbar,
\]
where $\hbar$ is the constant of Corollary \ref{CorollaryAPrioriEstimateOnSurface};
	\item for every compact subset $K \subset \pc$ such that $Z_{\rv}$ is contained in the interior of $K$,
\[
E\bigl( A,u_{\rv};K \bigr) + \sum_{j=1}^{N} m(\zeta_{j}) = \lim_{\nu\to\infty} E\bigl( g_{\nu}^{\ast} A_{\nu},g_{\nu}^{-1} u_{\nu};K \bigr).
\]
\end{enumerate}

\medskip

\schritt{3} As described in Section \ref{SubSectionVorticesVersusPseudoholomorphicMaps}, the connection $A$ gives rise to a symplectic form~$\om_{A}$ and an almost complex structure $J_{A}$ on the total space of the bundle $P(\X) = P \xop_{G} \X$ over $\pc$, defined by formulas (\ref{EquationSymplecticFormOnP(X)}) and (\ref{EquationAlmostComplexStructureOnP(X)}), respectively. By Proposition~\ref{PropositionVorticesAsPseudoholomorphicMaps}\,(i), taking $A$ as reference connection, the almost complex structure $J_{A}$ is tamed by $\om_{A}$. In particular, we have a Riemannian metric
\[
\l\cdot\,,\cdot\r_{J_{A}} \deq \frac{1}{2} \bigl( \om_{A}(\cdot\,,J_{A}\,\cdot) - \om_{A}(J_{A}\,\cdot,\,\cdot) \bigr)
\]
on $P(\X)$ determined by $\om_{A}$ and $J_{A}$. For later use, we recall from \cite[Sec.\,2.2]{McDuff/J-holomorphic-curves-and-symplectic-topology} that the energy of any $J_{A}$-holomorphic section $\map{u}{\pc}{P(\X)}$ is then given by
\[
E_{J_{A}}(u) \deq \frac{1}{2} \int_{\pc} \Abs{\dop\!u}_{J_{A}}^{2} \, \dvol_{\pc},
\]
where the norm $\Abs{\dop\!u}_{J_{A}}$ is understood with respect to the metric $\l\cdot\,,\cdot\r_{J_{A}}$ on the bundle $P(\X)$ and the metric~$\l\cdot\,,\cdot\r_{\pc}$ on $\pc$.

As in Remark \ref{RemarkMapVersusSection}, we regard the map $\map{u_{\rv}}{P}{\X}$ from Step 2 as a section $\map{u_{\rv}}{\pc}{P(\X)}$. Since $(A,u_{\rv})$ is a vortex, it follows from Lemma \ref{LemmaVorticesAsPseudoholomorphicMaps} that this section~$u_{\rv}$ is $J_{A}$-holomorphic.

\medskip

\schritt{4} By assertion (i) in Step 2, after passing to a subsequence, Proposition~\ref{PropositionVorticesAsPseudoholomorphicMaps} applies to the sequence of vortices $(g_{\nu}^{\ast}A_{\nu},g_{\nu}^{-1}u_{\nu})$.

More precisely, the connections $g_{\nu}^{\ast}A_{\nu}$ give rise to almost complex structures $J_{\nu} \deq J_{g_{\nu}^{\ast}A_{\nu}}$ on the bundle $P(\X)$, defined by formula (\ref{EquationAlmostComplexStructureOnP(X)}), which by Proposition~\ref{PropositionVorticesAsPseudoholomorphicMaps}\,(i) are all tamed by the symplectic form $\om_{A}$. Moreover, we see from formula (\ref{EquationAlmostComplexStructureOnP(X)}) that $J_{\nu}$ converges to~$J_{A}$ in $C^{0}$. As before, we have Riemannian metrics
\[
\l\cdot\,,\cdot\r_{J_{A_{\nu}}} \deq \frac{1}{2} \bigl( \om_{A}(\cdot\,,J_{\nu}\,\cdot) - \om_{A}(J_{\nu}\,\cdot,\,\cdot) \bigr)
\]
on $P(\X)$ determined by $\om_{A}$ and $J_{\nu}$, which converge to the metric $\l\cdot\,,\cdot\r_{J_{A}}$ in~$C^{0}$. The energy of any $J_{\nu}$-holomorphic section $\map{u}{\pc}{P(\X)}$ is then given by
\[
E_{J_{\nu}}(u) \deq \frac{1}{2} \int_{\pc} \Abs{\dop\!u}_{J_{\nu}}^{2} \, \dvol_{\pc},
\]
where the norm $\Abs{\dop\!u}_{J_{\nu}}$ is understood with respect to the metric $\l\cdot\,,\cdot\r_{J_{\nu}}$ on~$P(\X)$ and the metric~$\l\cdot\,,\cdot\r_{\pc}$ on $\pc$.

As in Remark \ref{RemarkMapVersusSection}, we regard the maps $\map{g_{\nu}^{-1}u_{\nu}}{P}{\X}$ as sections of~$P(\X)$. By Lemma \ref{LemmaVorticesAsPseudoholomorphicMaps} it follows that
\begin{enumerate}[topsep=1ex, itemsep=0.3ex, itemindent=0cm, labelsep=1ex, leftmargin=1.1cm]
	\item[(v)] the sections $\map{g_{\nu}^{-1}u_{\nu}}{\pc}{P(\X)}$ are $J_{\nu}$-holomorphic.
\end{enumerate}
Furthermore, Proposition~\ref{PropositionVorticesAsPseudoholomorphicMaps}\,(ii--iii) implies that
\begin{enumerate}[topsep=1ex, itemsep=0.3ex, itemindent=0cm, labelsep=1ex, leftmargin=1.1cm]
	\item[(vi)] the energy of the sections $g_{\nu}^{-1}u_{\nu}$ satisfies a uniform bound
\[
\sup_{\nu} E_{J_{\nu}}\bigl( g_{\nu}^{-1}u_{\nu} \bigr) < \infty;
\]
	\item[(vii)] there exist constants $r_{0} > 0$ and $\de,C>0$ such that for every~$\nu$ the section $g_{\nu}^{-1}u_{\nu}$ satisfies a mean value inequality of the following form: For all $z_{0} \in \pc$ and all $0 < r < r_{0}$,
\[
E_{J_{\nu}} \bigl( g_{\nu}^{-1}u_{\nu};B_{r}(z_{0}) \bigr) < \de \quad \Longrightarrow \quad \Abs{\dop\bigl( g_{\nu}^{-1}u_{\nu} \bigr)(z_{0})}_{J_{\nu}}^{2} \le \frac{C}{r^{2}} \cdot E_{J_{\nu}}\bigl( g_{\nu}^{-1}u_{\nu};B_{r}(z_{0}) \bigr) + C.
\]
\end{enumerate}
Note that in order to obtain (vi) we use that $\sup_{\nu} E(A_{\nu},u_{\nu}) < \infty$ by assumption, in combination with gauge invariance of the Yang-Mills-Higgs energy.

\medskip

\schritt{5} We rephrase assertions (iii) and (iv) in Step 2 in terms of the energy of the sections~$g_{\nu}^{-1}u_{\nu}$. More precisely, we claim that
\begin{enumerate}[topsep=1ex, itemsep=0.3ex, itemindent=0cm, labelsep=1ex, leftmargin=1.1cm]
	\item[(iii')] for every $j\in\{1,\ldots,N\}$ and every $\veps>0$ such that $B_{\veps}(\zeta_{j}) \cap Z_{\rv} = \{\zeta_{j}\}$ the limit
\[
m_{\veps}^{\prime}(\zeta_{j}) \deq \lim_{\nu\to\infty} E_{J_{\nu}}\bigl( g_{\nu}^{-1} u_{\nu};B_{\veps}(\zeta_{j}) \bigr)
\]
exists and is a continuous function of $\veps$, and
\[
m^{\prime}(\zeta_{j}) \deq \lim_{\veps\to 0} m_{\veps}^{\prime}(\zeta_{j}) \ge \hbar;
\]
	\item[(iv')] for every compact subset $K \subset \pc$ such that $Z_{\rv}$ is contained in the interior of $K$,
\[
E\bigl( A,u_{\rv};K \bigr) + \sum_{j=1}^{N} m^{\prime}(\zeta_{j}) = \lim_{\nu\to\infty} E\bigl( g_{\nu}^{\ast} A_{\nu},g_{\nu}^{-1} u_{\nu};K \bigr).
\]
\end{enumerate}
To see this, we first recall from (vi) in Step 4 that
\[
\sup_{\nu} E_{J_{\nu}}\bigl( g_{\nu}^{-1} u_{\nu};B_{\veps}(\zeta_{j}) \bigr) < \infty.
\]
As in the proof of \cite[Thm.\,4.6.1]{McDuff/J-holomorphic-curves-and-symplectic-topology} we conclude that, after passing to a subsequence, the limit~$m_{\veps}^{\prime}(z_{j})$ exists; moreover, it is a continuous function of $\veps$ for every~$j$ by (ii) in Step 2. Assertions (iii') and (iv') then follow from (iii) and (iv) in Step~2 once we show that
\begin{eqnarray} \label{EquationClaimStep5}
	m^{\prime}(\zeta_{j}) = m(\zeta_{j})
\end{eqnarray}
for every $j$. To prove this, we abbreviate $\Ahat_{\nu} \deq g_{\nu}^{\ast}A_{\nu}$ and $\uhat_{\nu} \deq g_{\nu}^{-1}u_{\nu}$. Recall from Step 4 that the energy of the section $\map{\uhat_{\nu}}{\pc}{P(\X)}$ on $B_{\veps}(\zeta_{j})$ is given by
\[
E_{J_{\nu}}\bigl( \uhat_{\nu};B_{\veps}(\zeta_{j}) \bigr) = \frac{1}{2} \int_{B_{\veps}(\zeta_{j})} \Abs{\dop\!\uhat_{\nu}}^{2}_{J_{\nu}} \dvol_{\pc},
\]
and recall from formula (\ref{EquationYMHEnergyGlobal}) that the Yang-Mills-Higgs energy of the vortex $(\Ahat_{\nu},\uhat_{\nu})$ on $B_{\veps}(\zeta_{j})$ is given by
\[
E\bigl( \Ahat_{\nu},\uhat_{\nu};B_{\veps}(\zeta_{j}) \bigr) = \frac{1}{2} \int_{B_{\veps}(\zeta_{j})} \Abs{\dop_{\Ahat_{\nu}}\!\uhat_{\nu}}^{2}_{J} \dvol_{\pc} + \int_{B_{\veps}(\zeta_{j})} \Abs{\mu(\uhat_{\nu})}^{2} \dvol_{\pc}.
\]
We may hence estimate
\begin{multline} \label{InequalityEstimateStep5}
	\Abs{ m_{\veps}^{\prime}(\zeta_{j}) - m_{\veps}(\zeta_{j}) } \le \frac{1}{2} \int_{B_{\veps}(\zeta_{j})} \lim_{\nu\to\infty} \left( \Abs{ \Abs{\dop\!\uhat_{\nu}}^{2}_{J_{\nu}} - \Abs{\dop_{\Ahat_{\nu}}\!\uhat_{\nu}}^{2}_{J} } \right) \dvol_{\pc} \\ + \norm{\mu}_{C^{0}(M)}^{2} \cdot \volume(B_{\veps}(\zeta_{j})).
\end{multline}
A computation as in the proof of part (i) of Lemma~\ref{LemmaTamingOnE} then yields
\begin{multline*}
	\Abs{\dop\!\uhat_{\nu}(v)}^{2}_{J_{\nu}} - \Abs{\dop_{\Ahat_{\nu}}\!\uhat_{\nu}(v)}_{J}^{2} = \om\bigl( X_{(\Ahat_{\nu}-A)(v)},J \dop_{\Ahat_{\nu}}\!\uhat_{\nu}(v) \bigr) \\ +\, \om\bigl( \dop_{\Ahat_{\nu}}\!\uhat_{\nu}(v),X_{(\Ahat_{\nu}-A)(j_{\pc} v)} \bigr) + \om\bigl( X_{(\Ahat_{\nu}-A)(v)},X_{(\Ahat_{\nu}-A)(j_{\pc} v)} \bigr) \\ -\, \bigl\l F_{A}(v,j_{\pc} v),\mu \bigr\r + (1 + c_{A,\mu}) \cdot \Abs{v}^{2}
\end{multline*}
for all $v\in T\pc$. Recall that
\[
\Abs{\l F_{A}(v,j_{\pc} v),\mu \r_{\gfr}} \le c_{A,\mu} \cdot \Abs{v}^{2}
\]
by definition of $c_{A,\mu}$ in (\ref{EstimateForTwoFormOnP(X)}), and that $\Ahat_{\nu}$ converges to $A$ in $C^{0}$ on $\pc$ by (i) in Step 2. It thus follows that
\[
\lim_{\nu\to\infty} \left( \Abs{ \Abs{\dop\!\uhat_{\nu}}^{2}_{J_{\nu}} - \Abs{\dop_{\Ahat_{\nu}}\!\uhat_{\nu}}^{2}_{J} } \right) \dvol_{\pc} \le c \, (1 + 2\,c_{A,\mu})
\]
for some constant $c>0$. Plugging this into inequality (\ref{InequalityEstimateStep5}) above we get
\[
\Abs{ m_{\veps}^{\prime}(\zeta_{j}) - m_{\veps}(\zeta_{j}) } \le \left( \frac{c}{2}\,(1+2\,c_{A,\mu}) + \norm{\mu}_{C^{0}(M)}^{2} \right) \cdot \volume(B_{\veps}(\zeta_{j})).
\]
Letting $\veps\to 0$, (\ref{EquationClaimStep5}) follows. This completes the proof of (iii') and (iv').

\medskip

\schritt{6} We prove that the sequence of $n$-marked $J_{\nu}$-holomorphic sections $(g_{\nu}^{-1}u_{\nu},\textbf{z}_{\nu})$ Gromov converges to a stable map
\[
(\textbf{u},\textbf{z}) = \bigl( u_{\rv},\{u_{\al}\}_{\al\,\in\,V_{S}},\{z_{\al\be}\}_{\edge{\al}{\be}}, \{ \al_{i},z_{i} \}_{1 \le i \le n} \bigr)
\]
in $P(\X)$ of combinatorial type $T=(V=\{\rv\} \sqcup V_{S},E,\Lam)$ in the sense of \cite[Def.\,5.2.1]{McDuff/J-holomorphic-curves-and-symplectic-topology}, where~$T$ is an $n$-labeled tree, with the modifications that
\begin{itemize}[itemsep=0.3ex, itemindent=0cm, labelsep=1ex, leftmargin=0.7cm]
	\item $\pc$ is of arbitrary genus but does not admit any automorphisms other than the identity;
	\item the (Energy) axiom will be formulated in a different way in Step 8 below.
\end{itemize}
The proof is basically the same as that of \cite[Thm.\,5.3.1]{McDuff/J-holomorphic-curves-and-symplectic-topology}, which is in turn based on \cite[Thm.\,4.6.1, Prop.\,4.7.1 and Prop.\,4.7.2]{McDuff/J-holomorphic-curves-and-symplectic-topology}. Except for certain alterations to be discussed below, the arguments from the proof of \cite[Thm.\,5.3.1]{McDuff/J-holomorphic-curves-and-symplectic-topology} will hence carry over to our situation if we replace the assertions of \cite[Thm.\,4.6.1]{McDuff/J-holomorphic-curves-and-symplectic-topology} with the corresponding assertions (ii), (iii') and (iv') in Step 2 and Step 5 above, and the assertions of \cite[Prop.\,4.7.1 and Prop.\,4.7.2]{McDuff/J-holomorphic-curves-and-symplectic-topology} with the respective assertions of Proposition~\ref{PropositionBubblesConnect}.

More precisely, we do not admit non-trivial automorphisms of the principal component $\pc$; hence $\pc$ will be a distinguished component of the stable map~$\textbf{u}$. We therefore need to apply an induction argument as in the proof of \cite[Thm.\,5.3.1]{McDuff/J-holomorphic-curves-and-symplectic-topology} in order to construct a separate bubble tree at each point $\zeta_{j} \in Z_{0}$, where $Z_{0}$ is the set of singular points obtained in Step 2. Technically, this is achieved by modifying the base step in the induction in the proof of \cite[Thm.\,5.3.1]{McDuff/J-holomorphic-curves-and-symplectic-topology} in the following way. We define the set $Z_{1}$ in that proof to be the set $Z_{\rv}$. Then there exists $r>0$ sufficiently small such that for every $j$ there exists a holomorphic chart
\[
\map{\vphi_{j}}{B}{B_{r}(\zeta_{j})},
\]
where $B \subset \C$ denotes the closed unit disk, such that
\[
\vphi_{j}(0)=\zeta_{j} \quad \text{and} \quad B_{r}(\zeta_{j}) \cap Z_{1} = \{\zeta_{j}\}.
\]
Moreover, we define the sequence of M\"obius transformations $\phi^{\nu}_{1}$ in the proof of \cite[Thm.\,5.3.1]{McDuff/J-holomorphic-curves-and-symplectic-topology} to be trivial, that is, $\phi^{\nu}_{1} \deq \id_{\C}$ for all $\nu$. Then it follows from (ii) in Step 2 that the sequence
\begin{eqnarray} \label{MapLocalSequenceGromovConvergence}
	\map{\bigl( g_{\nu}^{-1}u_{\nu} \bigr) \circ \vphi_{j} = \bigl( g_{\nu}^{-1}u_{\nu} \bigr) \circ \vphi_{j} \circ \phi^{\nu}_{1}}{B}{P(\X)}
\end{eqnarray}
converges to
\[
\map{u_{\rv} \circ \vphi_{j}}{B}{P(\X)}
\]
in $C^{\infty}$ on compact subsets of the punctured disk $B \setminus \{0\}$ for every~$j$. We may then apply Proposition~\ref{PropositionBubblesConnect} to the sequence (\ref{MapLocalSequenceGromovConvergence}).

We need to check that assumptions (a--d) of Proposition~\ref{PropositionBubblesConnect} are satisfied. In fact, (a) is satisfied by construction of $\phi^{\nu}_{1}$, (b) is satisfied by (ii) in Step~2, (c) is satisfied by (iii') in Step~5, and (d) is satisfied by (vii) in Step 4.

The induction then carries on as in the proof of \cite[Thm.\,5.3.1]{McDuff/J-holomorphic-curves-and-symplectic-topology}. Note the following: By construction, the rescalings $\phi_{j}^{\nu}$ from that proof all satisfy assumption (a) of Proposition~\ref{PropositionBubblesConnect}. Moreover, part (ii) of Proposition~\ref{PropositionBubblesConnect} only asserts $C^{1}$-convergence for the rescaled maps $v_{\nu}$. Hence we only get $C^{1}$-convergence for the rescaled maps $u_{\al}^{\nu}$ appearing in the (Map) axiom in \cite[Def.\,5.2.1]{McDuff/J-holomorphic-curves-and-symplectic-topology}.

\medskip

\schritt{7} We claim that the tuple
\[
\big( \sv \big) = \big( (A,u_{\rv}),\{u_{\al}\}_{\al\,\in\,V_{S}},\{z_{\al\be}\}_{\edge{\al}{\be}},\{\al_{i},z_{i}\}_{1 \le i \le n} \big)
\]
consisting of the vortex $(A,u_{\rv})$ obtained in Step 2 and the stable map $(\textbf{u},\textbf{z})$ obtained in Step~6, is a polystable vortex in the sense of Definition \ref{DefinitionPolystableVortex}. In fact, this follows from \cite[Def.\,5.1.1]{McDuff/J-holomorphic-curves-and-symplectic-topology} since by Proposition~\ref{PropositionBubblesConnect} the bubbles $\map{u_{\al}}{\proj{1}}{P(\X)}$, $\al \in V_{S}$, all map into the fiber $P(\X)_{z_{\rv\al}}$ of the bundle $P(\X)$ over the nodal point $z_{\rv\al} \in \pc$. At this point, recall from Section~\ref{SectionIntroductionAndMainResults} that $z_{\rv\al}$ denotes the nodal point on the principal component $\pc$ at which the bubble tree containing the spherical component $\bub$ is attached.

\medskip

\schritt{8} Combining (i) and (ii) in Step 2 with Step 6 above, we see that the sequence of marked vortices $(\vnu)$ Gromov converges against the polystable vortex $(\sv)$ in the sense of Definition~\ref{DefinitionGromovConvergence}, except for the (Energy) axiom.

It remains to check that the (Energy) axiom in Definition~\ref{DefinitionGromovConvergence} is satisfied. In fact, since the sequence $(g_{\nu}^{-1}u_{\nu},\textbf{z}_{\nu})$ Gromov converges against the stable map $(\textbf{u},\textbf{z})$ by Step 6, it follows from the (Energy) axiom in \cite[Def.\,5.2.1]{McDuff/J-holomorphic-curves-and-symplectic-topology} that
\[
\sum_{\ga\,\in\,T_{\rv\al}} E_{J_{A}}(u_{\ga}) = \lim_{\veps\to 0} \lim_{\nu\to\infty} E_{J_{\nu}}\bigl( g_{\nu}^{-1}u_{\nu};B_{\veps}(z_{\rv\al}) \bigr)
\]
for every $\al\in V_{S}$ such that $\edge{\rv}{\al}$. Here $T_{\rv\al}$ denotes the subtree of $T$ containing $\al$ after removing the edge connecting $\rv$ and $\al$ as in \cite[Sec.\,5.1]{McDuff/J-holomorphic-curves-and-symplectic-topology}. Note that we used $u^{\nu}_{\rv} = (g_{\nu}^{-1}u_{\nu}) \circ \phi^{\nu}_{\rv} = g_{\nu}^{-1}u_{\nu}$ by Step 6. On the other hand, taking $K = \pc$ it follows from (iv') in Step 5 that
\[
E(A,u_{\rv}) + \sum_{\al\in V_{S},\edge{\rv}{\al}} \lim_{\veps\to 0} \lim_{\nu\to\infty} E_{J_{\nu}}\bigl( g_{\nu}^{-1}u_{\nu};B_{\veps}(z_{\rv\al}) \bigr) = \lim_{\nu\to\infty} E\bigl( g_{\nu}^{\ast} A_{\nu},g_{\nu}^{-1} u_{\nu} \bigr).
\]
Hence by gauge invariance of the Yang-Mills-Higgs energy we get
\[
\lim_{\nu\to\infty} E(A_{\nu},u_{\nu}) = E(A,u_{\rv}) + \sum_{\al\in V_{S}} E_{J_{A}}(u_{\al}).
\]
Now we see from the definition of the almost complex structure $J_{A}$ in formula~(\ref{EquationAlmostComplexStructureOnP(X)}) that $J_{A}$ and~$J$ agree on the fibers of $P(\X)$. Since $u_{\al}$ maps into the fiber $P(\X)_{z_{\rv\al}}$ over the point $z_{\rv\al} \in \pc$ by Step 7, it follows that $E_{J_{A}}(u_{\al}) = E_{J}(u_{\al})$. We conclude that
\[
\lim_{\nu\to\infty} E(A_{\nu},u_{\nu}) = E(A,u_{\rv}) + \sum_{\al\,\in\,V_{S}} E_{J}(u_{\al}) = E(A,\textbf{u}).
\]
Hence the (Energy) axiom is satisfied.

\bigskip

The proof of Theorem~\ref{TheoremGromovCompactness} is now complete.

\bibliographystyle{amsplain}

\end{document}